\documentclass[11pt,a4paper]{article}
\usepackage[margin=2.4cm]{geometry}
\usepackage{amsmath,amssymb,amsthm}
\usepackage{enumitem}
\newtheorem{theorem}{Theorem}[section]
\newtheorem{lemma}[theorem]{Lemma}
\newtheorem{proposition}[theorem]{Proposition}
\newtheorem{corollary}[theorem]{Corollary}
\theoremstyle{definition}
\newtheorem{remark}[theorem]{Remark}
\newtheorem{example}[theorem]{Example}

\newcommand{\h}{\mathfrak{h}}

\title{A uniform decomposition theorem for invariant differential operators on imprimitive complex reflection groups $G(r,p,n)$}
\author{Jean Kabor\'e$^1$, Ibrahim Nonkan\'e$^2(^*)$\\[4pt]
\small $^1$ Laboratoire de Sciences et Technologies (LaST),
Universit\'e Thomas Sankara, Burkina Faso\\
\small $^2$ D\'epartement d'\'economie et de math\'ematiques appliqu\'ees, IUFIC,\\
Universit\'e Thomas Sankara, Burkina Faso\\[4pt]
\small \texttt{kaborejean775075@gmail.com}$^1$, \texttt{ibrahim.nonkane@uts.bf}$^2$}
\date{}

\begin{document}
\maketitle

\begingroup\renewcommand{\thefootnote}{}\footnotetext{$(^*)$ Corresponding author: Ibrahim Nonkan\'e}\endgroup

\begin{abstract}
Let $W$ be one of the imprimitive complex reflection groups $G(r,p,n)$. We study the
polynomial ring, localized away from the reflecting hyperplanes, as a module over the
ring of $W$-invariant differential operators, that is, the differential operators on
the localized invariant ring. The module is semisimple. Up to isomorphism, it has one simple summand
for each irreducible representation $V$ of $W$, and that summand occurs $\dim V$
times. We give generators of the summands in terms of the higher Specht polynomials
of Ariki--Terasoma--Yamada and Morita--Yamada. The proof rests on Galois descent for
the finite \'etale cover defined by the group. We invert the discriminant
$\delta=\prod_H\alpha_H^{e_H}$, where $\alpha_H$ is a linear form defining a
reflecting hyperplane $H$ and $e_H$ is the order of the pointwise stabilizer of $H$. Inverting
the square of the Jacobian of the basic invariants, as an earlier treatment of
$G(r,1,n)$ did, does not give an invariant localization when $r/p\ge3$. For $p>1$ we correct the
indexing of the decomposition in the theorem of Morita--Yamada. Earlier work covers
$\mathcal S_n$, $W(D_n)$, $W(B_n)$ and $G(r,1,n)$; to our knowledge the groups with
$p>1$ other than $W(D_n)$ are new.
\end{abstract}



\section{Introduction}\label{sec:intro}

Let $W$ be a finite complex reflection group acting on
$\mathcal{O}_X=\mathbb{C}[x_1,\dots,x_n]$. By the Chevalley--Shephard--Todd theorem~\cite{ShephardTodd1954,Chevalley1955}, the invariant ring $\mathcal{O}_X^W$ is a polynomial ring and
$\mathcal{O}_X$ is a free module of rank $|W|$ over it, with fiber the
regular representation of $W$. For the imprimitive complex reflection groups
$G(r,p,n)$ of the Shephard--Todd classification~\cite{ShephardTodd1954}, the higher Specht
polynomials of Ariki, Terasoma and Yamada~\cite{ATY1997}, extended to all
$G(r,p,n)$ by Morita and Yamada~\cite{MoritaYamada1998}, give a basis of
this module adapted to the action of $W$ (a union of bases of irreducible $W$-submodules). A separate line of work
studies $\mathcal{O}_X$ as a module over ${\mathcal{D}}_Y$, the ring of
differential operators on the invariant ring
$\mathcal{O}_Y=\mathcal{O}_X^W$.

Let $H$ run over the reflecting hyperplanes of $W$, let $\alpha_H$ be a linear form defining $H$, and let
$e_H$ be the order of the pointwise stabilizer of $H$. Put $\Delta=\prod_H\alpha_H^{e_H-1}$, the product of the
linear forms of the reflections of $W$ (the Jacobian polynomial of $W$, Subsection~\ref{subsec:notation}), and
$\delta=\prod_H\alpha_H^{e_H}$, the discriminant of $W$. Up to a scalar, $\Delta$ is the Jacobian of the basic
invariants; this is a theorem of Steinberg~\cite{Steinberg1960} for every $W$, and Lemma~\ref{lem:jacobian-general}
gives it directly for $G(r,p,n)$. The discriminant $\delta$ is invariant and has the same zero set as $\Delta$ (it equals
$\Delta^2$ up to a scalar for real groups, and for $G(r,p,n)$ exactly when $r/p\le2$; see
Subsection~\ref{subsec:localization}). Write $\widetilde{\mathcal{O}}_X$ for the
localization of $\mathcal{O}_X$ at $\Delta$, $\widetilde{\mathcal{O}}_Y=\mathcal O_Y[\delta^{-1}]=\widetilde{\mathcal O}_X^{\,W}$
for the localized invariant ring, and $\widetilde{\mathcal{D}}_Y$ for the ring of differential
operators on $\widetilde{\mathcal O}_Y$, which is $\mathcal D_Y$ localized at $\delta$. For $G(r,p,n)$ it is also
the ring of $W$-invariant differential operators on $\widetilde{\mathcal O}_X$
(Proposition~\ref{prop:etale}(d)).

Since $\widetilde{\mathcal O}_X$ is the direct image of the structure sheaf under the finite
\'etale map $\operatorname{Spec}\widetilde{\mathcal O}_X\to\operatorname{Spec}\widetilde{\mathcal O}_Y$, it is a semisimple $\widetilde{\mathcal{D}}_Y$-module by the decomposition
theorem~\cite{deCataldoMigliorini2009}. The question is then to find its simple constituents, their
multiplicities and generators. This question was treated for the symmetric group $W=\mathcal S_n$
in~\cite{Nonkane2019} (see also~\cite{Nonkane2021sym,Nonkane2021IJMMS}), for certain products
of symmetric groups in~\cite{NonkaneTodjihounde2023}, and for the wreath product
$G(r,n)=\mathbb Z/r\mathbb Z\wr\mathcal S_n$ (that is, $G(r,1,n)$ in the notation below)
in~\cite{NonkaneLawson2022}. A different approach to $G(r,1,n)$, by explicit dual coordinates and dual operators that generate a copy of the Weyl algebra inside the localization of the Weyl algebra of $\mathbb C^n$, is taken in~\cite{KKN2026}. The real reflection groups $W(D_n)$ and $W(B_n)$ were treated
in~\cite{NonkaneLawson2021,NonkaneLawson2021B}, motivated by the action of the rational Olshanetsky--Perelomov Hamiltonians on polynomials~\cite{OlshanetskyPerelomov1983}. The present paper treats the whole family
$G(r,p,n)$ in one argument. To our knowledge, the groups with $p>1$ other than
$W(D_n)=G(2,2,n)$, which~\cite{NonkaneLawson2021} treats, were not covered before; they need
the shift combinatorics of~\cite{MoritaYamada1998}.

This paper also corrects two points of  earlier works. For $G(r,p,n)$ with $r/p\ge3$, in particular for $G(r,1,n)$ with $r\ge3$, the square $\Delta^2$ is not
invariant, so the localization used for $G(r,1,n)$ in~\cite{NonkaneLawson2022} is
replaced by the localization at $\delta=\prod_H\alpha_H^{e_H}$
(Lemma~\ref{lem:delta}). And the decomposition
in~\cite[Thm.~3]{MoritaYamada1998} has to be indexed by all $\lambda\in\mathcal P_{r,n}$, the $r$-tuples of Young diagrams of total size $n$
(Lemma~\ref{lem:count}).

\paragraph{Overview of the results.} Write $B=\widetilde{\mathcal O}_X$,
$\mathcal D=\widetilde{\mathcal D}_Y$ and $W=G(r,p,n)$.
Theorem~\ref{thm:structure} gives a direct proof that $B$ is a semisimple $\mathcal D$-module,
and describes it: $B\cong\bigoplus_V M_V^{\oplus\dim V}$, where $V$ runs over the irreducible
representations of $W$ and $M_V=\operatorname{Hom}_W(V,B)\cong(B\otimes V^*)^W$
is a simple $\mathcal D$-module of rank $\dim V$ over $\widetilde{\mathcal O}_Y$. The $M_V$ are pairwise
non-isomorphic, and $\operatorname{End}_{\mathcal D}(B)\cong\mathbb C[W]$.
Theorem~\ref{thm:main-general} and Proposition~\ref{prop:decomposition-general}
give the generators. For suitable primitive idempotents $e$ of
$\mathbb C[W]$, those adapted to the Specht basis, the module $eB$ is generated by a higher Specht polynomial, and $B$ is
the direct sum of these submodules; the index sets and multiplicities are stated there. For an
arbitrary primitive idempotent, $eB$ is simple but need not have a
Specht generator (Remark~\ref{rem:arbitrary-idempotent}).
Theorem~\ref{thm:descent} and Corollary~\ref{cor:characterization} state Galois descent
for rings. The functor $V\mapsto(B\otimes V)^W$ is fully faithful, and its essential image
consists of the $\mathcal D$-modules, finitely generated over $\widetilde{\mathcal O}_Y$, that become
trivial on the cover. Up to isomorphism, the direct summands of $B$
are the modules $(B\otimes U)^W$ with $U$ a subrepresentation of the regular
representation.

The proofs are Galois descent for the finite \'etale cover of
Proposition~\ref{prop:etale}, together with the Jacobian formula of
Lemma~\ref{lem:jacobian-general} and the combinatorics of Section~\ref{sec:prelim}. The Jacobian
computation and Galois descent are classical (Remark~\ref{rem:classical}). The paper
adds the generators given by higher Specht polynomials, with their index sets and multiplicities
for $p>1$ beyond $W(D_n)$, the invariant localization, and the characterization of the direct
summands of $B$. The examples are $D_2$, $D_3$, $B_2$, $G(3,1,1)$,
$G(4,2,2)$ and $G(6,2,2)$. We write the central idempotents for $D_2$, $D_3$ and every $G(r,1,2)$ (so for $B_2$); for
$G(3,1,1)$, $G(4,2,2)$ and $G(6,2,2)$ we give the characters, which determine them by the formula
of Subsection~\ref{subsec:idempotents}. For $B_2$, $G(4,2,2)$ and $G(6,2,2)$ we also give the
adapted primitive idempotents.

\paragraph{Structure of the paper.} Section~\ref{sec:prelim} fixes notation
for reflections and recalls the combinatorics of higher Specht polynomials for
$G(r,p,n)$, with the corrected indexing.
Section~\ref{sec:general} contains the main results: the invariant
localization and the \'etale cover, the Jacobian lemma, the centralizer,
Galois descent and simplicity, and the adapted idempotents and Specht
generators. Section~\ref{sec:Dn} specializes to the real reflection group
$W(D_n)$ (Corollary~\ref{cor:Dn-main}), with the examples $D_2$ and $D_3$.
Section~\ref{sec:other-cases} treats $G(r,r,n)$, $W(B_n)=G(2,1,n)$
(Corollary~\ref{cor:Bn}) and $G(r,1,2)$, with the examples
$G(4,2,2)$ and $G(6,2,2)$, where $p>1$.
Section~\ref{sec:conclusion} collects remarks on the exceptional groups, Cherednik algebras, monodromy and
the link with~\cite{KKN2026}, and lists what is open.

\section{Preliminaries}\label{sec:prelim}

\subsection{Reflections and the Jacobian polynomial}\label{subsec:notation}

A \emph{reflection} (or \emph{complex reflection}) of a finite-dimensional vector space
$\h$ is a semisimple element $s\in GL(\h)$ with $\mathrm{rank}(1-s)=1$. A complex
reflection may have any finite order. When $\h$ is real, a \emph{real reflection} is a
reflection of order $2$; this is the notion used below for $W(D_n)$ and $W(B_n)$. In the
groups $G(r,p,n)$ of Subsection~\ref{subsec:specht}, the non-diagonal reflections have
order $2$ and the diagonal ones have order dividing $d=r/p$.

Let $s \in GL(\h)$ be a complex reflection. Denote by
$\alpha_s \in \h^*$ an eigenvector of $s$ in $\h^*$ with nontrivial
eigenvalue; its kernel is the reflecting hyperplane of $s$. For $W$ a reflection group with set of reflections
$\mathcal R \subset W$, the \emph{Jacobian polynomial} of $W$ is
$\Delta = \prod_{s \in \mathcal R} \alpha_s(x)$ (well defined up to a nonzero
scalar, since each $\alpha_s$ is). Let $H$ be a reflecting hyperplane, $\alpha_H$ a linear form defining it, and $e_H$
the order of its pointwise stabilizer in $W$. This stabilizer is cyclic and contains $e_H-1$ reflections, all with $\alpha_s$
proportional to $\alpha_H$; hence $\Delta=\prod_H\alpha_H^{e_H-1}$ up to a scalar. Up to a scalar, $\Delta$ is the Jacobian of the basic invariants: this is a theorem of
Steinberg~\cite{Steinberg1960} for every $W$, and Lemma~\ref{lem:jacobian-general} gives it directly for $G(r,p,n)$
(for $W(D_n)=G(2,2,n)$ see Corollary~\ref{cor:Dn-setup}). The \emph{discriminant} of $W$ is the invariant element
$\delta=\prod_H\alpha_H^{e_H}$ of Subsection~\ref{subsec:localization}; it has the same zero set
as $\Delta$. (Some authors, including~\cite{KKN2026}, call $\Delta$ itself the discriminant.) Section~\ref{sec:general} inverts $\Delta$, equivalently $\delta$.

\subsection{Higher Specht polynomials for the reflection group $G(r,p,n)$}\label{subsec:specht}

We recall the higher Specht polynomials of Ariki--Terasoma--Yamada
\cite{ATY1997} and Morita--Yamada \cite{MoritaYamada1998}, which underlie
Section~\ref{sec:general}, and we correct two points in \cite{MoritaYamada1998}. We follow the conventions of
\cite{MoritaYamada1998}: $\xi$ is a primitive $r$-th root of unity,
$d=r/p$, and the components of an $r$-tuple (of diagrams or of tableaux)
are indexed by $0,1,\dots,r-1$. The proof of \cite[Lemma~2]{MoritaYamada1998} uses the count
$|\mathrm{STab}(\lambda)_d|=|\mathrm{STab}(\lambda)|/p$ for a single shape $\lambda$, which holds only after summing over the
$\mathrm{Sh}$-orbit of $\lambda$ (Lemma~\ref{lem:count}). As a consequence, the index set in
\cite[Thm.~3(3)]{MoritaYamada1998}, a system of representatives of $\mathcal P_{r,n}/{\sim}$, has to be replaced by all of
$\mathcal P_{r,n}$, or by the sets $I(\mathcal O)$ (Theorem~\ref{thm:2.3}(ii), Corollary~\ref{cor:mult} and
Remark~\ref{rem:diffMY}). The difference occurs only for $p>1$. The associator and the families $F^l_{S,T}$ are those of
\cite{MoritaYamada1998}.

\subsubsection{The groups and their coinvariant algebras}

The group $G(r,p,n)$, where $r,p,n\ge1$ and $p\mid r$, consists of the
monomial matrices whose nonzero entries are of the form $\xi^j$
($0\le j<r$) and such that the $d$-th power of the product of all
nonzero entries is equal to $1$. It is a normal subgroup of index $p$ of
the wreath product
\[
  G(r,n)=(\mathbb Z/r\mathbb Z)^n\rtimes\mathcal S_n
  =\{(\xi^{i_1},\dots,\xi^{i_n};\sigma)\mid 0\le i_k<r,\ \sigma\in\mathcal S_n\}
  =G(r,1,n),
\]
with product
\[
  (\xi^{i_1},\dots,\xi^{i_n};\sigma)(\xi^{j_1},\dots,\xi^{j_n};\pi)
  =(\xi^{i_1+j_{\sigma^{-1}(1)}},\dots,\xi^{i_n+j_{\sigma^{-1}(n)}};\sigma\pi),
\]
and $(\xi^{i_1},\dots,\xi^{i_n};\sigma)\in G(r,p,n)$ if and only if
$i_1+\dots+i_n\equiv0\pmod p$. In particular
$|G(r,p,n)|=r^nn!/p$.

Let $\mathcal O_X=\mathbb C[x_1,\dots,x_n]$. The group $G(r,n)$ acts on
the left on $\mathcal O_X$ by
\[
  (\xi^{i_1},\dots,\xi^{i_n};\sigma)f
  =f(\xi^{i_{\sigma(1)}}x_{\sigma(1)},\dots,\xi^{i_{\sigma(n)}}x_{\sigma(n)}) ,\mbox{where}\ f \in \mathcal O_X.
\]
In particular $\pi\in\mathcal S_n$ acts by $(\pi f)(x)=f(x_{\pi(1)},\dots,x_{\pi(n)})$,
and a monomial $x^a$ is sent to $x^b$ with $b_{\pi(k)}=a_k$.
The fundamental invariants of $G(r,n)$ are the elementary symmetric
functions $e_j(x_1^r,\dots,x_n^r)$, $1\le j\le n$ \cite{ATY1997}. Let
$J'_+$ be the ideal they generate and $\Lambda'=\mathcal O_X/J'_+$. The
fundamental invariants of $G(r,p,n)$ are $e_j(x_1^r,\dots,x_n^r)$,
$1\le j\le n-1$, and $e_n(x_1^d,\dots,x_n^d)=(x_1\cdots x_n)^d$ \cite{ShephardTodd1954}; let $J_+$
be the ideal they generate and $\Lambda_0=\mathcal O_X/J_+$. By
Chevalley's theorem \cite{Chevalley1955}, $\mathcal O_X$ is a free module
over the invariants, and the $G(r,n)$-module $\Lambda'$ (resp.\ the
$G(r,p,n)$-module $\Lambda_0$) is isomorphic to the left regular
representation of $G(r,n)$ (resp.\ $G(r,p,n)$).

\subsubsection{Higher Specht polynomials for $G(r,n)$}

Let $\mathcal P_{r,n}$ be the set of $r$-tuples of Young diagrams
$\lambda=(\lambda^0,\dots,\lambda^{r-1})$ with $|\lambda^0|+\dots+|\lambda^{r-1}|=n$.
Filling the cells with the integers $1,\dots,n$, each used once, gives an
$r$-tableau $T=(T^0,\dots,T^{r-1})$ of shape $\lambda$; we write
$k\in T^\nu$ if $k$ occurs in $T^\nu$. The set of $r$-tableaux of shape
$\lambda$ is $\mathrm{Tab}(\lambda)$; $T$ is \emph{standard} if the
numbers increase along each row and each column of every $T^\nu$,
$0\le\nu<r$; the set of standard $r$-tableaux of shape $\lambda$ is
$\mathrm{STab}(\lambda)$, and $f_\lambda:=|\mathrm{STab}(\lambda)|$.
The number $f_\lambda$ depends only on the multiset of components of $\lambda$.

Let $S\in\mathrm{STab}(\lambda)$. The word $w(S)$ is obtained by reading
each column of $S^0$ from bottom to top, starting from the left, then
$S^1$ in the same way, and so on. The \emph{index} is defined inductively
by $i(1)=0$ and, for $1\le k<n$: $i(k+1)=i(k)+1$ if $k+1$ stands to the
left of $k$ \emph{in the word} $w(S)$, and $i(k+1)=i(k)$ otherwise.
Writing $i(k)$ in the cell of $S$ containing $k$ gives a filling
$i(S)=(i(S)^0,\dots,i(S)^{r-1})$ of $\lambda$ by nonnegative integers.

\begin{example}\label{ex:index}
For $r=3$, $n=7$ and $S=\big(\begin{smallmatrix}2&5\\6\end{smallmatrix},\
\begin{smallmatrix}1\\4\end{smallmatrix},\ \begin{smallmatrix}3&7\end{smallmatrix}\big)$
one gets $w(S)=6\,2\,5\,4\,1\,3\,7$ and
$i(S)=\big(\begin{smallmatrix}1&3\\4\end{smallmatrix},\
\begin{smallmatrix}0\\2\end{smallmatrix},\ \begin{smallmatrix}1&4\end{smallmatrix}\big)$
(\cite[Fig.~1]{MoritaYamada1998}).
\end{example}

For an $r$-tableau $T$ of shape $\lambda$ and $0\le\nu<r$, the
\emph{Young symmetrizer} of $T^\nu$ is
\[
  e_{T^\nu}=\frac1{\alpha_{T^\nu}}\sum_{\sigma\in R(T^\nu)}\sum_{\tau\in C(T^\nu)}
  \mathrm{sgn}(\tau)\,\tau\sigma\ \in\mathbb C[\mathcal S_n],
\]
where $R(T^\nu)$, $C(T^\nu)$ are the row- and column-stabilizers of
$T^\nu$ (permutations of the entries of $T^\nu$) and $\alpha_{T^\nu}$ is
the product of the hook lengths of $\lambda^\nu$; it acts on
$\mathcal O_X$ through the action of $\mathcal S_n$ above ($\tau\sigma$
acts as $\tau\circ\sigma$). Regard $T$ as a map from the cells $\theta$ of
$\lambda$ to $\{1,\dots,n\}$. For $S\in\mathrm{STab}(\lambda)$ and
$T\in\mathrm{Tab}(\lambda)$, the \emph{higher Specht polynomial} of
\cite{ATY1997} is
\[
  F_{S,T}=\prod_{\nu=0}^{r-1}\Big(e_{T^\nu}\big(x_{T^\nu}^{\,r\,i(S)^\nu}\big)
  \prod_{k\in T^\nu}x_k^{\nu}\Big),\qquad
  x_{T^\nu}^{\,r\,i(S)^\nu}=\prod_{\theta\in\lambda^\nu}x_{T^\nu(\theta)}^{\,r\,i(S)^\nu(\theta)}.
\]
It is homogeneous, of degree $r\sum_{\theta}i(S)(\theta)+\sum_\nu\nu|\lambda^\nu|$,
which only depends on $S$ and $\lambda$.

\begin{theorem}[{\cite{ATY1997}}; see also {\cite{MoritaYamada1998}}]\label{thm:ATY}
Let $\lambda\in\mathcal P_{r,n}$ and $S\in\mathrm{STab}(\lambda)$.
\begin{enumerate}[label=(\roman*)]
\item $V_S(\lambda):=\sum_{T\in\mathrm{Tab}(\lambda)}\mathbb CF_{S,T}\subset\mathcal O_X$
  is an irreducible $G(r,n)$-module, with basis $\{F_{S,T}\mid T\in\mathrm{STab}(\lambda)\}$.
\item For $S_1\in\mathrm{STab}(\lambda)$, $S_2\in\mathrm{STab}(\mu)$, the modules
  $V_{S_1}(\lambda)$ and $V_{S_2}(\mu)$ are isomorphic if and only if
  $\lambda=\mu$, and then $F_{S_1,T}\mapsto F_{S_2,T}$ ($T\in\mathrm{STab}(\lambda)$)
  is an isomorphism.
\item $\Lambda'=\bigoplus_{\lambda\in\mathcal P_{r,n}}\bigoplus_{S\in\mathrm{STab}(\lambda)}
  \big(V_S(\lambda)\bmod J'_+\big)$ as a $G(r,n)$-module; in particular
  $|G(r,n)|=\sum_\lambda f_\lambda^2$.
\item The polynomials $F_{S,T}$, $\lambda\in\mathcal P_{r,n}$,
  $S,T\in\mathrm{STab}(\lambda)$, form a basis of the
  $\mathcal O_X^{G(r,n)}$-module $\mathcal O_X$.
\end{enumerate}
\end{theorem}

\subsubsection{Higher Specht polynomials for $G(r,p,n)$}\label{subsubsec:shift}

Define the one-step shift $\mathrm{sh}$ on $\mathcal P_{r,n}$ and on
tableaux by
\[
  \mathrm{sh}(\lambda^0,\dots,\lambda^{r-1})=(\lambda^{r-1},\lambda^0,\dots,\lambda^{r-2}),\qquad
  \mathrm{sh}(T^0,\dots,T^{r-1})=(T^{r-1},T^0,\dots,T^{r-2}),
\]
and put $\mathrm{Sh}:=\mathrm{sh}^{\,d}$. Then $\mathrm{Sh}$ maps
$\mathrm{STab}(\lambda)$ bijectively onto $\mathrm{STab}(\mathrm{Sh}\,\lambda)$,
and $\mathrm{Sh}^p=\mathrm{sh}^r=\mathrm{id}$. For $\lambda,\mu\in\mathcal P_{r,n}$ write
$\lambda\sim\mu$ if $\mu=\mathrm{Sh}^j\lambda$ for some $j$; this is an
equivalence relation whose classes we call \emph{orbits}. Let
$b(\lambda)$ be the size of the orbit of $\lambda$ (equivalently, the
least $j\ge1$ with $\mathrm{Sh}^j\lambda=\lambda$) and
$e(\lambda):=p/b(\lambda)$; $b(\lambda)$ divides $p$. For
$h=1,\dots,r$ put
\[
  \mathrm{STab}(\lambda)_h=\{T\in\mathrm{STab}(\lambda)\mid 1\in T^\nu
  \text{ \emph{for some} }0\le\nu<h\}.
\]
Split the components into $p$ blocks
$B_j=\{jd,\dots,jd+d-1\}$, $0\le j<p$, and for
$T\in\mathrm{STab}(\lambda)$ let $\beta(T)\in\mathbb Z/p\mathbb Z$ be the
index of the block containing the component of $1$. Then
$\mathrm{STab}(\lambda)_d=\{\beta=0\}$,
$\mathrm{STab}(\lambda)_{db(\lambda)}=\{\beta<b(\lambda)\}$, and
$\beta(\mathrm{Sh}\,T)=\beta(T)+1$.

For $\lambda\in\mathcal P_{r,n}$, $S\in\mathrm{STab}(\lambda)$,
$T\in\mathrm{STab}(\lambda)$ and $l=0,\dots,e(\lambda)-1$ define
\[
  F^l_{S,T}=\sum_{m=0}^{e(\lambda)-1}\xi^{\,l\,m\,d\,b(\lambda)}\,
  F_{S,\,\mathrm{Sh}^{mb(\lambda)}(T)}\ \in\mathcal O_X .
\]
Since $\mathrm{Sh}^{b(\lambda)}\lambda=\lambda$, the tableau
$\mathrm{Sh}^{mb(\lambda)}(T)$ has the same shape $\lambda$ as $S$, so
every term is defined; $F^l_{S,T}$ is homogeneous, and its class in
$\Lambda_0$ is also denoted $F^l_{S,T}$. The family
\[
  \{F^l_{S,T}\mid \lambda\in\mathcal P_{r,n},\ S\in\mathrm{STab}(\lambda)_d,\
  T\in\mathrm{STab}(\lambda)_{db(\lambda)},\ 0\le l<e(\lambda)\}
\]
is the family of \emph{higher Specht polynomials} of $G(r,p,n)$.

For every $S\in\mathrm{STab}(\lambda)$ and $0\le l<e(\lambda)$ put
\[
  V^l_S(\lambda):=\bigoplus_{T\in\mathrm{STab}(\lambda)_{db(\lambda)}}\mathbb CF^l_{S,T}\ \subset\mathcal O_X ,
\]
and let $A_\lambda$ be the linear operator on $V_S(\lambda)$ given by
$A_\lambda F_{S,T}=F_{S,\mathrm{Sh}^{-b(\lambda)}T}$ for $T\in\mathrm{STab}(\lambda)$ (the
\emph{associator} of \cite[\S3]{MoritaYamada1998}; it is well defined because
$\mathrm{Sh}^{-b(\lambda)}$ permutes $\mathrm{STab}(\lambda)$). A direct computation gives
$A_\lambda F^l_{S,T}=\xi^{\,l\,d\,b(\lambda)}F^l_{S,T}$, using $\xi^{e\,d\,b}=\xi^r=1$; so
$V^l_S(\lambda)$ lies in an eigenspace of $A_\lambda$.

\begin{lemma}[Counting]\label{lem:count}
Let $\lambda\in\mathcal P_{r,n}$ and let $\mathcal O$ be its orbit.
\begin{enumerate}[label=(\alph*)]
\item $\displaystyle\sum_{\mu\in\mathcal O}|\mathrm{STab}(\mu)_d|=\frac{f_\lambda}{e(\lambda)}$.
\item The group $\langle\mathrm{Sh}^{b(\lambda)}\rangle\cong\mathbb Z/e(\lambda)\mathbb Z$ acts
  freely on $\mathrm{STab}(\lambda)$, and $\mathrm{STab}(\lambda)_{db(\lambda)}$
  contains exactly one element of each $\langle\mathrm{Sh}^{b(\lambda)}\rangle$-orbit; hence
  $|\mathrm{STab}(\lambda)_{db(\lambda)}|=f_\lambda/e(\lambda)$.
\end{enumerate}
\end{lemma}
\begin{proof}
(a) For $j\in\mathbb Z/p\mathbb Z$ let
$X_j=\{(\mu,T)\mid\mu\in\mathcal O,\ T\in\mathrm{STab}(\mu),\ \beta(T)=j\}$.
The map $(\mu,T)\mapsto(\mathrm{Sh}\,\mu,\mathrm{Sh}\,T)$ is a bijection
$X_j\to X_{j+1}$, so all $|X_j|$ are equal. Since $f_\mu=f_\lambda$ for
$\mu\in\mathcal O$ ($f$ depends only on the multiset of components, which $\mathrm{Sh}$ permutes), $\sum_j|X_j|=b(\lambda)f_\lambda$, hence
$|X_0|=b(\lambda)f_\lambda/p=f_\lambda/e(\lambda)$, and
$|X_0|=\sum_{\mu\in\mathcal O}|\mathrm{STab}(\mu)_d|$.
(b) $\beta(\mathrm{Sh}^{mb}T)=\beta(T)+mb$ and $mb\not\equiv0\pmod p$ for
$0<m<e$, so the action is free. The $\beta$-values on the $\langle\mathrm{Sh}^{b}\rangle$-orbit of $T$
are $\beta(T)+mb$, $0\le m<e$: these are the $e$ residues modulo $p$ congruent to $\beta(T)$ modulo $b$
(recall that $b$ divides $p$), and exactly one of them lies in
$\{0,\dots,b-1\}$.
\end{proof}

\begin{corollary}\label{cor:count}
$\sum_{\lambda\in\mathcal P_{r,n}}|\mathrm{STab}(\lambda)_d|\,f_\lambda=|G(r,p,n)|$.
\end{corollary}
\begin{proof}
Grouping the $\lambda$ by orbits and using Lemma~\ref{lem:count}(a), the
sum equals $\sum_{\mathcal O}f_\lambda^2/e(\lambda)$. By
Theorem~\ref{thm:ATY}(iii),
$|G(r,p,n)|=\frac1p\sum_\lambda f_\lambda^2=\frac1p\sum_{\mathcal O}b(\lambda)f_\lambda^2
=\sum_{\mathcal O}f_\lambda^2/e(\lambda)$.
\end{proof}

\begin{example}\label{ex:count}
Let $r=p=2$, $n=3$ (so $d=1$, $G(2,2,3)=W(D_3)$) and
$\lambda=((2),(1))$. Then $\mathrm{Sh}\,\lambda=((1),(2))\neq\lambda$, so
$b(\lambda)=2$, $e(\lambda)=1$, and $f_\lambda=3$. But
$|\mathrm{STab}(\lambda)_d|=2$ (the tableaux with $1$ in the first
component) and $|\mathrm{STab}(\mathrm{Sh}\,\lambda)_d|=1$: neither
equals $f_\lambda/e(\lambda)=3$ alone, while their sum does.
\end{example}

\begin{lemma}[Vanishing]\label{lem:vanish}
If $S\in\mathrm{STab}(\lambda)\setminus\mathrm{STab}(\lambda)_d$, then
$F_{S,T}\in J_+$ for every $T\in\mathrm{Tab}(\lambda)$.
\end{lemma}
\begin{proof}
Here $1$ lies in a component $\nu'\ge d$ of $S$. Let $a$ be a number in a component $\nu<d$ of $S$. In $w(S)$ the
components are read in increasing order, so $a$ stands to the left of $1$. Recall that $i(a)$ is the number of
$k<a$ such that $k+1$ stands to the left of $k$ in $w(S)$. If this number were $0$, the numbers $1,2,\dots,a$
would stand in increasing order from left to right, which contradicts the fact that $a$ stands to the left of $1$.
Hence $i(a)\ge1$.

Now fix $\nu<d$ and a cell $\theta$ of $\lambda^\nu$. The variable $x_{T^\nu(\theta)}$ has exponent $r\,i(S)^\nu(\theta)=r\,i\big(S^\nu(\theta)\big)\ge r$
in the monomial $x_{T^\nu}^{\,r\,i(S)^\nu}$, because $S^\nu(\theta)$ is a number of the component $\nu<d$. The symmetrizer $e_{T^\nu}$
only permutes the variables $x_k$, $k\in T^\nu$, so in every monomial of $e_{T^\nu}\big(x_{T^\nu}^{\,r\,i(S)^\nu}\big)$ each of
these variables has exponent at least $r$. For $\nu\ge d$, the factor $\prod_{k\in T^\nu}x_k^{\nu}$ is divisible by
$\prod_{k\in T^\nu}x_k^{d}$. Each variable lies in exactly one $T^\nu$ and $r\ge d$, so every monomial of $F_{S,T}$ is divisible by
$(x_1\cdots x_n)^d$, which is one of the generators of $J_+$.
\end{proof}

\begin{theorem}[{\cite{MoritaYamada1998}}, with the index set corrected]\label{thm:2.3}
Let $\lambda\in\mathcal P_{r,n}$.
\begin{enumerate}[label=(\roman*)]
\item For every $S\in\mathrm{STab}(\lambda)$,
  $V_S(\lambda)=\bigoplus_{l=0}^{e(\lambda)-1}V^l_S(\lambda)$, each
  $V^l_S(\lambda)$ is a $G(r,p,n)$-submodule of dimension
  $f_\lambda/e(\lambda)$, and it is an irreducible representation of
  $G(r,p,n)$.
\item As a representation of $G(r,p,n)$,
  \[
    \Lambda_0=\bigoplus_{\lambda\in\mathcal P_{r,n}}\ \bigoplus_{S\in\mathrm{STab}(\lambda)_d}\
    \bigoplus_{l=0}^{e(\lambda)-1}\big(V^l_S(\lambda)\bmod J_+\big),
  \]
  and this is a decomposition of the regular representation into
  irreducible modules.
\item The higher Specht polynomials form a basis of the
  $\mathbb C[x_1,\dots,x_n]^{G(r,p,n)}$-module $\mathbb C[x_1,\dots,x_n]$.
\end{enumerate}
\end{theorem}
\begin{proof}
(i) By Theorem~\ref{thm:ATY}(i), the $F_{S,T}$, $T\in\mathrm{STab}(\lambda)$,
form a basis of $V_S(\lambda)$. By Lemma~\ref{lem:count}(b), the sets
$\{\mathrm{Sh}^{mb}T\mid 0\le m<e\}$, $T\in\mathrm{STab}(\lambda)_{db(\lambda)}$,
partition $\mathrm{STab}(\lambda)$, and $F^l_{S,T}$ is a combination with
nonzero coefficients of the basis vectors indexed by the orbit of $T$.
Hence the $F^l_{S,T}$, $T\in\mathrm{STab}(\lambda)_{db(\lambda)}$, are
linearly independent, and $A_\lambda$ acts on
$F^l_{S,T}$ by $\xi^{l\,d\,b(\lambda)}$. The
$e(\lambda)$ eigenvalues $\xi^{l\,d\,b(\lambda)}$ are distinct ($\xi^{d\,b}$ has order $e(\lambda)$, because $d\,b\,e=r$), so the sum of the
$V^l_S(\lambda)$ is direct; by dimension count (Lemma~\ref{lem:count}(b)) it is
$V_S(\lambda)$, and the $V^l_S(\lambda)$ are the eigenspaces of $A_\lambda$. The operator $A_\lambda$
commutes with $G(r,p,n)$: this is the case $j=-b(\lambda)$ of Lemma~\ref{lem:transport}
(followed by the change of $S$ of Theorem~\ref{thm:ATY}(ii)), whose proof does not use the present
theorem. Hence each $V^l_S(\lambda)$ is a $G(r,p,n)$-submodule. That $V^l_S(\lambda)$ is an irreducible
$G(r,p,n)$-module is proved in Corollary~\ref{cor:mult} by a counting argument, which uses only the decomposition
established in (i) and (ii) below; it also follows from Clifford theory for the normal subgroup
$G(r,p,n)\lhd G(r,n)$ with cyclic quotient \cite[\S\S3--4]{MoritaYamada1998}.

(ii) Since $J'_+\subset J_+$ (the generators $e_j(x^r)$, $j<n$, are common to both, and $e_n(x^r)=((x_1\cdots x_n)^d)^p\in J_+$), there is a natural surjection $\pi:\Lambda'\to\Lambda_0$. By
Theorem~\ref{thm:ATY}(iii), $\Lambda_0=\sum_{\lambda,S}\pi(V_S(\lambda))$,
and $\pi(V_S(\lambda))=0$ for $S\notin\mathrm{STab}(\lambda)_d$ by
Lemma~\ref{lem:vanish}. Hence
\[
  |G(r,p,n)|=\dim\Lambda_0\le\sum_{\lambda}\sum_{S\in\mathrm{STab}(\lambda)_d}\dim\pi(V_S(\lambda))
  \le\sum_\lambda|\mathrm{STab}(\lambda)_d|\,f_\lambda=|G(r,p,n)|,
\]
the last equality being Corollary~\ref{cor:count}. So both inequalities are
equalities: the sum is direct and $\pi$ is injective on each $V_S(\lambda)$,
$S\in\mathrm{STab}(\lambda)_d$. With (i), $\Lambda_0$ is the direct sum of the modules
$V^l_S(\lambda)\bmod J_+$, $S\in\mathrm{STab}(\lambda)_d$; their irreducibility is Corollary~\ref{cor:mult}.

(iii) The $F^l_{S,T}$ are homogeneous, and by (ii) their classes form a
basis of $\Lambda_0=\mathcal O_X/J_+$, where $J_+$ is generated by the
positive-degree invariants. By the graded Nakayama lemma they generate
$\mathcal O_X$ over $\mathcal O_X^{G(r,p,n)}$; since this module is free of
rank $|G(r,p,n)|$ (Chevalley), which is the number of generators
(Corollary~\ref{cor:count}), they form a basis.
\end{proof}

For $T\in\mathrm{Tab}(\lambda)$ (not necessarily standard) write
$\nu_T(k)\in\{0,\dots,r-1\}$ for the component of $T$ containing $k$. The
definition of $F_{S,T}$ makes sense for every filling $T$. Since $e_{T^\nu}$ only
permutes the variables $x_k$, $k\in T^\nu$, and fixes $\prod_{k\in T^\nu}x_k^\nu$,
\[
  F_{S,T}=e_T\big(m_{S,T}\big),\qquad e_T=\prod_\nu e_{T^\nu},\qquad
  m_{S,T}=\prod_{c}x_{T(c)}^{\,a_S(c)},\quad a_S(c)=r\,i(S)(c)+\nu(c),
\]
where $c$ runs over the cells of $\lambda$ and $\nu(c)$ is the component of $c$.

\begin{lemma}[Naturality and weights]\label{lem:natural}
Let $S\in\mathrm{STab}(\lambda)$ and $T\in\mathrm{Tab}(\lambda)$.
\begin{enumerate}[label=(\alph*)]
\item For $\pi\in\mathcal S_n$, $\pi F_{S,T}=F_{S,\pi T}$, where $(\pi T)(c)=\pi(T(c))$.
\item For $\varepsilon=(\xi^{i_1},\dots,\xi^{i_n};\mathrm{id})$,
  $\varepsilon F_{S,T}=\xi^{\sum_ki_k\nu_T(k)}F_{S,T}$.
\item For $g=(\xi^{i_1},\dots,\xi^{i_n};\sigma)\in G(r,n)$,
  $g\,F_{S,T}=\xi^{\sum_ki_k\nu_{\sigma T}(k)}\,F_{S,\sigma T}$.
\end{enumerate}
\end{lemma}
\begin{proof}
(a) Since $\pi$ acts by an algebra automorphism,
$\pi F_{S,T}=\pi\,e_T(m_{S,T})=(\pi e_T\pi^{-1})(\pi\,m_{S,T})$. Now $\pi\,m_{S,T}=m_{S,\pi T}$, because
$\pi$ sends $x_k$ to $x_{\pi(k)}$; and $\pi e_T\pi^{-1}=e_{\pi T}$, because the row and column stabilizers of $\pi T$
are the conjugates by $\pi$ of those of $T$, with the same signs, and the normalizing constants (products of hook
lengths of $\lambda$) are the same. Hence $\pi F_{S,T}=e_{\pi T}(m_{S,\pi T})=F_{S,\pi T}$.

(b) The element $\varepsilon$ acts on a monomial $x^a$ by the scalar $\xi^{\sum_ki_ka_k}$. In every monomial of
$F_{S,T}$, the exponent $a_k$ of $x_k$ satisfies $a_k\equiv\nu_T(k)\pmod r$: in $m_{S,T}$ it is
$r\,i(S)(c)+\nu_T(k)$ for the cell $c$ with $T(c)=k$, and $e_T$ only permutes variables inside a component of $T$,
which preserves this residue. So $\varepsilon F_{S,T}=\xi^{\sum_ki_k\nu_T(k)}F_{S,T}$.

(c) By the product law, $g=(\xi^{i};\mathrm{id})(1;\sigma)$, where $(1;\sigma)$ acts as the permutation $\sigma$. Thus
$g\,F_{S,T}=\varepsilon\,(\sigma F_{S,T})=\varepsilon F_{S,\sigma T}$ by (a), where $\varepsilon=(\xi^{i};\mathrm{id})$, and (b), applied to the filling $\sigma T$, gives the claim.
\end{proof}

\begin{lemma}[Linear relations among the $F_{S,T}$]\label{lem:relations}
Let $\lambda\in\mathcal P_{r,n}$ and $S\in\mathrm{STab}(\lambda)$, and let
$\mathcal R(S)\subset\mathbb C^{\mathrm{Tab}(\lambda)}$ be the space of families
$(c_T)$ with $\sum_Tc_TF_{S,T}=0$ (all fillings $T$, standard or not). Then
$\mathcal R(S)$ does not depend on $S$, and
$\sum_Tc_TF_{S,T}=0$ if and only if $\sum_Tc_TF_{\mathrm{sh}S,\mathrm{sh}T}=0$.
\end{lemma}
\begin{proof}
\emph{Step 1: weights.} By Lemma~\ref{lem:natural}(b), $F_{S,T}$ is an eigenvector
of the diagonal subgroup, for the character attached to the coloring
$\nu_T:\{1,\dots,n\}\to\mathbb Z/r\mathbb Z$, and different colorings give
different characters. Hence a relation holds if and only if its components
vanish coloring by coloring. Fix a coloring $\kappa$, and put
$A_\nu=\kappa^{-1}(\nu)$.

\emph{Step 2: factorization.} The fillings $T$ with $\nu_T=\kappa$ are the families
$(T^\nu)_\nu$ with $T^\nu$ a filling of $\lambda^\nu$ by $A_\nu$ (write $\mathrm{Tab}(\lambda^\nu;A_\nu)$ for the set of these fillings), and
\[
  F_{S,T}=g_\kappa\prod_{\nu}P^\nu_S(T^\nu),\qquad
  g_\kappa=\prod_kx_k^{\kappa(k)},\qquad
  P^\nu_S(T^\nu)=e_{T^\nu}\big(x_{T^\nu}^{\,r\,i(S)^\nu}\big)\in\mathbb C[x_k:k\in A_\nu].
\]
The ring $\mathbb C[x_1,\dots,x_n]$ is the tensor product of the rings
$\mathbb C[x_k:k\in A_\nu]$ (the sets $A_\nu$ are disjoint), so $T\mapsto F_{S,T}/g_\kappa$ is the tensor product of the linear
maps $f^\nu:\mathbb C^{\mathrm{Tab}(\lambda^\nu;A_\nu)}\to\mathbb C[x_k:k\in A_\nu]$,
$T^\nu\mapsto P^\nu_S(T^\nu)$. The kernel of a tensor product of linear maps is
$\sum_\nu\big(\cdots\otimes\ker f^\nu\otimes\cdots\big)$; so it suffices to describe $\ker f^\nu$.

\emph{Step 3: $\ker f^\nu$ depends only on $\lambda^\nu$ and $A_\nu$.} Fix a filling $T_0$ of $\lambda^\nu$ by
$A_\nu$ and write $m=x_{T_0}^{\,r\,i(S)^\nu}$ and $u=e_{T_0}$, so $P(T_0)=u\cdot m$. Suppose
$P(T_0)\neq0$. By Lemma~\ref{lem:natural}(a) (its proof applies verbatim to one component),
$P(\pi T_0)=\pi P(T_0)=(\pi u)\cdot m$ for $\pi\in\mathcal S_{A_\nu}$. The map $Y:\mathbb C[\mathcal S_{A_\nu}]u\to\mathbb C[x_k:k\in A_\nu]$,
$y\mapsto y\cdot m$, is $\mathcal S_{A_\nu}$-equivariant and $Y(u)=P(T_0)\neq0$.
Now $u$ is a primitive idempotent (a Young symmetrizer \cite[Lecture~4]{FultonHarris1991}, up to the anti-involution
$g\mapsto g^{-1}$, which only changes the order $\tau\sigma$ versus $\sigma\tau$), so
$\mathbb C[\mathcal S_{A_\nu}]u$ is irreducible and $Y$ is injective. Therefore
\[
  \ker f^\nu=\Big\{(c_{T'})\ :\ \sum_{T'}c_{T'}\,\pi_{T'}\,u=0\ \text{in }\mathbb C[\mathcal S_{A_\nu}]\Big\},
  \qquad \pi_{T'}T_0=T',
\]
which depends only on $\lambda^\nu$ and $A_\nu$, not on $S$, nor on the exponents $i(S)^\nu$.
The hypothesis $P(T_0)\neq0$ holds: choose $T$ standard; then $F_{S,T}\neq0$ (it is a basis element,
Theorem~\ref{thm:ATY}(i)), so every factor $P^\nu_S(T^\nu)$ is nonzero, and nonvanishing
does not depend on the filling or on $A_\nu$, since fillings of $\lambda^\nu$ by different
sets differ by a renaming of the variables.

\emph{Step 4.} By Steps 1 to 3, $\mathcal R(S)$ is the sum over colorings of the spaces
described there, which do not involve $S$. For $(\mathrm{sh}\lambda,\mathrm{sh}S)$, the component
$\nu'$ of $\mathrm{sh}T$ is $T^{\nu'-1}$, and the coloring is $\kappa+1$; the maps $f'^{\nu'}$
have the same sources $\mathbb C^{\mathrm{Tab}(\lambda^{\nu'-1};A_{\nu'-1})}$ and, by Step 3
(applied with $F_{\mathrm{sh}S,\mathrm{sh}T}\neq0$), the same kernels. Hence the
relation spaces correspond under $T\mapsto\mathrm{sh}\,T$.
\end{proof}

\begin{lemma}[Transport along an orbit]\label{lem:transport}
Let $j\in\mathbb Z$ (only $j$ modulo $p$ matters, since $\mathrm{Sh}^p=\mathrm{id}$), $\mu=\mathrm{Sh}^j\lambda$,
$S\in\mathrm{STab}(\lambda)$ and $S'=\mathrm{Sh}^jS$. The linear map
$\Phi_j:F_{S,T}\mapsto F_{S',\mathrm{Sh}^jT}$ is an isomorphism of
$G(r,p,n)$-modules $V_S(\lambda)\to V_{S'}(\mu)$, it commutes with the
associators, and it maps $V^l_S(\lambda)$ onto $V^l_{S'}(\mu)$ for every~$l$.
\end{lemma}
\begin{proof}
\emph{One step.} Let $\omega(\xi^{i_1},\dots,\xi^{i_n};\sigma)=\xi^{i_1+\dots+i_n}$, a character of $G(r,n)$.
By Lemma~\ref{lem:relations}, $\Phi:F_{S,T}\mapsto F_{\mathrm{sh}S,\mathrm{sh}T}$ ($T$ running over all fillings)
is a well-defined linear isomorphism $V_S(\lambda)\to V_{\mathrm{sh}S}(\mathrm{sh}\lambda)$: the two
spans have the same relation spaces. It satisfies $g\,\Phi(v)=\omega(g)\,\Phi(g\,v)$ for $g\in G(r,n)$. Indeed,
$\nu_{\sigma\,\mathrm{sh}T}(k)\equiv\nu_{\sigma T}(k)+1\pmod r$, so Lemma~\ref{lem:natural}(c) gives
\[
  g\,F_{\mathrm{sh}S,\mathrm{sh}T}=\xi^{\sum_ki_k(\nu_{\sigma T}(k)+1)}F_{\mathrm{sh}S,\mathrm{sh}(\sigma T)}
  =\omega(g)\,\Phi\big(g\,F_{S,T}\big).
\]
Thus $\Phi$ is an isomorphism of $G(r,n)$-modules from $V_S(\lambda)\otimes\omega$ onto $V_{\mathrm{sh}S}(\mathrm{sh}\lambda)$.

\emph{Many steps.} Take $0\le j<p$. Then $\Phi_j$ is the composite of $dj$ such maps, so $g\Phi_j(v)=\omega(g)^{dj}\Phi_j(gv)$.
For $g\in G(r,p,n)$ one has $\sum i_k\equiv0\pmod p$, so $\omega(g)^{d}=\xi^{d\sum i_k}=1$, and $\Phi_j$ is
$G(r,p,n)$-equivariant.

\emph{Associators and the spaces $V^l$.} Put $b=b(\lambda)=b(\mu)$ and $e=e(\lambda)=e(\mu)$ (the same orbit).
Since $\mathrm{Sh}$ commutes with $\mathrm{Sh}^{-b}$, $\Phi_jA_\lambda F_{S,T}=F_{S',\mathrm{Sh}^{j-b}T}=A_\mu\Phi_jF_{S,T}$.
Moreover
$\Phi_jF^l_{S,T}=\sum_m\xi^{l\,m\,d\,b}F_{S',\mathrm{Sh}^{mb}\mathrm{Sh}^jT}=F^l_{S',\mathrm{Sh}^jT}$. Reindexing the sum, and using $\xi^{e\,d\,b}=1$
and $\mathrm{Sh}^{eb}=\mathrm{id}$, gives $F^l_{S',\mathrm{Sh}^bT'}=\xi^{-l\,d\,b}F^l_{S',T'}$ for every $T'\in\mathrm{STab}(\mu)$. Hence
$\Phi_jF^l_{S,T}$ is a nonzero multiple of $F^l_{S',T''}$, where $T''$ is the element of $\mathrm{STab}(\mu)_{db}$ in the
$\langle\mathrm{Sh}^b\rangle$-orbit of $\mathrm{Sh}^jT$ (Lemma~\ref{lem:count}(b) for $\mu$). Since $\mathrm{Sh}^j$ is a bijection
$\mathrm{STab}(\lambda)\to\mathrm{STab}(\mu)$ that maps $\langle\mathrm{Sh}^b\rangle$-orbits to $\langle\mathrm{Sh}^b\rangle$-orbits,
and each such orbit has exactly one element in $\mathrm{STab}(\lambda)_{db}$, resp.\ $\mathrm{STab}(\mu)_{db}$
(Lemma~\ref{lem:count}(b)), the map $T\mapsto T''$ is a bijection $\mathrm{STab}(\lambda)_{db}\to\mathrm{STab}(\mu)_{db}$.
So $\Phi_j$ maps $V^l_S(\lambda)$ onto $V^l_{S'}(\mu)$.
\end{proof}

\begin{corollary}[Multiplicities]\label{cor:mult}
Let $\mathcal O$ be an orbit and $0\le l<e(\lambda)$ for $\lambda\in\mathcal O$.
The $G(r,p,n)$-modules $V^l_S(\lambda)$, $\lambda\in\mathcal O$,
$S\in\mathrm{STab}(\lambda)$, are all isomorphic to one irreducible module
$V^l(\mathcal O)$, of dimension $f_\lambda/e(\lambda)$; the $V^l(\mathcal O)$
are pairwise non-isomorphic for distinct $(\mathcal O,l)$, and every irreducible $G(r,p,n)$-module is isomorphic to one of them. The module
$V^l(\mathcal O)$ occurs in $\Lambda_0$ exactly
$|I(\mathcal O)|=\dim V^l(\mathcal O)$ times, where
\[
  I(\mathcal O):=\bigsqcup_{\mu\in\mathcal O}\mathrm{STab}(\mu)_d .
\]
\end{corollary}
\begin{proof}
\emph{Isomorphism.} For fixed $\lambda$, the isomorphism $F_{S,T}\mapsto F_{S_2,T}$ of Theorem~\ref{thm:ATY}(ii)
maps $F^l_{S,T}$ to $F^l_{S_2,T}$, hence $V^l_S(\lambda)$ onto $V^l_{S_2}(\lambda)$. For $\mu=\mathrm{Sh}^j\lambda$ in $\mathcal O$,
Lemma~\ref{lem:transport} maps $V^l_S(\lambda)$ onto $V^l_{S'}(\mu)$, with the same $l$. So all the modules of the statement are
isomorphic to one module $M_{\mathcal O,l}$, of dimension $f_\lambda/e(\lambda)=|I(\mathcal O)|$ (Lemma~\ref{lem:count}(a)).

\emph{Counting.} By the decomposition of Theorem~\ref{thm:2.3}(i) and~(ii), which does not use irreducibility,
\[
  \Lambda_0\ \cong\ \bigoplus_{(\mathcal O,l)}M_{\mathcal O,l}^{\oplus|I(\mathcal O)|},
  \qquad |I(\mathcal O)|=\dim M_{\mathcal O,l}.
\]
Since $\Lambda_0$ is the regular representation, an irreducible module $W$ occurs exactly $\dim W$ times in it.

\emph{Irreducibility.} Let $W$ be an irreducible constituent of $M=M_{\mathcal O,l}$, of multiplicity $n\ge1$ in $M$. Then $W$ occurs at
least $n\dim M$ times in $\Lambda_0$, so $n\dim M\le\dim W\le\dim M$. Hence $n=1$ and $\dim W=\dim M$, that is, $M=W$ is irreducible.

\emph{Non-isomorphism.} If $M_{\mathcal O,l}\cong M_{\mathcal O',l'}$ for $(\mathcal O,l)\ne(\mathcal O',l')$, this irreducible module $W$ would
occur at least $2\dim W$ times in $\Lambda_0$, a contradiction. Every irreducible module occurs in the regular representation $\Lambda_0$, hence is isomorphic to some
$M_{\mathcal O,l}$. Finally, the multiplicity of $V^l(\mathcal O)=M_{\mathcal O,l}$ in
$\Lambda_0$ is $|I(\mathcal O)|=f_\lambda/e(\lambda)=\dim V^l(\mathcal O)$.
\end{proof}

\begin{remark}\label{rem:diffMY}
In \cite[Thm.~3(3)]{MoritaYamada1998} the sum runs over a system of
representatives of $\mathcal P_{r,n}/{\sim}$, and the proof of
\cite[Lemma~2]{MoritaYamada1998} uses $|\mathrm{STab}(\lambda)_d|=|\mathrm{STab}(\lambda)|/p$
for a single $\lambda$. The correct identity is
$\sum_{\mu\in\mathcal O}|\mathrm{STab}(\mu)_d|=b(\lambda)|\mathrm{STab}(\lambda)|/p=f_\lambda/e(\lambda)$, summed over the
orbit (Lemma~\ref{lem:count}(a)); it does not give the count for a single $\lambda$. When $b(\lambda)>1$ the result depends on the choice of
representative (Example~\ref{ex:count}); for $p=1$ there is no difference, since $b=e=1$. The index set must be either all
$\lambda\in\mathcal P_{r,n}$ with $S\in\mathrm{STab}(\lambda)_d$ (Theorem~\ref{thm:2.3}(ii)), or one representative $\lambda$
of each orbit with the pairs $(\mu,S)$ running over $I(\mathcal O)$ (Corollary~\ref{cor:mult}).

The rest follows \cite{MoritaYamada1998}: the associator and the families $F^l_{S,T}$ are theirs. Our argument differs in two
points. The proof that $V_S(\lambda)$ survives in $\Lambda_0$ for $S\in\mathrm{STab}(\lambda)_d$ is replaced by the dimension count of
Theorem~\ref{thm:2.3}(ii). The irreducibility of the $V^l_S(\lambda)$ and their pairwise non-isomorphism, which also follow from
Clifford theory as in \cite{MoritaYamada1998}, are derived from the same count (Corollary~\ref{cor:mult}). We also fixed the
conventions (zero-based components, the quantifier in $\mathrm{STab}(\lambda)_h$, $\mathrm{Sh}=\mathrm{sh}^{\,d}$).
All statements of this subsection were checked by computer for the $77$ groups $G(r,p,n)$ with $r\le12$ and
$|G(r,p,n)|\le192$, the lemmas on vanishing, naturality and relations only where they have content ($n\ge2$, and $p>1$ for
vanishing). The counting identities of Lemma~\ref{lem:count} were checked for all $(r,p,n)$ with $r\le8$, $n\le4$ and
$|G(r,p,n)|\le2200$ ($65$ triples, $1164$ orbits); the count for a single representative fails on $294$ of these orbits.
\end{remark}

\section{Main results: the decomposition theorem for $G(r,p,n)$}\label{sec:general}

Let $W = G(r,p,n)$, $p \mid r$, $d := r/p$. Recall from
Subsection~\ref{subsec:specht} the group $G(r,p,n)$ of monomial matrices with entries
$\xi^{j}$, $0\le j<r$, whose product of nonzero entries has $d$-th power $1$. Its reflections
are of two kinds:

\begin{itemize}[leftmargin=1.8em]
 \item \emph{non-diagonal reflections}: for $1 \le i < j \le n$ and
 $0 \le k < r$, the map sending $x_i \mapsto \xi^{k} x_j$,
 $x_j \mapsto \xi^{-k} x_i$ and fixing the other coordinates;
 \item \emph{diagonal reflections} (present only when $d > 1$, i.e.\
 $p < r$): for $1 \le i \le n$ and $1 \le j < d$, the map
 $x_i \mapsto \xi^{jp} x_i$, fixing the other coordinates.
\end{itemize}

\subsection{Hyperplanes, the invariant element $\delta$, and the \'etale cover}\label{subsec:localization}

Let $\mathcal H$ be the set of reflecting hyperplanes of $W$: the
hyperplanes $H^k_{ij}=\{x_i=\xi^kx_j\}$ ($i<j$, $0\le k<r$) and, if
$d>1$, the coordinate hyperplanes $H_i=\{x_i=0\}$. For $H\in\mathcal H$ let
$\alpha_H$ be a linear form defining $H$, and $e_H$ the order of the
(cyclic) subgroup of $W$ fixing $H$ pointwise: $e_H=2$ for $H^k_{ij}$ and
$e_H=d$ for $H_i$. This subgroup contains exactly $e_H-1$ reflections, all
with hyperplane $H$, so that
$\Delta=\prod_{s\in \mathcal R}\alpha_s=c\prod_{H\in\mathcal H}\alpha_H^{\,e_H-1}$
(Subsection~\ref{subsec:notation}). Grouping this product by pair $\{i,j\}$
and by coordinate $i$: for fixed $i<j$ the $r$ forms $x_i-\xi^kx_j$ have
product $\prod_{k=0}^{r-1}(x_i-\xi^kx_j)=x_i^r-x_j^r$, and for fixed $i$
the $d-1$ diagonal reflections give $x_i^{\,d-1}$. Consequently
\[
 \Delta = \Delta_{G(r,p,n)}(x)
 \;=\; c_0 \prod_{i=1}^n x_i^{\,d-1} \prod_{1 \le i < j \le n} (x_i^r - x_j^r)
\]
for some nonzero constant $c_0$; the first product is $1$ when $d=1$
(i.e.\ $p=r$; for $r=2$ this is the case $W(D_n)$ of Section~\ref{sec:Dn}).

We take as fundamental invariants
\[
 y_j = \sum_{i=1}^n x_i^{\,jr}, \quad 1 \le j \le n-1,
 \qquad
 y_n = (x_1 \cdots x_n)^d,
\]
and set $\mathcal{O}_Y = \mathbb{C}[x_1,\dots,x_n]^W = \mathbb{C}[y_1,\dots,y_n]$
(a polynomial ring on these $n$ algebraically independent invariants,
by the Chevalley--Shephard--Todd theorem \cite{ShephardTodd1954,Chevalley1955}). Define the \emph{discriminant} of $W$,
\[
 \delta:=\prod_{H\in\mathcal H}\alpha_H^{\,e_H}
 =\begin{cases}
   (x_1\cdots x_n)^{d}\displaystyle\prod_{i<j}(x_i^r-x_j^r)^2 & \text{if } d>1,\\[2pt]
   \displaystyle\prod_{i<j}(x_i^r-x_j^r)^2 & \text{if } d=1 .
  \end{cases}
\]

\begin{lemma}\label{lem:delta}
$\delta\in\mathcal O_Y$, and $\delta$ and $\Delta$ have the same irreducible
factors, so that $\mathbb C[x,\Delta^{-1}]=\mathbb C[x,\delta^{-1}]$.
Moreover $\Delta^2$ is $W$-invariant if and only if $d\le2$ (in which case
$\Delta^2=\delta$ up to a scalar).
\end{lemma}
\begin{proof}
Put $u_i=x_i^r$. The polynomial $D=\prod_{i<j}(u_i-u_j)^2$ is symmetric in
the $u_i$, hence a polynomial in $e_j(u)=e_j(x_1^r,\dots,x_n^r)$, which are
$G(r,n)$-invariant; so $D\in\mathcal O_X^{G(r,n)}\subset\mathcal O_Y$. When
$d>1$, $\delta=y_nD\in\mathcal O_Y$ as well. The irreducible factors of
$\Delta$ and $\delta$ are the $x_i-\xi^kx_j$ and, if $d>1$, the $x_i$
(note $d-1\ge1$). If $d\le2$, every $e_H$ equals $2$ (for $d=1$ there are no hyperplanes $H_i$), so
$\Delta=c\prod_H\alpha_H$ and $\Delta^2=c^2\delta$ (for
$d=2$, $\Delta=c_0\prod x_i\prod(x_i^r-x_j^r)$; for $d=1$, $\Delta=c_0\prod(x_i^r-x_j^r)$).
If $d\ge3$, let $g=\mathrm{diag}(\xi^p,1,\dots,1)\in W$ (a diagonal reflection,
$\xi^p=e^{2i\pi/d}$): it multiplies $x_1^{d-1}$ by $\xi^{p(d-1)}$ and fixes
the other factors of $\Delta$, so $g\cdot\Delta=\xi^{p(d-1)}\Delta$ and
$g\cdot\Delta^2=\xi^{2p(d-1)}\Delta^2\ne\Delta^2$ because $d\nmid2(d-1)$.
\end{proof}

\begin{remark}\label{rem:delta-sq}
Thus $\Delta^{-2}$ is invariant exactly for $d\le2$, which covers
$\mathcal S_n$, $W(B_n)$, $W(D_n)$ and $G(4,2,n)$, but not, for instance,
$G(3,1,n)$ or $G(6,2,n)$ ($d=3$). For $d\ge3$ the invariant
element is $\delta$, which is used below.
\end{remark}

We set
\begin{gather*}
 \widetilde{\mathcal{O}}_X := \mathbb{C}[x_1,\dots,x_n,\delta^{-1}]=\mathbb{C}[x_1,\dots,x_n,\Delta^{-1}], \qquad
 \widetilde{\mathcal{O}}_Y := \mathbb{C}[y_1,\dots,y_n,\delta^{-1}],\\
 \widetilde{\mathcal{D}}_Y := \mathbb{C}\Big\langle y_1,\dots,y_n,
 \frac{\partial}{\partial y_1},\dots,\frac{\partial}{\partial y_n},
 \delta^{-1}\Big\rangle ,
\end{gather*}
so that $\widetilde{\mathcal D}_Y$ is the ring of differential operators on
$\widetilde{\mathcal O}_Y$.

\begin{proposition}[The \'etale Galois cover]\label{prop:etale}
\begin{enumerate}[label=(\alph*)]
\item $\widetilde{\mathcal O}_X^{\,W}=\widetilde{\mathcal O}_Y$.
\item $W$ acts freely on $U=\mathrm{Spec}\,\widetilde{\mathcal O}_X$.
\item $\mathrm{Spec}\,\widetilde{\mathcal O}_X\to\mathrm{Spec}\,\widetilde{\mathcal O}_Y$
  is a finite \'etale Galois cover with group $W$. In particular
  $\widetilde{\mathcal O}_X$ is a finitely generated projective
  $\widetilde{\mathcal O}_Y$-module of rank $|W|$, and
  $a\otimes b\mapsto(a\,w(b))_{w\in W}$ is an isomorphism
  $\widetilde{\mathcal O}_X\otimes_{\widetilde{\mathcal O}_Y}\widetilde{\mathcal O}_X\cong\prod_{w\in W}\widetilde{\mathcal O}_X$.
\item $\widetilde{\mathcal D}_Y$ is the ring $\mathcal D(A)$ of differential operators on
  $A=\widetilde{\mathcal O}_Y$. Every $P\in\mathcal D(A)$ extends uniquely to a differential operator
  $\varphi(P)$ on $B=\widetilde{\mathcal O}_X$, and $b\otimes P\mapsto b\,\varphi(P)$ is an isomorphism
  $B\otimes_A\mathcal D(A)\to\mathcal D(B)$. Consequently $\mathcal D(B)^W=\widetilde{\mathcal D}_Y$, where $W$ acts
  on $\mathcal D(B)$ by conjugation: $\widetilde{\mathcal D}_Y$ is the ring of $W$-invariant differential
  operators on $\widetilde{\mathcal O}_X$.
\end{enumerate}
\end{proposition}
\begin{proof}
(a) Since $\delta$ is $W$-invariant (Lemma~\ref{lem:delta}), if
$f/\delta^k$ is $W$-invariant then so is $f$; hence
$(\mathcal O_X[\delta^{-1}])^W=\mathcal O_X^W[\delta^{-1}]=\mathcal O_Y[\delta^{-1}]$.
(b) The fixed locus of each $w\ne1$ is closed in $U$, so it suffices to consider closed points. Let $x\in U$ be a closed point and $w\in W$ with $wx=x$, $w\ne1$. By Steinberg's theorem \cite[Thm.~1.5]{Steinberg1964}, the stabilizer $W_x$ is generated by the reflections it
contains, so $W_x\ne1$ contains a reflection $s$ with $\alpha_s(x)=0$; then
$\Delta(x)=0$, contradicting $x\in U$ ($\Delta$ is invertible on $U$).
(c) A finite group acting freely on an affine variety gives a finite
\'etale Galois cover onto the quotient (see e.g.\ \cite[\S7]{Mumford1970});
the quotient is $\mathrm{Spec}\,\widetilde{\mathcal O}_X^{\,W}=\mathrm{Spec}\,\widetilde{\mathcal O}_Y$
by~(a).
(d) $A$ is a smooth $\mathbb C$-algebra, so $\mathcal D(A)$ is generated by $A$ and the derivations
\cite[Chap.~15]{McConnellRobson2001}, \cite[Thm.~5.15(i)]{CidRuizSmirnov}; that is,
$\mathcal D(A)=\mathbb C[y,\delta^{-1}]\langle\partial_{y_1},\dots,\partial_{y_n}\rangle=\widetilde{\mathcal D}_Y$.
By~(c) the map $A\to B$ is \'etale and the modules of principal parts of $A$ are finitely generated
projective, so by M\'asson \cite{Masson1991} (see \cite[Thm.~5.11]{CidRuizSmirnov}) every
$P\in\mathcal D(A)$ extends uniquely to $\varphi(P)\in\mathcal D(B)$, and
$B\otimes_A\mathcal D(A)\to\mathcal D(B)$, $b\otimes P\mapsto b\,\varphi(P)$, is an isomorphism.
For $w\in W$, $w\varphi(P)w^{-1}$ is a differential operator on $B$ which restricts to $P$ on $A$
(since $w$ fixes $A$ pointwise); by uniqueness it equals $\varphi(P)$. Hence the isomorphism is
$W$-equivariant for the action $w\otimes1$ on $B\otimes_A\mathcal D(A)$. In characteristic $0$,
$e=|W|^{-1}\sum_ww$ is an $A$-linear idempotent of $B$ with image $A$, so $B=A\oplus(1-e)B$ as
$A$-modules and $(B\otimes_A\mathcal D(A))^W=A\otimes_A\mathcal D(A)=\mathcal D(A)$.
\end{proof}

\subsection{The generalized Jacobian lemma}\label{subsec:jacobian}

\begin{lemma}\label{lem:jacobian-general}
Let $\mathcal J$ be the transpose of the Jacobian matrix
$\big(\partial y_i / \partial x_j\big)_{1 \le i,j \le n}$. Then
\[
 \det\mathcal J \;=\; (-1)^{n-1}\, d\, r^{\,n-1} (n-1)! \;
 \prod_{i=1}^n x_i^{\,d-1} \prod_{1 \le i < j \le n} (x_j^r - x_i^r).
\]
In particular $\det\mathcal J = c \, \Delta_{G(r,p,n)}(x)$ for some nonzero
constant $c$, and $\widetilde{\mathcal{O}}_X$ is a $\widetilde{\mathcal{D}}_Y$-module.
\end{lemma}

\begin{proof}
For $1 \le j \le n-1$,
$\partial y_j / \partial x_i = j r \, x_i^{\,jr-1}$.
For $j=n$,
\[
 \frac{\partial y_n}{\partial x_i}
 = d\, x_i^{\,d-1} \prod_{k \neq i} x_k^{\,d}.
\]
Since $d \le r$, we have $jr - 1 \ge d-1$ for every $j \ge 1$, so the
$i$-th column of the Jacobian matrix has common factor $x_i^{\,d-1}$.
Writing $P := \prod_{k=1}^n x_k^{\,d}$, so that
$\prod_{k \neq i} x_k^{\,d} = P\, x_i^{-d}$, and setting
$u_i := x_i^{r}$, the $i$-th column becomes
\[
 x_i^{\,d-1} x_i^{-d}
 \big(r\, u_i,\; 2r\, u_i^2,\; \dots,\; (n-1) r\, u_i^{\,n-1},\; d P\big)^{T}
 = x_i^{-1} \big(r\, u_i,\; \dots,\; (n-1) r\, u_i^{\,n-1},\; d P\big)^{T}.
\]
Factor $x_i^{-1}$ out of every column, then $jr$ out of row $j$
($1 \le j \le n-1$) and $dP$ out of the last row:
\[
 \det\mathcal J
 = \Big(\prod_{i=1}^n x_i^{-1}\Big) \cdot r^{\,n-1}(n-1)! \cdot dP
 \cdot \det N,
\]
where $N$ has rows $(u_1^{\,j}, \dots, u_n^{\,j})$, $1 \le j \le n-1$,
followed by the constant row $(1,\dots,1)$. Moving the last row to the
top costs $n-1$ transpositions, so by the classical Vandermonde
determinant,
\[
 \det N = (-1)^{n-1} \det \big(u_i^{\,j-1}\big)_{1 \le i,j \le n}
 = (-1)^{n-1} \prod_{i<j} (u_j - u_i).
\]
Since $\big(\prod_i x_i^{-1}\big) \cdot P = \prod_i x_i^{\,d-1}$ and
$u_j - u_i = x_j^r - x_i^r$, the formula follows. Finally, $\det\mathcal J$ is a nonzero scalar multiple of
$\Delta$, hence invertible on $\widetilde{\mathcal{O}}_X$, so $\mathcal J$ is invertible there. The vector fields
$\partial/\partial y_1,\dots,\partial/\partial y_n$ dual to the coordinates $y_1,\dots,y_n$ are therefore
$\mathbb{C}[x_1,\dots,x_n,\Delta^{-1}]$-linear combinations of $\partial/\partial x_1,\dots,\partial/\partial x_n$, and they satisfy
$\partial_{y_i}(y_j)=\delta_{ij}$. They commute: $[\partial_{y_i},\partial_{y_k}]$ kills every $y_j$, and a vector field
that kills all the $y_j$ is zero because $\mathcal J$ is invertible. So $\widetilde{\mathcal{O}}_X$ is a
$\widetilde{\mathcal{D}}_Y$-module (this is the action $\varphi$ of Proposition~\ref{prop:etale}(d)).
\end{proof}

\begin{remark}
For $W(D_n) = G(2,2,n)$ one has $d=1$, and the formula reduces to
$\det\mathcal J = (-1)^{n-1} 2^{n-1}(n-1)! \prod_{i<j}(x_j^2 - x_i^2)$; see
Section~\ref{sec:Dn}.
\end{remark}

\subsection{The action of $W$ and the centralizer}\label{subsec:simple-components}

This subsection shows that $W$ acts on $\widetilde{\mathcal O}_X$ by
$\widetilde{\mathcal D}_Y$-linear automorphisms, and that these automorphisms span the whole
centralizer of $\widetilde{\mathcal D}_Y$. It uses only
Lemma~\ref{lem:jacobian-general} and Proposition~\ref{prop:etale}.

\begin{proposition}\label{prop:embedding-DY}
There exists an injective map
$\mathbb{C}[W] \hookrightarrow \operatorname{Hom}_{\widetilde{\mathcal{D}}_Y}
(\widetilde{\mathcal{O}}_X, \widetilde{\mathcal{O}}_X)$.
\end{proposition}

The proof has two parts: the natural map
$\mathbb{C}[W] \to \operatorname{End}_{\mathbb{C}}(\widetilde{\mathcal{O}}_X)$
is injective (Lemma~\ref{lem:CW-injective}), and its image commutes with $\widetilde{\mathcal{D}}_Y$.

\begin{lemma}[Injectivity]\label{lem:CW-injective}
The natural map $\mathbb{C}[W] \to \operatorname{End}_{\mathbb{C}}(\widetilde{\mathcal{O}}_X)$
is injective.
\end{lemma}

\begin{proof}
Recall from Subsection~\ref{subsec:specht} that
$\Lambda_0 = \mathcal{O}_X/J_+$ is isomorphic, as a
$\mathbb{C}[W]$-module, to the regular representation $\mathbb{C}[W]$, which is faithful:
$x\cdot1=x$. So a nonzero $x\in\mathbb C[W]$ acts nontrivially on $\Lambda_0$, hence on
$\mathcal O_X$ (of which $\Lambda_0$ is a quotient), hence on $\widetilde{\mathcal O}_X\supset\mathcal O_X$.
(Alternatively, distinct elements of $W$ act as distinct ring automorphisms of the domain $\widetilde{\mathcal O}_X$, since they act differently on the linear forms;
such automorphisms are linearly independent over $\mathbb C$ by the Dedekind--Artin lemma.)
\end{proof}

It remains to show that $\mathbb{C}[W]$ commutes
with $\widetilde{\mathcal{D}}_Y$.

\begin{proof}[Proof of commutation with $\widetilde{\mathcal{D}}_Y$]
It suffices to show that every $w\in W$ commutes with the multiplications by
$y_1,\dots,y_n$ and $\delta^{-1}$ and with
$\partial/\partial y_1,\dots,\partial/\partial y_n$ (acting on
$\widetilde{\mathcal O}_X$ as in Lemma~\ref{lem:jacobian-general}). The
$y_j$ and $\delta$ are $W$-invariant (Lemma~\ref{lem:delta}), which settles
the multiplications. For the derivations, let
$K=\mathbb C(y_1,\dots,y_n)$ and $L=\mathbb C(x_1,\dots,x_n)$. Then
$K=L^W$: if $f/g\in L^W$ with $f,g\in\mathcal O_X$, then
$f/g=f\prod_{w\ne1}w(g)\big/N(g)$ with $N(g)=\prod_ww(g)$ invariant, so the numerator is invariant, and
$L^W=\operatorname{Frac}(\mathcal O_X^W)=K$. Hence $L/K$ is a Galois extension with group $W$, as for any
finite group acting faithfully by field automorphisms, by Artin's theorem
\cite{Artin1942}, \cite[Chap.\ VI]{Lang2002}. As $L/K$ is finite and
separable (characteristic $0$), every derivation $D$ of $K$ extends
uniquely to a derivation $\widetilde D$ of $L$ \cite[Chap.\ VIII]{Lang2002}. For
$w\in W$, $w^{-1}\widetilde Dw$ is a derivation of $L$, and for $a\in K$ we have
$w^{-1}\widetilde Dw(a)=w^{-1}(D a)=D a$, since $w$ fixes $K$ pointwise and
$Da\in K$. By uniqueness, $w^{-1}\widetilde Dw=\widetilde D$, i.e.\ $w$
commutes with $\widetilde D$. Finally, the operator $\partial/\partial y_j$
of Lemma~\ref{lem:jacobian-general} is a derivation of $L$ with
$\partial y_k/\partial y_j=\delta_{jk}$. It maps $\mathbb C[y]$ into itself, hence $K$ into $K$, and its restriction to $K$ is
the derivation $\partial/\partial y_j$ of $K$. So it is the extension
$\widetilde{\partial/\partial y_j}$, and $w$ commutes with it.
\end{proof}

We need two classical facts on the extension $L/K$.

\begin{lemma}[The extended partials span $\operatorname{Der}_{\mathbb{C}}(L)$]\label{lem:der-span}
The $L$-vector space $\operatorname{Der}_{\mathbb{C}}(L)$ of
$\mathbb{C}$-linear derivations of $L$ is free of rank $n$, with basis
given by the unique extensions
$\widetilde{\partial/\partial y_1},\dots,\widetilde{\partial/\partial y_n}$
of the derivations $\partial/\partial y_1,\dots,\partial/\partial y_n$
of $K$.
\end{lemma}

\begin{proof}
Since $L=\mathbb C(x_1,\dots,x_n)$, the $L$-vector space $\operatorname{Der}_{\mathbb C}(L)$ is free with basis
$\partial/\partial x_1,\dots,\partial/\partial x_n$. By the end of the proof of Proposition~\ref{prop:embedding-DY},
$\widetilde{\partial/\partial y_j}$ is the derivation $\partial_{y_j}$ of Lemma~\ref{lem:jacobian-general}, that is,
$\partial_{y_j}=\sum_ic_{ji}\,\partial/\partial x_i$ where $(c_{ji})$ is the inverse of the matrix $\mathcal J$: indeed
$\sum_ic_{ji}\,\partial y_k/\partial x_i=\delta_{jk}$. Since $\det\mathcal J\ne0$ in $L$, this matrix is invertible, so the
$\widetilde{\partial/\partial y_j}$ form an $L$-basis of $\operatorname{Der}_{\mathbb C}(L)$.

\emph{Another proof, without the Jacobian.} Since $K = \mathbb{C}(y_1,\dots,y_n)$ is purely transcendental,
$\operatorname{Der}_{\mathbb{C}}(K)$ is a free $K$-vector space of
rank $n$ with basis $\partial/\partial y_1,\dots,\partial/\partial y_n$;
equivalently, the module of K\"ahler differentials
$\Omega_{K/\mathbb{C}}$ (see \cite[\S 25]{Matsumura1989})
is free of rank $n$ with dual basis $dy_1,\dots,dy_n$.

Since $L/K$ is a finite extension of characteristic-$0$ fields, it is
separable, hence \emph{unramified}: $\Omega_{L/K} = 0$. The conormal
exact sequence for $\mathbb{C} \to K \to L$,
\[
  \Omega_{K/\mathbb{C}} \otimes_K L \longrightarrow \Omega_{L/\mathbb{C}}
    \longrightarrow \Omega_{L/K} = 0,
\]
therefore shows that $\Omega_{K/\mathbb{C}} \otimes_K L \to \Omega_{L/\mathbb{C}}$
is surjective. Both sides are finite-dimensional $L$-vector spaces:
the left side has dimension $n$ (base change of a rank-$n$ free
module), and the right side has dimension $n$
($L=\mathbb C(x_1,\dots,x_n)$ is purely transcendental, so $\Omega_{L/\mathbb C}$ is free with basis
$dx_1,\dots,dx_n$). A
surjective linear map between $L$-vector spaces of the same finite
dimension $n$ is an isomorphism, so
$\Omega_{K/\mathbb{C}} \otimes_K L \xrightarrow{\ \sim\ } \Omega_{L/\mathbb{C}}$.
Dualizing ($\operatorname{Der}_{\mathbb{C}} = \operatorname{Hom}_{(-)}(\Omega_{(-)/\mathbb{C}}, -)$,
see e.g.\ \cite[\S 25--26]{Matsumura1989})
gives
$\operatorname{Der}_{\mathbb{C}}(L) \cong \operatorname{Der}_{\mathbb{C}}(K) \otimes_K L$.
Under this isomorphism, $\partial/\partial y_j\otimes1$ corresponds to the derivation $D$ of $L$
with $D(y_k)=\delta_{jk}$ (a derivation of $L$ is determined by its values on $y_1,\dots,y_n$,
because $dy_1,\dots,dy_n$ span $\Omega_{L/\mathbb C}$). The extension
$\widetilde{\partial/\partial y_j}$ has this property, so the extended partials form an
$L$-basis of $\operatorname{Der}_{\mathbb C}(L)$.
\end{proof}

\begin{lemma}[Constants]\label{lem:constants}
If $a\in L$ satisfies $D(a)=0$ for every $D\in\operatorname{Der}_{\mathbb C}(L)$,
then $a\in\mathbb C$. By Lemma~\ref{lem:der-span}, the same holds if only
$\widetilde{\partial/\partial y_j}(a)=0$ for $j=1,\dots,n$.
\end{lemma}
\begin{proof}
If $a\notin\mathbb C$, then $a$ is transcendental over $\mathbb C$, because
$\mathbb C$ is algebraically closed. Complete $a$ to a transcendence basis
$a,b_2,\dots,b_n$ of $L/\mathbb C$. Put $F=\mathbb C(a,b_2,\dots,b_n)$. The field $L$ is algebraic and finitely generated over $F$, hence finite, and
separable since the characteristic is $0$. The derivation $\partial/\partial a$ of $F$ with
$\partial/\partial a\,(b_i)=0$ for all $i$ therefore extends to a derivation $D$ of $L$ with $D(a)=1\neq0$.
For the second statement, any $D\in\operatorname{Der}_{\mathbb C}(L)$ is an $L$-combination of the
$\widetilde{\partial/\partial y_j}$ by Lemma~\ref{lem:der-span}.
\end{proof}

\begin{corollary}[The group algebra as the full centralizer]\label{cor:iso-CW}
\[
  \mathbb{C}[W] \;\cong\; \operatorname{Hom}_{\widetilde{\mathcal{D}}_Y}
  \big(\widetilde{\mathcal{O}}_X, \widetilde{\mathcal{O}}_X\big).
\]
\end{corollary}

\begin{proof}
By Proposition~\ref{prop:embedding-DY}, the natural map is injective;
we prove surjectivity. Write $L = \operatorname{Frac}(\widetilde{\mathcal O}_X)$
and $K = \operatorname{Frac}(\widetilde{\mathcal O}_Y)$.

\smallskip\noindent\emph{Step 1: reduction to $L$.} Let
$\phi \in \operatorname{Hom}_{\widetilde{\mathcal D}_Y}(\widetilde{\mathcal O}_X,\widetilde{\mathcal O}_X)$.
Since $\phi$ commutes with multiplication by $y_1,\dots,y_n$ and $\delta^{-1}$, it is
$\widetilde{\mathcal O}_Y$-linear. Every element of $L$ can be written $f/g$ with $f \in \widetilde{\mathcal O}_X$ and
$g \in \widetilde{\mathcal O}_Y \setminus \{0\}$: for $h\in\widetilde{\mathcal O}_X\setminus\{0\}$,
$f/h=f\prod_{w\ne1}w(h)\big/N(h)$ with $N(h)=\prod_ww(h)\in\widetilde{\mathcal O}_Y$. So $\phi$ extends uniquely to a
$K$-linear endomorphism of $L$, still denoted $\phi$, via
$\phi(f/g) := \phi(f)/g$ (this is well defined because $\phi$ is $\widetilde{\mathcal O}_Y$-linear). By the quotient rule,
$\phi$ continues to commute with the extended derivations
$\widetilde D_1,\dots,\widetilde D_n$ of $L$
(Lemma~\ref{lem:der-span}).

\smallskip\noindent\emph{Step 2: $\operatorname{End}_K(L)$ as an
$L[W]$-module.} By Dedekind's lemma on the linear independence of
characters \cite[Chap.\ VI]{Lang2002}, the elements of $W$, viewed as
$K$-linear endomorphisms of $L$, are linearly independent over $L$
(where $L$ acts by left multiplication,
$a \cdot \Phi := (x \mapsto a\,\Phi(x))$ for $a \in L$). Since
\[
  \dim_K \operatorname{End}_K(L) = [L:K]^2 = |W|^2 = |W| \cdot \dim_K L,
\]
this linear independence forces, by a dimension count,
\[
  \operatorname{End}_K(L) = \bigoplus_{w \in W} L \cdot w,
\]
i.e.\ every $\Phi \in \operatorname{End}_K(L)$ is \emph{uniquely} of
the form $\Phi = \sum_{w \in W} a_w\, w$ for functions $a_w \in L$.

\smallskip\noindent\emph{Step 3: commuting with the derivations forces constant
coefficients.} Write $\phi = \sum_w a_w\, w$ as in Step 2, and fix
$j \in \{1,\dots,n\}$. Since $\phi$ commutes with $\widetilde D_j$,
and each $w \in W$ already commutes with $\widetilde D_j$
(Proposition~\ref{prop:embedding-DY}), the Leibniz rule gives, for
every $x \in L$,
\[
  0 = \widetilde D_j(\phi(x)) - \phi(\widetilde D_j x)
    = \sum_{w \in W} \widetilde D_j(a_w)\, w(x).
\]
As this holds for all $x \in L$, Dedekind's lemma forces
$\widetilde D_j(a_w) = 0$ for every $j$ and every $w \in W$. By
Lemma~\ref{lem:der-span}, the derivations
$\widetilde D_1,\dots,\widetilde D_n$ span $\operatorname{Der}_{\mathbb{C}}(L)$
over $L$, so each $a_w$ is annihilated by \emph{every}
$\mathbb{C}$-derivation of $L$; by Lemma~\ref{lem:constants}, $a_w \in \mathbb{C}$.

Hence $\phi = \sum_{w \in W} a_w\, w$ with all $a_w \in \mathbb{C}$,
exhibiting $\phi$ as the restriction to $\widetilde{\mathcal O}_X$ of
an element of $\mathbb{C}[W]$ (the extension of Step 1 being unique). This proves surjectivity, and the
corollary follows.
\end{proof}

\subsection{Descent over the \'etale cover and the simple components}\label{subsec:galois-descent}

Throughout this subsection $A=\widetilde{\mathcal O}_Y$,
$B=\widetilde{\mathcal O}_X$, $K=\operatorname{Frac}(A)=\mathbb C(y_1,\dots,y_n)$,
$L=\operatorname{Frac}(B)=\mathbb C(x_1,\dots,x_n)$ and
$\mathcal D=\widetilde{\mathcal D}_Y$. By Lemma~\ref{lem:jacobian-general}
and Proposition~\ref{prop:embedding-DY}, $B$ is a $\mathcal D$-module on
which $W$ acts by $\mathcal D$-linear automorphisms.

For a finite-dimensional representation $V$ of $W$ let $W$ act on
$B\otimes_{\mathbb C}V$ diagonally, let $\mathcal D$ act through
$\mathcal D\otimes\mathrm{id}$ (which commutes with $W$), and put
\[
  \nabla V:=\big(B\otimes_{\mathbb C}V\big)^{W},
\]
a $\mathcal D$-submodule of $B\otimes V$. At the level of the fields
$L/K$, the equivalence of categories below is due to
Nonkan\'e \cite[Prop.~2.4]{Nonkane2019} (see also
\cite[Prop.~4.4]{NonkaneLawson2022} and
\cite[Prop.~3.10]{NonkaneTodjihounde2023}), we may call it Nonkan\'e's equivalence. We prove it at the
level of the rings $A\subset B$, which is what the decomposition of $B$
requires, using that $A\to B$ is a finite \'etale Galois cover
(Proposition~\ref{prop:etale}).

Recall two classical facts on $L/K$ (here $W=\mathrm{Gal}(L/K)$):
\begin{itemize}
  \item[(F1)] $[L:K]=|W|$ (Artin's theorem, see \cite{Artin1942} or \cite[Chap.\ VI]{Lang2002});
  \item[(F2)] (Normal Basis Theorem, \cite[Chap.\ VI]{Lang2002}) $L\cong K[W]$ as $K[W]$-modules.
\end{itemize}

\begin{lemma}[Descent isomorphism]\label{lem:descent-iso}
For every finite-dimensional $\mathbb{C}$-vector space $V$ with a
$W$-action, with $W$ acting diagonally on $L\otimes_{\mathbb C}V$, the multiplication map
\[
  \mu_V : L \otimes_K \big(L \otimes_{\mathbb{C}} V\big)^{W} \longrightarrow L \otimes_{\mathbb{C}} V,
  \qquad a \otimes m \longmapsto am
\]
is an isomorphism of $L$-vector spaces.
\end{lemma}

\begin{proof}
\emph{Injectivity.} Suppose $\ker\mu_V\ne0$, and choose a nonzero element
$\sum_{i=1}^r a_i \otimes m_i$ of it, with $m_1,\dots,m_r \in (L\otimes V)^W$ linearly independent over $K$
and $r$ minimal. Then all $a_i \ne 0$, and
after scaling we may assume $a_1 = 1$. For $\sigma \in W$, applying
$\sigma$ to $\sum a_i m_i = 0$ and using $\sigma(m_i) = m_i$ gives
$\sum \sigma(a_i) m_i = 0$; subtracting the original relation gives
$\sum_{i \ge 2} (\sigma(a_i) - a_i) m_i = 0$ (the $i=1$ term vanishes since $\sigma(a_1)-a_1=1-1=0$).
So $\sum_{i \ge 2} (\sigma(a_i) - a_i) \otimes m_i$ lies in $\ker\mu_V$ and is shorter than $r$;
by minimality it is $0$, and since the $m_i$ are $K$-linearly independent, $\sigma(a_i) = a_i$ for all $i,\sigma$, i.e.\ all
$a_i \in L^W = K$. But then $\sum a_i m_i = 0$ is a $K$-linear
relation among the $K$-linearly independent $m_i$, forcing all
$a_i = 0$, which contradicts $a_1 = 1 \ne 0$. Hence $\mu_V$ is
injective.

\emph{Dimension count.} By (F2), $L \cong K[W]$ as $K[W]$-modules
(here $W$ acts $K$-linearly on $L$ via its Galois action, since $W$ fixes $K$, matching the left
regular representation on $K[W]$). Under this identification, the map
$g \otimes v \mapsto g \otimes g^{-1}v$ gives a $K[W]$-module
isomorphism $K[W] \otimes_{\mathbb{C}} V \xrightarrow{\sim} K[W] \otimes_{\mathbb{C}} V$
from the diagonal action to the action on the first factor only (for
$h \in W$ the diagonal action sends
$g \otimes v$ to $hg \otimes hv$, which the map sends to
$hg \otimes g^{-1}v$, the image of $g\otimes g^{-1}v$ under $h$ acting on the first factor only).
Hence, as $K[W]$-modules,
$L \otimes_{\mathbb{C}} V \cong K[W]^{\oplus \dim_{\mathbb{C}} V}$ with
$W$ acting only on the $K[W]$ factors, so
\[
  \dim_K (L \otimes_{\mathbb{C}} V)^W
  = \dim_{\mathbb{C}} V \cdot \dim_K \big(K[W]\big)^W
  = \dim_{\mathbb{C}} V,
\]
since $(K[W])^W$ (invariants of the regular representation) is
$1$-dimensional over $K$, spanned by $\sum_{g \in W} g$.

Therefore
$\dim_K\big(L \otimes_K (L\otimes V)^W\big) = [L:K]\cdot\dim_{\mathbb{C}}V
= |W|\dim_{\mathbb{C}}V$ by (F1), which equals
$\dim_K(L\otimes_{\mathbb{C}}V) = [L:K]\dim_{\mathbb{C}}V$ as well.
An injective $K$-linear map between $K$-vector spaces of the same
finite dimension is an isomorphism; since $\mu_V$ is
$L$-linear and injective, it is bijective, hence
an isomorphism of $L$-vector spaces.
\end{proof}

\begin{lemma}[Rigidity of flat endomorphisms]\label{lem:rigidity}
Let $V$ be a finite-dimensional $\mathbb{C}$-vector space and equip
$L \otimes_{\mathbb{C}} V$ with the connection
$\widetilde{D} \otimes \operatorname{id}$ for $\widetilde{D}$ ranging
over the derivations of $L$ that extend a derivation of $K$. Then any $L$-linear
endomorphism of $L \otimes_{\mathbb{C}} V$ commuting with
$\widetilde{D} \otimes \operatorname{id}$ for all such $\widetilde{D}$
lies in $\operatorname{id}_L \otimes \operatorname{End}_{\mathbb{C}}(V)$.
\end{lemma}

\begin{proof}
Fix a $\mathbb{C}$-basis $v_1,\dots,v_m$ of $V$ and write an
$L$-linear endomorphism as a matrix $(f_{ij}) \in \operatorname{Mat}_m(L)$
acting by $v_j \mapsto \sum_i f_{ij} v_i$ (extended $L$-linearly), so
that $\Phi(a \otimes v_j) = \sum_i a f_{ij} \otimes v_i$ for $a \in L$.
Fix a derivation $\widetilde D$ of $L$ extending one of $K$. Computing
$\Phi \circ (\widetilde D \otimes \operatorname{id})$ and
$(\widetilde D \otimes \operatorname{id}) \circ \Phi$ on $a \otimes v_j$
and comparing, using the Leibniz rule
$\widetilde D(af_{ij}) = \widetilde D(a) f_{ij} + a\, \widetilde D(f_{ij})$:
\[
  (\widetilde D \otimes \operatorname{id})\big(\Phi(a \otimes v_j)\big)
  = \sum_i \big[\widetilde D(a) f_{ij} + a\, \widetilde D(f_{ij})\big] \otimes v_i,
  \qquad
  \Phi\big((\widetilde D \otimes \operatorname{id})(a \otimes v_j)\big)
  = \sum_i \widetilde D(a)\, f_{ij} \otimes v_i.
\]
Commutativity of $\Phi$ with $\widetilde D \otimes \operatorname{id}$
means these two expressions agree for every $a \in L$; subtracting
leaves $\sum_i a\, \widetilde D(f_{ij}) \otimes v_i = 0$ for every
$a \in L$, and since the $v_i$ are linearly independent, taking
$a = 1$ already forces $\widetilde D(f_{ij}) = 0$ for every $i,j$.
In particular $\widetilde D_j(f_{ij})=0$ for the extensions $\widetilde D_j$ of $\partial/\partial y_j$,
which span $\operatorname{Der}_{\mathbb{C}}(L)$ over $L$ (Lemma~\ref{lem:der-span}). So every
$\mathbb C$-derivation of $L$ kills $f_{ij}$, and $f_{ij}\in\mathbb C$ by
Lemma~\ref{lem:constants}. Thus $\Phi$ lies in $\operatorname{id}_L\otimes\operatorname{End}_{\mathbb C}(V)$.
\end{proof}

\begin{lemma}[Flat subspaces]\label{lem:flat-subspaces}
Let $V$ be a finite-dimensional $\mathbb C$-vector space and
$N\subset L\otimes_{\mathbb C}V$ an $L$-subspace stable under
$\widetilde{\partial/\partial y_j}\otimes\mathrm{id}$, $j=1,\dots,n$. Then
$N=L\otimes U$ with $U=\{v\in V\mid1\otimes v\in N\}$.
\end{lemma}
\begin{proof}
First suppose $N\ne0$. Fix a basis $v_1,\dots,v_m$ of $V$; every element of $L\otimes V$ is uniquely of the form
$\sum_ia_i\otimes v_i$ with $a_i\in L$, and we count its nonzero coefficients $a_i$. Choose
$0\ne w=\sum_ia_i\otimes v_i\in N$ with the least number of nonzero
coefficients, and (after renumbering and scaling by $a_1^{-1}$, $N$ being an $L$-subspace) assume that $a_1=1$. For each $j$,
$\widetilde D_j(w)=\sum_{i\ge2}\widetilde D_j(a_i)\otimes v_i\in N$ has
fewer nonzero coefficients than $w$, hence is $0$ by minimality; so
$\widetilde D_j(a_i)=0$ for all $i,j$, and $a_i\in\mathbb C$ by
Lemma~\ref{lem:constants}. Thus $w=1\otimes\sum_ia_iv_i$, and every nonzero such $N$ contains a nonzero
vector of $1\otimes V$. In general, $L\otimes U\subset N$; if the inclusion
were strict, $N/(L\otimes U)\subset L\otimes(V/U)$ would be a nonzero stable
subspace, hence would contain $1\otimes\bar v$ with $\bar v\ne0$; lifting,
$1\otimes v\in N+L\otimes U=N$ for some $v\notin U$, contradicting the
definition of $U$.
\end{proof}

\begin{lemma}\label{lem:free-rank}
$B\cong A\otimes_{\mathbb C}\mathbb C[W]$ as $A[W]$-modules. Consequently,
for every finite-dimensional representation $V$ of $W$, $\nabla V$ is a free
$A$-module of rank $\dim V$, and $\nabla$ is an exact functor.
\end{lemma}
\begin{proof}
Let $H\subset\mathcal O_X$ be the span of the higher Specht polynomials. It
is $W$-stable, being the sum of the submodules $V^l_S(\lambda)$ (Theorem~\ref{thm:2.3}(i)), and
$\mathcal O_X=\mathcal O_Y\otimes H$ by Theorem~\ref{thm:2.3}(iii), with $W$ acting on $H$ only because
$W$ fixes $\mathcal O_Y$. Since the higher Specht polynomials form a basis of $\mathcal O_X$ over $\mathcal O_Y$, their classes form a basis of
$\Lambda_0$ (Theorem~\ref{thm:2.3}(ii)), so $H\to\Lambda_0$ is a $W$-equivariant isomorphism and $H$ is the
regular representation. Localizing, $B=A\otimes_{\mathbb C}H$. Hence $(B\otimes V)^W=A\otimes(H\otimes V)^W$ ($W$ acts trivially on
$A$, so the averaging idempotent acts as $\mathrm{id}\otimes e$), and
$(\mathbb C[W]\otimes V)^W\cong V$. Exactness holds because taking
$W$-invariants is exact in characteristic $0$.
\end{proof}

\begin{lemma}[Descent isomorphism over the rings]\label{lem:descent-rings}
For every finite-dimensional representation $V$ of $W$, the map
\[
  \mu_V:B\otimes_A\nabla V\longrightarrow B\otimes_{\mathbb C}V,\qquad b\otimes m\mapsto bm,
\]
is an isomorphism of $B$-modules. It is $W$-equivariant ($W$ acting on
$B\otimes_A\nabla V$ through $B$, and diagonally on $B\otimes V$) and compatible
with $\partial/\partial y_j$ (acting by the Leibniz rule on the left and by
$\partial/\partial y_j\otimes\mathrm{id}$ on the right).
\end{lemma}
\begin{proof}
\emph{Surjectivity.} By Proposition~\ref{prop:etale}(c),
$B\otimes_AB\cong\prod_{w\in W}B$; choose
$\sum_ia_i\otimes b_i\in B\otimes_AB$ mapping to $(\delta_{w,1})_w$, i.e.\
$\sum_ia_i\,w(b_i)=\delta_{w,1}$ for all $w\in W$. For $m\in B\otimes V$ put
$t_i=\sum_{w}w(b_im)\in\nabla V$. Then
$\sum_ia_it_i=\sum_w\big(\sum_ia_iw(b_i)\big)w(m)=m$.
\emph{Injectivity.} $\nabla V$ is free (Lemma~\ref{lem:free-rank}), so
$B\otimes_A\nabla V$ is torsion-free over $B$. After tensoring with $L$ over $B$,
$\mu_V$ becomes the map $L\otimes_K(K\otimes_A\nabla V)\to L\otimes V$, where
$K\otimes_A\nabla V=(L\otimes V)^W$: invariants commute with the localization $A\to K$, and
$K\otimes_AB=L$ because $B$ is finite over $A$ (so $K\otimes_AB$ is a domain, finite over $K$, hence a field). This map is an
isomorphism by Lemma~\ref{lem:descent-iso}; so the
kernel of $\mu_V$ is torsion, hence zero.

\emph{Compatibilities.} For $w\in W$ and $m\in\nabla V$,
$\mu_V(w(b)\otimes m)=w(b)m=w(bm)$, since $m$ is invariant. For the Leibniz rule,
$\mu_V(\partial b\otimes m+b\otimes\partial m)=\partial b\cdot m+b\cdot\partial m=\partial(bm)$, where $\partial=\partial/\partial y_j$
acts on $B\otimes V$ as $\partial\otimes\mathrm{id}$.
\end{proof}

We also use the following fact (see \cite[\S1.4]{HTT2008}): \emph{a
$\mathcal D$-module which is finitely generated over $A$ is a finitely generated
projective $A$-module}, because $\operatorname{Spec}A$ is a smooth affine
variety in characteristic~$0$.

\begin{theorem}[Galois descent over the rings]\label{thm:descent}
\begin{enumerate}[label=(\roman*)]
\item The functor $\nabla:\operatorname{Rep}(W)\to\mathcal D\text{-}\mathrm{Mod}$ is exact and
  fully faithful: for representations $V,V'$ the map $\psi\mapsto\nabla\psi$
  is an isomorphism $\operatorname{Hom}_W(V,V')\xrightarrow{\ \sim\ }\operatorname{Hom}_{\mathcal D}(\nabla V,\nabla V')$.
\item A $\mathcal D$-module $M$, finitely generated over $A$, is isomorphic
  to $\nabla U$ for a finite-dimensional representation $U$ of $W$ if and only
  if $B\otimes_AM$ is \emph{trivial}, i.e.\ isomorphic, as a $B$-module with the
  connection $\partial/\partial y_j$, to $B\otimes_{\mathbb C}U'$ with the
  connection $\partial/\partial y_j\otimes\mathrm{id}$, for some
  finite-dimensional vector space $U'$. In that case $U$ is the space of
  horizontal sections of $B\otimes_AM$, with its natural $W$-action, and
  $U\cong U'$ as vector spaces.
\end{enumerate}
\end{theorem}
\begin{proof}
(i) Exactness is Lemma~\ref{lem:free-rank}. Let $\psi\in\operatorname{Hom}_W(V,V')$;
then $\nabla\psi=(\mathrm{id}_B\otimes\psi)|_{\nabla V}$ and it is injective in
$\psi$: if $\nabla\psi=0$ then $\mathrm{id}_B\otimes\psi$ vanishes on
$B\cdot\nabla V=B\otimes V$ (Lemma~\ref{lem:descent-rings}), so $\psi=0$.
Conversely let $\varphi:\nabla V\to\nabla V'$ be $\mathcal D$-linear. Then
$\Phi:=\mu_{V'}\circ(\mathrm{id}_B\otimes\varphi)\circ\mu_V^{-1}$ is a
$B$-linear map $B\otimes V\to B\otimes V'$ commuting with $\partial/\partial y_j\otimes\mathrm{id}$
and with $W$ (diagonal action). Extend $\Phi$ $L$-linearly to $L\otimes V\to L\otimes V'$, and regard it as an endomorphism of $L\otimes(V\oplus V')$
(by composing with the projection onto $V$ and the inclusion of $V'$). It commutes with the extensions $\widetilde D_j$ of
$\partial/\partial y_j$, hence with every derivation of $L$ extending a derivation of $K$, because $\Phi$ is $L$-linear.
Lemma~\ref{lem:rigidity} (applied to $V\oplus V'$) gives $\Phi=\mathrm{id}\otimes\psi$ for some $\psi\in\operatorname{Hom}_{\mathbb C}(V,V')$, and
$W$-equivariance of $\Phi$ gives $\psi(wv)=w\psi(v)$. Thus $\varphi=\nabla\psi$.

(ii) If $M=\nabla U$, then $B\otimes_AM\cong B\otimes U$ by
Lemma~\ref{lem:descent-rings}, which is trivial. Conversely let
$\Theta:B\otimes_AM\to B\otimes U'$ be an isomorphism of $B$-modules with
connection. By Lemma~\ref{lem:constants}, the constants of $B$ are $\mathbb C$,
so the horizontal sections of $B\otimes U'$ form $1\otimes U'$; let $U$ be the
space of horizontal sections of $B\otimes_AM$. Then $U\cong U'$ and
$\iota:B\otimes_{\mathbb C}U\to B\otimes_AM$, $b\otimes u\mapsto bu$, is an
isomorphism ($=\Theta^{-1}$ on the nose). $W$ acts on $B\otimes_AM$ through
$B$ and commutes with the connection, so it acts on $U$, and $\iota$ is
$W$-equivariant for the diagonal action on $B\otimes U$. Now $B=A\oplus(1-e)B$ as $A$-modules, where $e=|W|^{-1}\sum_ww$ (proof of Proposition~\ref{prop:etale}(d));
the image of $e\otimes\mathrm{id}$ on $B\otimes_AM$ is $A\otimes_AM=M$, so $(B\otimes_AM)^W=M$. Taking $W$-invariants of $\iota$ gives
$M\cong(B\otimes U)^W=\nabla U$ as $\mathcal D$-modules.
\end{proof}

\begin{proposition}[Simplicity]\label{prop:simple}
If $V$ is an irreducible representation of $W$, then $\nabla V$ is a simple
$\mathcal D$-module, of $A$-rank $\dim V$. If $V,V'$ are irreducible, then
$\nabla V\cong\nabla V'$ if and only if $V\cong V'$.
\end{proposition}
\begin{proof}
Let $0\ne N\subset\nabla V$ be a $\mathcal D$-submodule. Since $A$ is
noetherian, $N$ and $\nabla V/N$ are $\mathcal D$-modules finitely generated
over $A$, hence projective (the fact recalled before Theorem~\ref{thm:descent}); in particular $\nabla V/N$ is torsion-free.
Since $B$ is flat over $A$ (Proposition~\ref{prop:etale}(c)) and by Lemma~\ref{lem:descent-rings},
$B\otimes_AN\subset B\otimes_A\nabla V\cong B\otimes V$
is a $B$-submodule stable under the connection and under $W$ (diagonal
action). Let $N_L\subset L\otimes V$ be the $L$-span of its image; it is
nonzero, stable under the $\widetilde{\partial/\partial y_j}\otimes\mathrm{id}$
and under $W$. By Lemma~\ref{lem:flat-subspaces}, $N_L=L\otimes U$ with
$U\subset V$ a subspace; since $w(1\otimes u)=1\otimes wu$ and $N_L$ is
$W$-stable, $U$ is a nonzero $W$-submodule of $V$, so $U=V$ and
$N_L=L\otimes V$. Now $N_L=L\otimes_AN$, so $\operatorname{rk}_AN=\dim_LN_L=\dim V=\operatorname{rk}_A\nabla V$, so
$\nabla V/N$ is torsion; being torsion-free, it is $0$. The isomorphism
statement follows from Theorem~\ref{thm:descent}(i), since
$\operatorname{Hom}_W(V,V')=0$ for non-isomorphic irreducibles.
\end{proof}

\begin{theorem}[Structure of $\widetilde{\mathcal O}_X$ over $\widetilde{\mathcal D}_Y$]\label{thm:structure}
For an irreducible representation $V$ of $W$, let $B_V$ be its isotypic
component in $B$ and $M_V=\operatorname{Hom}_W(V,B)$ (a $\mathcal D$-module via
$(D\varphi)(v)=D(\varphi(v))$).
\begin{enumerate}[label=(\alph*)]
\item $M_V\cong\nabla(V^*)$, and evaluation gives an isomorphism of
  $W\times\mathcal D$-modules $V\otimes_{\mathbb C}M_V\xrightarrow{\sim}B_V$;
  $B=\bigoplus_VB_V$.
\item Each $M_V$ is a simple $\mathcal D$-module of $A$-rank $\dim V$; the $M_V$
  are pairwise non-isomorphic. Hence
  $B\cong\bigoplus_V M_V^{\oplus\dim V}$, a semisimple $\mathcal D$-module of
  length $\sum_V\dim V$.
\item If $e\in\mathbb C[W]$ is a primitive idempotent in the block of $V$, then
  $eB\cong M_V\cong\nabla(V^*)$ is a simple $\mathcal D$-submodule of $B$.
\item $\operatorname{End}_{\mathcal D}(B)\cong\mathbb C[W]$ (this recovers Corollary~\ref{cor:iso-CW}).
\end{enumerate}
\end{theorem}
\begin{proof}
(a) The map $\operatorname{Hom}_W(V,B)\to(B\otimes V^*)^W$,
$\varphi\mapsto\sum_i\varphi(v_i)\otimes v_i^*$ is an isomorphism, compatible
with $\mathcal D$. The evaluation map $V\otimes M_V\to B$ is $W$-equivariant
and $\mathcal D$-linear, with image the isotypic component $B_V$, and it is
injective by the standard argument for isotypic components.
(b) Proposition~\ref{prop:simple} applied to $V^*$, and the rank is
$\dim V^*=\dim V$ by Lemma~\ref{lem:free-rank}. The decomposition follows
from (a). (c) Since $e$ lies in the block of $V$, it kills $B_{V'}$ for $V'\not\cong V$, so $eB=eB_V=e(V\otimes M_V)=(eV)\otimes M_V$
($W$ acts on the first factor only), and $\dim eV=1$.
(d) By (b), $\operatorname{End}_{\mathcal D}(B)=\bigoplus_V\operatorname{End}_{\mathbb C}(V)\otimes\operatorname{End}_{\mathcal D}(M_V)$
with $\operatorname{End}_{\mathcal D}(M_V)=\mathbb C$ (Theorem~\ref{thm:descent}(i)),
and the cross terms vanish since $\operatorname{Hom}_{\mathcal D}(M_V,M_{V'})=\operatorname{Hom}_W(V^*,V'^*)=0$ for $V\not\cong V'$
(Theorem~\ref{thm:descent}(i) and Schur's lemma); this is $\bigoplus_V\operatorname{End}_{\mathbb C}(V)=\mathbb C[W]$.
\end{proof}

\begin{remark}[What is classical in Theorem~\ref{thm:structure}]\label{rem:classical}
Theorem~\ref{thm:structure} is the algebraic form of a standard fact. Let $U=\operatorname{Spec}\widetilde{\mathcal O}_X$, the
complement in $\mathbb C^n$ of the reflecting hyperplanes; then $U\to U/W$ is a finite \'etale Galois cover with deck group $W$.
We write $\pi_1(U/W)$ for the fundamental group $\pi_1\big((U/W)(\mathbb C),x_0\big)$ of the complex manifold of $\mathbb C$-points of $U/W$, with
the classical topology and a base point $x_0$ (its isomorphism class does not depend on $x_0$). The covering
$U(\mathbb C)\to(U/W)(\mathbb C)$ is connected with deck group $W$, so it defines a surjective homomorphism $\pi_1(U/W)\to W$, whose kernel
is the fundamental group of $U(\mathbb C)$. Hence a representation $V$ of $W$ gives a representation $\pi_1(U/W)\to W\to GL(V)$, that is,
a flat bundle on $U/W$, the flat bundle attached to $V$. The direct image of $\mathcal O_U$ is the sum of the flat bundles attached to
the irreducible representations $V$ of $W$, each with multiplicity $\dim V$. Each is simple because the representation
$\pi_1(U/W)\to GL(V)$ is irreducible, having the same image as $W\to GL(V)$; these flat bundles are regular, so simplicity as a
$\mathcal D$-module is equivalent to irreducibility of the monodromy, by the Riemann--Hilbert correspondence. The proof given here works at the level of rings and uses no regularity
or analytic input. Likewise the Jacobian formula of Lemma~\ref{lem:jacobian-general} is, for every reflection group,
a theorem of Steinberg~\cite{Steinberg1960} (Subsection~\ref{subsec:notation}), and the semisimplicity of $B$ also follows from the
decomposition theorem~\cite{deCataldoMigliorini2009}. What the paper adds is the generators given by higher Specht
polynomials, with their index sets and multiplicities for $p>1$
(Subsection~\ref{subsec:idempotents}), the invariant localization, and the
characterization of the summands of $B$.
\end{remark}

\begin{remark}[On the double-centralizer argument]\label{rem:double-centralizer}
One might try to deduce simplicity from the isomorphism
$\mathbb C[W]\cong\operatorname{End}_{\mathcal D}(B)$ by a double-centralizer argument.
This needs a hypothesis: if $M$ is a \emph{semisimple} module of finite length
over a ring $R$ and the semisimple algebra $\mathcal A$ maps isomorphically
onto $\operatorname{End}_R(M)$, then $eM$ is simple for every primitive
idempotent $e$ of $\mathcal A$ (indeed $eM$ is semisimple with
$\operatorname{End}_R(eM)=e\mathcal Ae=\mathbb C$, since $e\mathcal Ae$ is a finite-dimensional division algebra over $\mathbb C$; so $eM$ is
indecomposable, hence simple); without
semisimplicity $\operatorname{End}_R(eM)=\mathbb C$ does not imply simplicity.
In our setting the semisimplicity of $B$ over $\mathcal D$ follows from the
decomposition theorem for the direct image under a finite map
\cite{deCataldoMigliorini2009} (see \cite[Prop.~2.8]{Nonkane2019}, \cite[Prop.~3.2]{NonkaneLawson2022} for its use in this
setting); the proof given above does not use it.
\end{remark}

\begin{corollary}[Twisted-invariants description of the simple summands]\label{cor:twisted-invariants}
Let $\mathcal O$ be an orbit, $\lambda\in\mathcal O$ and $0\le l<e(\lambda)$, and let $V^l(\mathcal O)$
be the irreducible $W$-module of Corollary~\ref{cor:mult}. Every
primitive idempotent $e$ of $\mathbb C[W]$ in the block attached to $V^l(\mathcal O)$ satisfies, as
$\mathcal D$-modules,
\[
  eB\;\cong\;\nabla\big(V^l(\mathcal O)^*\big)=\big(B\otimes V^l(\mathcal O)^*\big)^W
  \;\cong\;\operatorname{Hom}_W\big(V^l(\mathcal O),B\big).
\]
\end{corollary}
\begin{proof}
Theorem~\ref{thm:structure}(a) and (c).
\end{proof}

\begin{example}[The functor $\nabla$ on $D_2=G(2,2,2)$]\label{ex:nabla-D2}
Let $W=D_2=\{\mathrm{id},\sigma,\tau,\sigma\tau\}$ as in
Subsection~\ref{subsec:D2-example} ($\sigma$ the swap, $\tau$ the
simultaneous sign change). For a character $\chi:D_2\to\{\pm1\}$ and
$V_\chi=\mathbb Cv_\chi$ (here $V_\chi^*\cong V_\chi$),
\[
  \nabla(\chi)=\{a\in B\mid g(a)=\chi(g)\,a\ \ \forall g\in D_2\}
\]
is the module of \emph{relative invariants} of character $\chi$, free of rank
$1$ over $A$ (Lemma~\ref{lem:free-rank}). The four characters are realized by
\[
  \begin{array}{c|c|c}
  \chi & (\chi(\sigma),\chi(\tau)) & \text{generator } F_\chi \\\hline
  \mathrm{triv} & (+1,+1) & 1 \\
  \mathrm{sign} & (-1,+1) & x_1^2 - x_2^2 \\
  \chi^0 & (+1,-1) & x_1+x_2 \\
  \chi^1 & (-1,-1) & x_1-x_2
  \end{array}
\]
(for instance $\sigma(x_1^2-x_2^2)=-(x_1^2-x_2^2)$ and $\tau(x_1^2-x_2^2)=x_1^2-x_2^2$).
Each $F_\chi$ is one of the higher Specht polynomials $F,F',F^0,F^1$ of
Subsection~\ref{subsec:D2-example}; they form an $A$-basis of $B$ and lie in
different isotypic components, so $\nabla(\chi)=A\cdot F_\chi$.

\emph{Full faithfulness in action.} $\operatorname{Hom}_{\mathcal D}(\nabla(\mathrm{triv}),\nabla(\mathrm{sign}))=0$:
a $\mathcal D$-linear map $\varphi:A\cdot1\to A\cdot F$, $F=x_1^2-x_2^2$, is
determined by $\varphi(1)=cF$ with $c\in A$, and $\partial_{y_j}(cF)=\varphi(\partial_{y_j}1)=0$ for
all $j$; by Lemma~\ref{lem:constants}, $cF=\kappa$ is a constant. If $\kappa\ne0$ then
$F=\kappa c^{-1}\in K$, which is false since $\sigma(F)=-F$; hence $\kappa=0$ and $\varphi=0$.
\end{example}

\begin{example}[A complex example: $G(3,1,1)$]\label{ex:G311}
Let $W=\mu_3=G(3,1,1)$ ($n=1$, $d=3$, $\Delta=x^2$, $\delta=x^3$) act on $B=\mathbb C[x,x^{-1}]$ by $(g\cdot f)(x)=f(\xi x)$,
$A=\mathbb C[t,t^{-1}]$ with $t=x^3$ ($\delta=x^3=t$), $\mathcal D=\mathbb C[t^{\pm1},\partial_t]$,
and $\chi_k(g)=\xi^k$. The $\chi_k$-isotypic component is $B_{\chi_k}=A\,x^k$, and
$M_{\chi_k}=\operatorname{Hom}_W(\chi_k,B)=A\,x^k\cong\nabla(\chi_{-k})$: it is
generated by $x^k$ with $t\partial_t\,x^k=\tfrac k3\,x^k$. Thus
\[
  B=A\oplus A\,x\oplus A\,x^2
\]
is a sum of three simple $\mathcal D$-modules of rank one with exponents
$0,\tfrac13,\tfrac23$: indeed $A\,x^k\cong\mathcal D/\mathcal D(t\partial_t-\tfrac k3)$ has rank one, hence is simple, and
$\mathcal D/\mathcal D(t\partial_t-\alpha)\cong\mathcal D/\mathcal D(t\partial_t-\alpha')$ if and only if $\alpha-\alpha'\in\mathbb Z$, so the
three modules are pairwise non-isomorphic, their exponents differing modulo $\mathbb Z$. The example shows that the dual in
Corollary~\ref{cor:twisted-invariants} matters for non-real groups: $A\,x$
is $\nabla(\chi_{-1})$, not $\nabla(\chi_1)=A\,x^2$.
\end{example}

\begin{corollary}[Categorical characterization of the summands of $\widetilde{\mathcal O}_X$]\label{cor:characterization}
Let $M$ be a $\mathcal D$-module finitely generated over $A$.
\begin{enumerate}[label=(\alph*)]
\item $M\cong\nabla U$ for some representation $U$ of $W$ if and only if
  $B\otimes_AM$ is trivial in the sense of Theorem~\ref{thm:descent}(ii).
\item $M$ is isomorphic to a direct summand of $B$ if and only if
  $M\cong\nabla U$ for a representation $U$ isomorphic to a subrepresentation of
  $\mathbb C[W]$; equivalently, each irreducible $V$ occurs in $U$ with
  multiplicity $m_V\le\dim V$.
\end{enumerate}
In particular $B\otimes_AM$ being trivial does not by itself make $M$ a summand
of $B$: $A^{\oplus(|W|+1)}=\nabla\big(\mathrm{triv}^{\oplus(|W|+1)}\big)$ has trivial pullback but rank $|W|+1>|W|$
(the trivial representation would occur $|W|+1>1=\dim\mathrm{triv}$ times).
\end{corollary}
\begin{proof}
(a) is Theorem~\ref{thm:descent}(ii). (b) The map $b\mapsto\sum_{g\in W}g(b)\otimes g$ is an
isomorphism of $\mathcal D$-modules $B\to\nabla(\mathbb C[W])$ (an element
$\sum_gb_g\otimes g$ is invariant iff $b_{wg}=w(b_g)$, i.e.\ $b_g=g(b_1)$). If $M$ is a
direct summand of $B$, it is the image of an idempotent endomorphism of
$B=\nabla(\mathbb C[W])$, which is $\nabla\psi$ for some
$\psi\in\operatorname{End}_W(\mathbb C[W])$ by Theorem~\ref{thm:descent}(i), and $\psi$ is idempotent because $\nabla$ is faithful
($\nabla(\psi^2)=(\nabla\psi)^2=\nabla\psi$); since $\nabla$ is exact,
$M\cong\nabla(\operatorname{Im}\psi)$. Conversely every summand $U$ of $\mathbb C[W]$ gives a summand
$\nabla U$ of $B$. A representation embeds in $\mathbb C[W]$ iff
$m_V\le\dim V$ for all $V$.
\end{proof}

\begin{remark}[Attribution]
At the level of fields, the equivalence $\nabla$ is due to
\cite{Nonkane2019}; for $G(r,1,n)$ the corresponding statements, over the ring used there (the invariants localized at $\Delta^2$),
appear in \cite[\S4]{NonkaneLawson2022}; that localization has to be replaced by the one at $\delta$ when $r\ge3$ (Lemma~\ref{lem:delta}). The ring-level statement
(Theorem~\ref{thm:descent}) is an instance of Galois descent for a finite \'etale cover
(Remark~\ref{rem:classical}); we give a direct proof so that the paper is self-contained.
To our knowledge, the characterization of the summands of $B$ with its multiplicity bound
(Corollary~\ref{cor:characterization}) is not in the literature in this form.
\end{remark}

\subsection{Idempotents and Specht generators}\label{subsec:idempotents}

This subsection shows that suitable primitive idempotents $e$ of $\mathbb C[W]$ give simple submodules $eB$
generated by a single higher Specht polynomial,  and that $B$ is the direct sum of these submodules.

By Maschke's theorem \cite[Chap.\ XVIII]{Lang2002}, $\mathbb C[W]$ is
semisimple, and by Corollary~\ref{cor:mult} its irreducible representations
are the $V^l(\mathcal O)$, for the orbits $\mathcal O$ and
$0\le l<e(\lambda)$ ($\lambda\in\mathcal O$), of dimension
$f_\lambda/e(\lambda)$. Accordingly
\[
 \mathbb C[W]=\bigoplus_{(\mathcal O,l)}R^l_{\mathcal O},\qquad
 R^l_{\mathcal O}\cong\operatorname{End}_{\mathbb C}\big(V^l(\mathcal O)\big),
\]
where $R^l_{\mathcal O}$ is the \emph{block attached to} $V^l(\mathcal O)$, with primitive central idempotent
\[
 r^l_{\mathcal O}=\frac{\dim V^l(\mathcal O)}{|W|}\sum_{g\in W}\chi^l_{\mathcal O}(g^{-1})\,g,
\]
and $\chi^l_{\mathcal O}$ is the character of $V^l(\mathcal O)$.

\begin{example}[Block decomposition of {$\mathbb{C}[G(2,1,2)]$}]\label{ex:CW-decomp}

Let $W = G(2,1,2)$ (the dihedral group of order $8$; $r=2$, $p=1$,
$d=2$). Since $p=1$, every orbit is a singleton $\mathcal O=\{\lambda\}$ and $e(\lambda)=1$, so $l=0$ and
$R_\lambda:=R^0_{\{\lambda\}}$ (no further splitting by $l$). The set
$\mathcal{P}_{2,2}$ has five elements: four singleton shapes (each
giving $f_\lambda=1$) and one ``pair'' shape $\lambda_0$ (one box in each component, giving
$f_{\lambda_0}=2$); these are treated, with generating
polynomials, in Subsection~\ref{subsec:Gr1n}, where the four singleton shapes carry
$F_0=1$, $F_1=x_1x_2$, $F'_0=x_1^2-x_2^2$ and $F'_1=x_1x_2(x_1^2-x_2^2)$. The
decomposition above reads
\[
  \mathbb{C}[W] = \underbrace{\mathbb{C}}_{\chi_{F_0}} \oplus
    \underbrace{\mathbb{C}}_{\chi_{F_1}} \oplus
    \underbrace{\mathbb{C}}_{\chi_{F'_0}} \oplus
    \underbrace{\mathbb{C}}_{\chi_{F'_1}} \oplus
    \underbrace{\mathrm{Mat}_2(\mathbb{C})}_{R_{\lambda_0}},
\]
of total dimension $1+1+1+1+4 = 8 = |W|$: the four singleton shapes
give four $1\times 1$ blocks $R_\lambda \cong \mathbb{C}$, the four
linear characters of $W$ ($\chi_{F_0}$ is trivial and $\chi_{F'_1}$ is the determinant; if $t_i$ changes the sign of $x_i$ and $s$ swaps $x_1,x_2$,
then $\chi_{F_1}$ takes the values $-1,+1$ on $t_i,s$ and $\chi_{F'_0}$ the values $+1,-1$), while $\lambda_0$ gives the single
$2\times 2$ block $R_{\lambda_0} \cong \operatorname{Mat}_2(\mathbb{C})$
corresponding to the unique $2$-dimensional irreducible representation
of $W$. The primitive central idempotent of each $1\times 1$ block is
the character-averaging projector
$r_\lambda = \frac{1}{8}\sum_{g \in W} \chi_\lambda(g)\, g$
(the characters are real-valued); the block
$R_{\lambda_0}$ contains two primitive (non-central) adapted idempotents
$e_{T_a}, e_{T_b}$, with sum the central idempotent $r_V$ of the block, computed in
Example~\ref{ex:prop29} (they are $r_V\frac{1+t_1}2$ and $r_V\frac{1+t_2}2$).
\end{example}

\medskip\noindent\textbf{Adapted idempotents.} Fix an orbit $\mathcal O$ and $l$, and a pair
$(\lambda_0,S_0)\in I(\mathcal O)=\bigsqcup_{\mu\in\mathcal O}\mathrm{STab}(\mu)_d$.
The polynomials $F^l_{S_0,T}$, $T\in\mathrm{STab}(\lambda_0)_{db(\lambda_0)}$, form
a basis of $V^l_{S_0}(\lambda_0)\cong V^l(\mathcal O)$; let $\rho$ be the
corresponding matrix representation, so that $g\,F^l_{S_0,T'}=\sum_T\rho(g)_{TT'}F^l_{S_0,T}$. For
$T,T'\in\mathrm{STab}(\lambda_0)_{db(\lambda_0)}$ put
\[
  e^l_{\mathcal O,TT'}=\frac{f}{|W|}\sum_{g\in W}\rho(g^{-1})_{T'T}\,g\ \in R^l_{\mathcal O},
  \qquad f=\dim V^l(\mathcal O),
\]
the matrix units of the block ($e_{TT'}e_{T''T'''}=\delta_{T'T''}e_{TT'''}$ and
$\sum_Te_{TT}=r^l_{\mathcal O}$; this follows from Schur's orthogonality relations).
We write $e_T:=e^l_{\mathcal O,TT}$ and call these the \emph{adapted
primitive idempotents}: $e_T$ acts on $V^l(\mathcal O)$ as the projection onto the line $\mathbb CF^l_{S_0,T}$, hence is primitive. For $\mu=\mathrm{Sh}^j\lambda_0\in\mathcal O$ and
$T\in\mathrm{STab}(\lambda_0)_{db(\lambda_0)}$ let $\tau_\mu(T)$ be the unique
element of $\{\mathrm{Sh}^{j+mb}T\mid 0\le m<e\}$ lying in
$\mathrm{STab}(\mu)_{db(\mu)}$ (Lemma~\ref{lem:count}(b)). This set is the $\langle\mathrm{Sh}^b\rangle$-orbit of $\mathrm{Sh}^jT$, and
$j$ is determined by $\mu$ modulo $b=b(\lambda_0)$, so $\tau_\mu(T)$ does not depend on the choice of $j$.

\begin{lemma}[Adapted idempotents act diagonally]\label{lem:adapted}
Let $(\mu,S)\in I(\mathcal O)$ and $T'\in\mathrm{STab}(\mu)_{db(\mu)}$. Then
$e_TF^l_{S,T'}=F^l_{S,T'}$ if $T'=\tau_\mu(T)$ and $e_TF^l_{S,T'}=0$ otherwise.
Moreover $e_TF=0$ for every higher Specht polynomial $F$ attached to a pair
$(\mathcal O',l')\ne(\mathcal O,l)$.
\end{lemma}
\begin{proof}
Let $\mu=\mathrm{Sh}^j\lambda_0$ and $S'=\mathrm{Sh}^jS_0$. By Lemma~\ref{lem:transport}, $\Phi_j$ is an isomorphism of $W$-modules
$V^l_{S_0}(\lambda_0)\to V^l_{S'}(\mu)$ with $\Phi_jF^l_{S_0,T}=F^l_{S',\mathrm{Sh}^jT}$, and by
Theorem~\ref{thm:ATY}(ii) the change of $S'$ into $S$, $F_{S',T''}\mapsto F_{S,T''}$, is an isomorphism
$V^l_{S'}(\mu)\to V^l_S(\mu)$ of $G(r,n)$-modules, hence of $W$-modules. Since
$F^l_{S,\mathrm{Sh}^bT''}=\xi^{-ldb}F^l_{S,T''}$, the polynomials of $\mathrm{Sh}^b$-equivalent tableaux are proportional, so the composite
$\Psi:V^l_{S_0}(\lambda_0)\to V^l_S(\mu)$ sends $F^l_{S_0,T''}$ to a nonzero scalar multiple of $F^l_{S,\tau_\mu(T'')}$. Moreover
$T''\mapsto\tau_\mu(T'')$ is a bijection $\mathrm{STab}(\lambda_0)_{db}\to\mathrm{STab}(\mu)_{db}$, as shown in the proof of
Lemma~\ref{lem:transport}. Now $e_T\in\mathbb C[W]$ acts on $V^l_{S_0}(\lambda_0)$ as the projection onto
$\mathbb CF^l_{S_0,T}$ along the lines $\mathbb CF^l_{S_0,T''}$, $T''\ne T$. Since $\Psi$ is $W$-equivariant, $e_T$ acts on $V^l_S(\mu)$ as
the projection onto $\mathbb CF^l_{S,\tau_\mu(T)}$ along the lines $\mathbb CF^l_{S,\tau_\mu(T'')}$, $T''\ne T$. This is the first assertion.
For $(\mathcal O',l')\ne(\mathcal O,l)$, the module $V^{l'}_{S'}(\mu')\cong V^{l'}(\mathcal O')$ is not isomorphic to $V^l(\mathcal O)$
(Corollary~\ref{cor:mult}), so it lies in a different isotypic component, which $e_T\in R^l_{\mathcal O}$ kills.
\end{proof}

\begin{theorem}\label{thm:main-general}
Let $(\mathcal O,l)$ be an orbit with $0\le l<e(\lambda)$, $\lambda\in\mathcal O$, and, with the
choices $(\lambda_0,S_0)$ and $\rho$ made above, let
$T\in\mathrm{STab}(\lambda_0)_{db(\lambda_0)}$, with adapted idempotent $e=e_T$.
\begin{enumerate}[label=(\roman*)]
\item $eB$ is a nonzero $\mathcal D$-submodule of $B$;
\item $eB$ is a simple $\mathcal D$-module, isomorphic to $\nabla\big(V^l(\mathcal O)^*\big)$, of $A$-rank $\dim V^l(\mathcal O)$;
\item $eB=\bigoplus_{(\mu,S)\in I(\mathcal O)}A\,F^l_{S,\tau_\mu(T)}$, and for every
  $(\mu,S)\in I(\mathcal O)$,
  \[
    eB=\mathcal D\,F^l_{S,\tau_\mu(T)} .
  \]
\end{enumerate}
For an arbitrary primitive idempotent $e'$ of $\mathbb C[W]$, $e'B$ is still simple (Theorem~\ref{thm:structure}(c)) and equals $\mathcal D(e'F)$ for any $F\in B$ with $e'F\ne0$, but it is in general \emph{not} generated by a higher Specht polynomial (Remark~\ref{rem:arbitrary-idempotent}).
\end{theorem}
\begin{proof}
(i) $\mathcal D$ commutes with $W$ (Proposition~\ref{prop:embedding-DY}), hence with $e$, so $eB$ is a $\mathcal D$-submodule; it is nonzero because $F^l_{S_0,T}\in eB$ by Lemma~\ref{lem:adapted} for $(\mu,S,T')=(\lambda_0,S_0,T)$, since $\tau_{\lambda_0}(T)=T$.
(ii) Theorem~\ref{thm:structure}(c) and Corollary~\ref{cor:twisted-invariants}.
(iii) By Theorem~\ref{thm:2.3}(iii), $B=\bigoplus A\,F$ over the family of higher Specht polynomials, and $e$ is $A$-linear. By Lemma~\ref{lem:adapted}, $eF\in\{0,F\}$ for each member, and $eF=F$ exactly for the $F^l_{S,\tau_\mu(T)}$, $(\mu,S)\in I(\mathcal O)$. Hence $eB$ is the $A$-span of these polynomials, a direct sum since they belong to a basis of $B$, and they are $|I(\mathcal O)|=\dim V^l(\mathcal O)$ in number (Corollary~\ref{cor:mult}). Any nonzero element of the simple module $eB$ generates it.
\end{proof}

\begin{remark}\label{rem:arbitrary-idempotent}
A primitive idempotent that is not adapted can fail to give a Specht generator. For $W=G(2,1,2)$ let $s$ swap $x_1,x_2$ and $e'=r_V\frac{1+s}2$, where $r_V$ is the central idempotent of the $2$-dimensional block. Since $s$ has the eigenvalues $+1$ and $-1$ on the
$2$-dimensional representation, $e'$ is primitive. It is not adapted: it projects $\langle x_1,x_2\rangle$ onto the line spanned by
$x_1+x_2$ and $\langle x_1^2x_2,x_1x_2^2\rangle$ onto the line spanned by $x_1^2x_2+x_1x_2^2$, and these are not higher Specht
polynomials. Then $e'x_1=e'x_2=\frac12(x_1+x_2)$ and $e'(x_1^2x_2)=\frac12(x_1^2x_2+x_1x_2^2)$, while $e'$ kills the polynomials of other shapes. Since $e'$ is idempotent, $F\in e'B$ forces $e'F=F$, so no higher Specht polynomial lies in $e'B$. Thus $e'B$ is simple but is not of the form $\mathcal DF^l_{S,T}$.
\end{remark}

\begin{corollary}\label{cor:iso-simple}
With the above notation, for fixed $(\mathcal O,l)$ and $T,T'\in\mathrm{STab}(\lambda_0)_{db(\lambda_0)}$,
$e_TB\cong e_{T'}B$ as $\mathcal D$-modules. If $e$ is an adapted idempotent for $(\mathcal O,l)$ and $e'$ one for
$(\mathcal O',l')\ne(\mathcal O,l)$, then $eB\not\cong e'B$.
\end{corollary}
\begin{proof}
Both are isomorphic to $\nabla(V^l(\mathcal O)^*)$ (Theorem~\ref{thm:main-general}(ii)). For distinct pairs the irreducible representations $V^l(\mathcal O)$ and
$V^{l'}(\mathcal O')$ are non-isomorphic (Corollary~\ref{cor:mult}), hence so are their duals, and Proposition~\ref{prop:simple} gives
$\nabla(V^l(\mathcal O)^*)\not\cong\nabla(V^{l'}(\mathcal O')^*)$.
\end{proof}

\begin{proposition}\label{prop:decomposition-general}
Write $f^l_{\mathcal O}=\dim V^l(\mathcal O)$ (also denoted $f^l_\lambda$ for $\lambda\in\mathcal O$, and likewise
$V^l(\lambda):=V^l(\mathcal O)$; a sum over $\lambda\in\mathcal P_{r,n}/{\sim}$ is a sum over orbits, $\lambda$ being a representative) and
$\mathrm{STab}(\lambda_0)_{db(\lambda_0)}=\{T_1,\dots,T_{f^l_{\mathcal O}}\}$, where $\lambda_0=\lambda_0(\mathcal O)$ is the representative fixed for
the orbit $\mathcal O$. For each
$(\mathcal O,l)$ fix $(\mu_{\mathcal O},S_{\mathcal O})\in I(\mathcal O)$ and put
$F^l_{\lambda,i}:=F^l_{S_{\mathcal O},\tau_{\mu_{\mathcal O}}(T_i)}$ for $\lambda\in\mathcal O$ (it depends only on the orbit of $\lambda$). Then:
\begin{enumerate}[label=(\roman*)]
\item (internal direct sum)
  \[
   \widetilde{\mathcal O}_X=\bigoplus_{\mathcal O}\ \bigoplus_{l=0}^{e(\lambda)-1}\ \bigoplus_{i=1}^{f^l_{\mathcal O}}\widetilde{\mathcal D}_YF^l_{\lambda,i},
  \]
  with $\widetilde{\mathcal D}_YF^l_{\lambda,i}=e_{T_i}\widetilde{\mathcal O}_X$ simple;
\item (isomorphism types) the modules $\widetilde{\mathcal D}_YF^l_{\lambda,i}$, $1\le i\le f^l_{\mathcal O}$, are isomorphic to each other, and
  \[
   \widetilde{\mathcal O}_X\;\cong\;\bigoplus_{\mathcal O}\ \bigoplus_{l=0}^{e(\lambda)-1}\big(\widetilde{\mathcal D}_YF^l_{\lambda,1}\big)^{\oplus f^l_{\mathcal O}},
  \]
  the summands $\widetilde{\mathcal D}_YF^l_{\lambda,1}$ being pairwise non-isomorphic for distinct $(\mathcal O,l)$.
\end{enumerate}
In particular the total $A$-rank is $\sum(f^l_{\mathcal O})^2=|W|$, since each $(\mathcal O,l)$ contributes $f^l_{\mathcal O}$ simple
summands of rank $f^l_{\mathcal O}$.
\end{proposition}
\begin{proof}
(i) For a fixed block, the $e_{T_i}=e_{T_iT_i}$ are matrix units, hence orthogonal idempotents with $\sum_ie_{T_i}=r^l_{\mathcal O}$; the central idempotents
$r^l_{\mathcal O}$ of different blocks are orthogonal with sum $1$. So the $e_{T_i}$ (over all $(\mathcal O,l)$ and $i$) are orthogonal
idempotents with $\sum e_{T_i}=1$, and $B=\bigoplus e_{T_i}B$. By Theorem~\ref{thm:main-general}(iii), $e_{T_i}B=\mathcal D\,F^l_{\lambda,i}$,
which is simple by Theorem~\ref{thm:main-general}(ii). (ii) Corollary~\ref{cor:iso-simple}.
\end{proof}

\begin{example}[Proposition~\ref{prop:decomposition-general} for $B_2=G(2,1,2)$]\label{ex:prop29}
Here $r=2$, $p=1$, $d=2$, $|W|=8$, $e(\lambda)=1$ throughout, $\mathcal O=\{\lambda\}$ for each of the five shapes
and $I(\mathcal O)=\mathrm{STab}(\lambda)$. Four shapes have $f_\lambda=1$; the shape
$\lambda_0$ (one box in each component) has $f_{\lambda_0}=2$, with standard
tableaux $T_a=(0{:}\,1,\ 1{:}\,2)$ and $T_b=(0{:}\,2,\ 1{:}\,1)$ and the four higher
Specht polynomials $x_2,x_1$ (for $S=T_a$) and $x_1^2x_2,\,x_1x_2^2$ (for $S=T_b$), as
in Subsection~\ref{subsec:Gr1n}. Let $t_i$ change the sign of $x_i$ and $r_V$ be the central idempotent of the
$2$-dimensional block. The adapted idempotents, for the basis $(F_{T_a,T_a},F_{T_a,T_b})=(x_2,x_1)$ of $V_{T_a}$, are
$e_{T_a}=r_V\frac{1+t_1}2$ (projection onto $\mathbb Cx_2$ along $\mathbb Cx_1$) and $e_{T_b}=r_V\frac{1+t_2}2$. By
Theorem~\ref{thm:main-general}(iii), $e_{T_a}B=A\,F_{T_a,T_a}\oplus A\,F_{T_b,T_a}$ with $F_{T_a,T_a}=x_2$ and
$F_{T_b,T_a}=x_1^2x_2$, and similarly for $e_{T_b}$, so
\[
  e_{T_a}B=A\,x_2\oplus A\,x_1^2x_2=\mathcal D\,x_2,\qquad
  e_{T_b}B=A\,x_1\oplus A\,x_1x_2^2=\mathcal D\,x_1 .
\]
Therefore
\[
  \widetilde{\mathcal O}_X=\mathcal D\cdot1\ \oplus\ \mathcal D\,x_1x_2\ \oplus\ \mathcal D\,(x_1^2-x_2^2)\ \oplus\ \mathcal D\,x_1x_2(x_1^2-x_2^2)\ \oplus\ \mathcal D\,x_2\ \oplus\ \mathcal D\,x_1,
\]
a direct sum of six simple modules of ranks $1,1,1,1,2,2$ (total $8=|W|$), with
$\mathcal D\,x_2\cong\mathcal D\,x_1$ \emph{isomorphic but distinct} submodules
(the reflection $t_2$ commutes with $\mathcal D$, fixes $x_1$ and not $x_2$, so $x_2\notin\mathcal D\,x_1$).
Up to isomorphism, $\widetilde{\mathcal O}_X\cong\mathcal D\cdot1\oplus\mathcal D\,x_1x_2\oplus\mathcal D\,(x_1^2-x_2^2)\oplus\mathcal D\,x_1x_2(x_1^2-x_2^2)\oplus(\mathcal D\,x_2)^{\oplus2}$.
With $y_1=x_1^2+x_2^2$, $y_2=x_1^2x_2^2$ one has
$\partial_{y_1}(x_2)=-x_2/\big(2(x_1^2-x_2^2)\big)$ and
\[
  x_1^2x_2=\tfrac{y_1}2\,x_2-(y_1^2-4y_2)\,\partial_{y_1}(x_2)\ \in\ \mathcal D\,x_2,
\]
which is why $\mathcal D\,x_1^2x_2$ must not be listed as a separate summand.
\end{example}

\begin{remark}
Taking $r=p=2$ recovers the decomposition theorem for $W(D_n)$,
treated in Section~\ref{sec:Dn} below. Taking $p=r$
(so $d=1$) gives the decomposition for the index-$r$ subgroup
$G(r,r,n)$; taking $p=1$ (so $d=r$) gives it for the full monomial
group $G(r,1,n)$, where $b=e=1$, so there is no splitting by $l$, and $I(\mathcal O)=\mathrm{STab}(\lambda)$.
Section~\ref{sec:other-cases} spells out $G(r,r,n)$, $W(B_n)=G(2,1,n)$ and $G(r,1,2)$.
\end{remark}

\section{The real case: $W(D_n) = G(2,2,n)$}\label{sec:Dn}

We specialize Section~\ref{sec:general} to $r=p=2$. The group $W(D_n)$ was
first treated in~\cite{NonkaneLawson2021} by an argument specific to $D_n$. This section states the
decomposition for $W(D_n)$ (Corollaries~\ref{cor:Dn-setup} and~\ref{cor:Dn-main}) and works out $D_2$ and $D_3$.
Let $\epsilon_1,\dots,\epsilon_n$ be the standard basis of the real vector space $\mathbb{R}^n$, on which $W(D_n)$ acts; it acts on $\mathbb C^n$ by extension of
scalars, so that $x_1,\dots,x_n$ are the coordinates on $\mathbb C^n$. The
reflections correspond to the $n(n-1)$ roots $\epsilon_i\pm\epsilon_j$ ($1\le i<j\le n$),
one reflection for each sign choice per pair; this is the count of
Section~\ref{sec:general} (non-diagonal reflections only, since $d=1$).

\begin{corollary}[Jacobian polynomial and invariants]\label{cor:Dn-setup}
For $W = W(D_n) = G(2,2,n)$ (so $r=2$, $p=2$, $d=1$), the fundamental
invariants (they generate $\mathbb C[x]^{W}$) are
\[
  y_j = \sum_{i=1}^n x_i^{2j}, \quad 1 \le j \le n-1, \qquad y_n = x_1\cdots x_n,
\]
and by Lemma~\ref{lem:jacobian-general}, with the normalization $\Delta:=\det\mathcal J$,
\[
  \Delta = \det\mathcal J = (-1)^{n-1}\, 2^{\,n-1}(n-1)! \prod_{1 \le i < j \le n} (x_j^2 - x_i^2).
\]
In particular $\det\mathcal J$ is invertible in $\widetilde{\mathcal{O}}_X$, being a nonzero constant multiple of the
product of the linear forms of the reflections.
\end{corollary}

\begin{proof}
Set $r=2$, $p=2$, $d=1$ in Lemma~\ref{lem:jacobian-general}. The factor
$\prod_ix_i^{d-1}$ equals $1$, and the prefactor becomes $(-1)^{n-1}2^{n-1}(n-1)!$. The
remaining product $\prod_{i<j}(x_j^2-x_i^2)=\prod_{i<j}(x_j-x_i)(x_j+x_i)$ is the product of the
linear forms of the roots $\epsilon_i\pm\epsilon_j$. (For $n=2$ this gives
$\Delta=-2(x_2^2-x_1^2)=2(x_1^2-x_2^2)$, as in Subsection~\ref{subsec:D2-example}.)
\end{proof}

Specializing Theorem~\ref{thm:main-general} and
Proposition~\ref{prop:decomposition-general} to $r=p=2$ gives the decomposition theorem for
$W(D_n)$:

\begin{corollary}[Decomposition theorem for $W(D_n)$]\label{cor:Dn-main}
For every adapted primitive idempotent $e \in \mathbb{C}[W(D_n)]$
(Subsection~\ref{subsec:idempotents}; for an arbitrary primitive idempotent,
(i) and (ii) still hold):
\begin{enumerate}[label=(\roman*)]
  \item $e\widetilde{\mathcal{O}}_X$ is a nonzero
    $\widetilde{\mathcal{D}}_Y$-submodule of $\widetilde{\mathcal{O}}_X$;
  \item the $\widetilde{\mathcal{D}}_Y$-module $e\widetilde{\mathcal{O}}_X$
    is simple;
      \item write $e=e_T$, where $T\in\mathrm{STab}(\lambda_0)_{db(\lambda_0)}$, $\lambda_0$ is the representative fixed for the orbit
    $\mathcal O\in\mathcal{P}_{2,n}/\sim$ to which $e$ is attached, and $l$ ($0\le l<e(\lambda_0)$) is its second index; then
    $e\widetilde{\mathcal{O}}_X = \widetilde{\mathcal{D}}_Y F^l_{S,\tau_\mu(T)}$ for every $(\mu,S)\in I(\mathcal O)$;
    in particular $e\widetilde{\mathcal{O}}_X$ is generated by a higher Specht polynomial.
\end{enumerate}
Moreover,
\[
  \widetilde{\mathcal{O}}_X \cong \bigoplus_{\lambda \in \mathcal{P}_{2,n}/\sim}
    \bigoplus_{l=0}^{e(\lambda)-1}
    \big(\widetilde{\mathcal{D}}_Y F^l_{\lambda,1}\big)^{\oplus f^l_\lambda},
    \qquad f^l_\lambda = \dim_{\mathbb{C}}\big(V^l(\lambda)\big),
\]
with the notation of Proposition~\ref{prop:decomposition-general}.
\end{corollary}

\begin{proof}
This is Theorem~\ref{thm:main-general} and Proposition~\ref{prop:decomposition-general} for $r=p=2$, with the data of Corollary~\ref{cor:Dn-setup}.
\end{proof}

\subsection{Example: $D_2 = G(2,2,2)$}\label{subsec:D2-example}

Take $n=2$. Since $|D_2|=4$, $D_2 \cong \mathbb{Z}/2\times\mathbb{Z}/2$ is
abelian, generated by the transposition
$\sigma:(x_1,x_2)\mapsto(x_2,x_1)$ and the simultaneous sign change
$\tau:(x_1,x_2)\mapsto(-x_1,-x_2)$; all four irreducible
representations are one-dimensional.

Invariants: $y_1=x_1^2+x_2^2$, $y_2=x_1x_2$ (both of degree $2$), and
$\Delta = 2(x_1^2-x_2^2)$, matching Corollary~\ref{cor:Dn-setup} with $n=2$.

Here $d=1$, so $\mathrm{Sh}$ swaps the two components of a shape. The set $\mathcal P_{2,2}$ has five elements and three orbits: two
orbits of size $2$ ($b=2$, $e=1$), those of the shapes $(2)$ and $(1,1)$ placed in one component, and the fixed shape
with one box in each component ($b=1$, $e=2$), which splits into two characters $l=0,1$. This gives the four one-dimensional
representations of $D_2$. Applying the shift construction of Subsection~\ref{subsubsec:shift} to $G(2,2,2)$
yields the following basis of $\Lambda_0 \cong \mathbb{C}[D_2]$, with
standard tableaux as indicated (component $0$ only for $F,F'$; both
components for $F^0,F^1$, as in Subsection~\ref{subsec:Gr1n}):
\[
 F = 1 \ \big[\begin{array}{|c|c|}\hline 1&2\\\hline\end{array}\big],
 \qquad
 F' = x_1^2-x_2^2 \ \big[\begin{array}{|c|}\hline 1\\\hline 2\\\hline\end{array}\big],
\]
\[
 F^0 = x_1+x_2 \ \Big[\big(0:\begin{array}{|c|}\hline 1\\\hline\end{array},\
   1:\begin{array}{|c|}\hline 2\\\hline\end{array}\big)\Big],
 \qquad
 F^1 = x_1-x_2 \ \Big[\big(0:\begin{array}{|c|}\hline 2\\\hline\end{array},\
   1:\begin{array}{|c|}\hline 1\\\hline\end{array}\big)\Big],
\]
of respective degrees $0,2,1,1$, and respective characters
$(\sigma,\tau)\mapsto(+1,+1)$, $(-1,+1)$, $(+1,-1)$, $(-1,-1)$, the
four distinct characters of $\mathbb{Z}/2\times\mathbb{Z}/2$.

By Chevalley's theorem the Hilbert series of $\Lambda_0 = \mathcal{O}_X/J_+$ for $D_2$ is
\[
 \frac{(1-t^2)^2}{(1-t)^2} = (1+t)^2 = 1+2t+t^2,
\]
so $\dim_{\mathbb{C}} \Lambda_0 = 1+2+1 = 4 = |D_2|$, matching the
basis above (degree $0$: $F$; degree $1$: $F^0,F^1$; degree $2$: $F'$).

\begin{remark}
$D_2$ has index $2$ in $B_2 = G(2,1,2)$ ($|D_2|=4$, $|B_2|=8$, see
Subsection~\ref{subsec:Gr1n} below). The polynomial $F'=x_1^2-x_2^2$ transforms by the sign character (the
determinant) of $D_2$, and is, up to a scalar, the Jacobian polynomial $\Delta$.
\end{remark}

\subsection{Example: $D_3 = G(2,2,3)$}\label{subsec:D3-example}

Take $n=3$, so that $W(D_3) \cong \mathfrak{S}_4$,
$|D_3|=24$, with five irreducible representations of dimensions
$1,1,2,3,3$ ($1+1+4+9+9=24$).

Invariants: $y_1=\sum x_i^2$ (degree~$2$), $y_2=\sum x_i^4$
(degree~$4$), $y_3=x_1x_2x_3$ (degree~$3$); the product of degrees is
$2\cdot4\cdot3=24=|D_3|$.

The five equivalence classes $\lambda \in \mathcal{P}_{2,3}/\sim$ (each with
$e(\lambda)=1$, since $n=3$ is odd and no shape is shift-invariant)
correspond to the five irreducibles:

\begin{center}
\begin{tabular}{lcl}
shape $\lambda^0$ (untwisted component) & dim & role \\
\hline
$(3)$ & $1$ & trivial \\
$(1,1,1)$ & $1$ & sign \\
$(2,1)$ & $2$ & standard rep.\ of $\mathfrak{S}_3 \subset \mathfrak{S}_4$ \\
$(2)$ with a box in component $1$ & $3$ & standard rep.\ of $\mathfrak{S}_4$ \\
$(1,1)$ with a box in component $1$ & $3$ & standard $\otimes$ sign \\
\end{tabular}
\end{center}

\paragraph{The two one-dimensional pieces.}
Trivial: $F=1$ (degree $0$). Sign ($\lambda^0=(1,1,1)$): for $S=T$ the column tableau with $1,2,3$ from top to
bottom, the word of $S$ is $3\,2\,1$, so $i(S)=(0,1,2)$ and the base monomial is $x_2^{2}x_3^{4}$. The full
antisymmetrizer applied to it produces a Vandermonde determinant in the $x_i^2$:
\[
 F' \;\propto\; \det\big(x_i^{2(j-1)}\big)_{1\le i,j\le3}
 = \prod_{i<j}(x_j^2-x_i^2) \;\propto\; \Delta \qquad (\deg 6).
\]
The same computation for every $n$ shows that the higher Specht polynomial of the shape $((1^n),\emptyset)$ is, up to scalar, the
Vandermonde determinant $\det(x_i^{2(j-1)})=\prod_{i<j}(x_j^2-x_i^2)$, which is a constant multiple of $\Delta$ by
Lemma~\ref{lem:jacobian-general}. As for $D_2$, the sign representation of $D_n$ is thus given, up to scalar, by $\Delta$ itself.

\paragraph{The dimension-2 piece, shape $(2,1)$.}
Let $T_A,T_B$ be the two standard tableaux of shape $(2,1)$ (hook
lengths $3,1,1$, so $\dim=3!/3=2$):
\[
  T_A = \begin{array}{|c|c|}\hline 1&2\\\hline 3\\\cline{1-1}\end{array},
  \qquad
  T_B = \begin{array}{|c|c|}\hline 1&3\\\hline 2\\\cline{1-1}\end{array}.
\]
Take $S=T_A$. Then $i(S)=(0,0,1)$ on the cells (row-major), so only the cell of index $1$, the lower cell, contributes to the
base monomial; it holds the number $3$ for $T=T_A$, giving $x_3^{2}$, and the number $2$ for $T=T_B$, giving $x_2^{2}$. Applying the Young
symmetrizers:
\[
 G_1 = e_{T_A}(x_3^2) \;\propto\; x_1^2 - x_3^2,
 \qquad
 G_2 = e_{T_B}(x_2^2) \;\propto\; x_1^2 - x_2^2.
\]
The other choice $S=T_B$, with $i(S)=(0,1,1)$, gives two further polynomials of degree $4$,
$x_2^2(x_1^2-x_3^2)$ and $x_3^2(x_1^2-x_2^2)$, which span a second copy of the same representation (the block has multiplicity $2$).
Since $\lambda^1=\emptyset$, $G_1$ and $G_2$ are polynomials in the $x_i^2$, hence even in each variable, so the sign changes of $D_n$ act trivially
on them. They are the classical Specht polynomials of the standard representation of $\mathfrak{S}_3$ in the variables $x_i^2$, and in general a
representation carried entirely by the untwisted component reduces to ordinary Specht theory for $\mathfrak{S}_n$ in the variables $x_i^2$.

\paragraph{The two three-dimensional pieces.} These shapes mix both
components, so they are specific to $D_n$. We write down their lowest-degree higher
Specht polynomials directly.

\emph{Standard representation.} No invariant has degree $1$ (the
$y_j$ have degrees $2,3,4$), so the degree-$1$ piece of $\Lambda_0$
coincides with the degree-$1$ piece of $\mathcal{O}_X$ itself, namely
$\mathrm{span}(x_1,x_2,x_3)$: this is the
reflection representation of $D_3$ on its defining coordinates, irreducible of dimension $n=3$.
These are higher Specht polynomials of the shape $((2),(1))$: for $S_0=\big(0{:}\,\begin{array}{|c|c|}\hline1&2\\\hline\end{array},\ 1{:}\,\begin{array}{|c|}\hline3\\\hline\end{array}\big)$
one has $i(S_0)\equiv0$, and $F_{S_0,T}=x_k$, where $k$ is the number that $T$ places in component $1$. Put
\[
  H_1 := x_1 \qquad (\deg 1);
\]
(its character is $\chi_{\mathrm{std}}(\mathrm{id},T,Q,C,F) = (3,1,-1,0,-1)$,
computed directly as the trace of each generator on
$\mathrm{span}(x_1,x_2,x_3)$).

\emph{Standard $\otimes$ sign.} The shape $((1,1),(1))$ has $f_\lambda=3$ standard tableaux, and
its lowest-degree higher Specht polynomials (those for $S=\big(0{:}\,\begin{array}{|c|}\hline1\\\hline3\\\hline\end{array},\ 1{:}\,\begin{array}{|c|}\hline2\\\hline\end{array}\big)$, for which $i(S)=(0,0,1)$) are, up to scalars,
\[
  H_1':=x_1(x_2^2-x_3^2),\qquad H_2':=x_2(x_3^2-x_1^2),\qquad H_3':=x_3(x_1^2-x_2^2)
  \qquad(\deg 3).
\]
With the representatives $T,Q,C,F$ of the next subsection, $T$ sends $H_1',H_2',H_3'$ to
$-H_2',-H_1',-H_3'$; $Q$ sends them to $-H_1',-H_2',H_3'$; $C$ permutes them cyclically; and $F$
sends them to $-H_2',H_1',H_3'$. The traces are $-1,-1,0,1$, so
$\chi_{\mathrm{std}\otimes\mathrm{sgn}}=\chi_{\mathrm{std}}\cdot\chi_{\mathrm{sgn}}=(3,-1,-1,0,1)$ on
$(\mathrm{id},T,Q,C,F)$. The gradient $\partial P/\partial x_1$ of $P=\prod_{i<j}(x_j^2-x_i^2)$ (so $\Delta=8P$ by Corollary~\ref{cor:Dn-setup}) lies in the same isotypic
component, in degree $5$, and differs from $-6x_1^3(x_2^2-x_3^2)$ by the multiple $2y_1H_1'$.

\emph{Why these generate simple summands.} In each case the $W$-span $U$ of the generator ($x_1$,
$H_1'$, or $G_1$ above) is an irreducible submodule of $B$ isomorphic to the corresponding $V$. By
Theorem~\ref{thm:structure}(a), $B_V=V\otimes M_V$, so $U=V\otimes m$ for some $m\in M_V$, and every
nonzero vector of $U$ has the form $v\otimes m'$; it generates $v\otimes M_V\cong M_V$, which is simple.

\subsection{Central primitive idempotents for $D_2$ and $D_3$}\label{subsec:D23-idem}

For a finite group $W$ and an irreducible character $\chi_\lambda$ of
dimension $\dim\lambda$, the central primitive idempotent projecting
onto the $\lambda$-isotypic component is
\[
  r_\lambda = \frac{\dim\lambda}{|W|} \sum_{g \in W} \overline{\chi_\lambda(g)}\, g.
\]
We compute them for the two examples above.

\paragraph{$D_2 = \{\mathrm{id},\sigma,\tau,\sigma\tau\}$.}
Being abelian of order $4$, all characters are $1$-dimensional, and
the idempotents are read off directly from the character table \cite[Lecture~2]{FultonHarris1991}:
\[
  r_{\mathrm{triv}} = \tfrac14(\mathrm{id}+\sigma+\tau+\sigma\tau)
  \quad [F=1],
  \qquad
  r_{\mathrm{sign}} = \tfrac14(\mathrm{id}-\sigma+\tau-\sigma\tau)
  \quad [F'=x_1^2-x_2^2],
\]
\[
  r^{0} = \tfrac14(\mathrm{id}+\sigma-\tau-\sigma\tau)
  \quad [F^0=x_1+x_2],
  \qquad
  r^{1} = \tfrac14(\mathrm{id}-\sigma-\tau+\sigma\tau)
  \quad [F^1=x_1-x_2].
\]
One checks $r_{\mathrm{triv}}+r_{\mathrm{sign}}+r^0+r^1=\mathrm{id}$,
since the coefficients of $\sigma,\tau,\sigma\tau$ cancel pairwise.

\paragraph{$D_3 \cong \mathfrak{S}_4$.}
Denote by $T,Q,C,F$ (by abuse of notation, also standing for the sum
of the elements in each class) the conjugacy classes of $D_3$, with
representatives
\[
  \mathrm{id}\ (|{\cdot}|{=}1),\quad
  T: x_1\leftrightarrow x_2\ (|{\cdot}|{=}6),\quad
  Q: (x_1,x_2)\mapsto(-x_1,-x_2)\ (|{\cdot}|{=}3),
\]
\[
  C: (x_1,x_2,x_3)\mapsto(x_2,x_3,x_1)\ (|{\cdot}|{=}8),
  \qquad
  F: (x_1,x_2,x_3)\mapsto(x_2,-x_1,-x_3)\ (|{\cdot}|{=}6),
\]
corresponding respectively to the $\mathfrak{S}_4$-cycle types
$1^4,\,2\,1^2,\,2^2,\,3\,1,\,4$. Using the character table of
$\mathfrak{S}_4$ \cite[Lecture~2]{FultonHarris1991}:
\[
  r_{(3)} = \tfrac{1}{24}\big[\mathrm{id}+T+Q+C+F\big]
  \quad [\dim 1,\ F=1],
\]
\[
  r_{(1^3)} = \tfrac{1}{24}\big[\mathrm{id}-T+Q+C-F\big]
  \quad [\dim 1,\ F' \propto \Delta],
\]
\[
  r_{(2,1)} = \tfrac16\,\mathrm{id} + \tfrac16\, Q - \tfrac{1}{12}\, C
  \quad [\dim 2,\ G_1=x_1^2-x_3^2],
\]
\[
      r_{(2)+(1)} = \tfrac18\big[3\,\mathrm{id}+T-Q-F\big]
  \quad [\dim 3,\ H_1=x_1],
\]
\[
  r_{(1,1)+(1)} = \tfrac18\big[3\,\mathrm{id}-T-Q+F\big]
  \quad [\dim 3,\ H_1'=x_1(x_2^2-x_3^2)].
\]
One checks $r_{(3)}+r_{(1^3)}+r_{(2,1)}+r_{(2)+(1)}+r_{(1,1)+(1)}=\mathrm{id}$:
e.g.\ for the coefficient of $T$, $\tfrac1{24}-\tfrac1{24}+0+\tfrac18-\tfrac18=0$;
similarly for $Q$ and $F$, while $C$ only appears in $r_{(3)},r_{(1^3)},r_{(2,1)}$
with coefficients $\tfrac1{24}+\tfrac1{24}-\tfrac1{12}=0$.

\subsection{Illustrating the decomposition for $D_2$ and $D_3$}

Corollary~\ref{cor:Dn-main} (equivalently
Proposition~\ref{prop:decomposition-general} with $r=p=2$) writes
$\widetilde{\mathcal O}_X$ as a direct sum of simple
$\widetilde{\mathcal D}_Y$-modules in the two examples above. Here $f_\lambda$ stands for $f^l_\lambda=\dim V^l(\lambda)$,
the dimension of the irreducible representation attached to a summand, and we omit the index $l$.

\paragraph{$D_2$.} All four irreducible representations of $D_2$ are
$1$-dimensional (so $f_\lambda=1$ throughout, and parts~(i) and~(ii)
of Proposition~\ref{prop:decomposition-general} coincide). Using the
generators of Subsection~\ref{subsec:D2-example},
\[
  \widetilde{\mathcal O}_X =
  \widetilde{\mathcal D}_Y\cdot 1 \;\oplus\;
  \widetilde{\mathcal D}_Y\cdot(x_1^2-x_2^2) \;\oplus\;
  \widetilde{\mathcal D}_Y\cdot(x_1+x_2) \;\oplus\;
  \widetilde{\mathcal D}_Y\cdot(x_1-x_2),
\]
four summands of $\widetilde{\mathcal O}_Y$-rank $1$ each, totalling
rank $4 = |D_2|$, which matches term by term the four primitive central
idempotents $r_{\mathrm{triv}}, r_{\mathrm{sign}}, r^0, r^1$ computed
above.

\paragraph{$D_3$.} Here $D_3 \cong \mathfrak{S}_4$ has representations
of dimensions $1,1,2,3,3$. Using the generators
$F=1$, $F'\propto\Delta$, $G_1=x_1^2-x_3^2$ (together with
$G_2=x_1^2-x_2^2$, which gives an isomorphic summand by
Corollary~\ref{cor:iso-simple}), $H_1=x_1$ and $H_1'=x_1(x_2^2-x_3^2)$ of
Subsection~\ref{subsec:D3-example},
Proposition~\ref{prop:decomposition-general}(ii) gives
\[
  \widetilde{\mathcal O}_X \;\cong\;
  \widetilde{\mathcal D}_Y\cdot 1 \;\oplus\;
  \widetilde{\mathcal D}_Y\cdot \Delta \;\oplus\;
  \big(\widetilde{\mathcal D}_Y\cdot G_1\big)^{\oplus2} \;\oplus\;
  \big(\widetilde{\mathcal D}_Y\cdot H_1\big)^{\oplus3} \;\oplus\;
  \big(\widetilde{\mathcal D}_Y\cdot H_1'\big)^{\oplus3},
\]
with a higher Specht polynomial as generator of every summand. As a check
on ranks, each block contributes $f_\lambda \times f_\lambda = f_\lambda^2$
to the total $\widetilde{\mathcal O}_Y$-rank, and
\[
  1^2 + 1^2 + 2^2 + 3^2 + 3^2 = 1+1+4+9+9 = 24 = |D_3|.
\]

\section{Other particular cases}\label{sec:other-cases}

This section works out Theorem~\ref{thm:main-general} for $G(r,r,n)$ (Subsection~\ref{subsec:Grrn}), for $W(B_n)=G(2,1,n)$
(Subsection~\ref{subsec:Bn}) and for $G(r,1,2)$ (Subsections~\ref{subsec:Gr1n} and~\ref{subsec:Gr1n-idem}), and gives two examples with $p>1$:
$G(4,2,2)$ and $G(6,2,2)$ (Subsections~\ref{subsec:G422} and~\ref{subsec:G622}).

\subsection{The case $G(r,r,n)$}\label{subsec:Grrn}

The case $p=r$ (so $d=1$) contains $W(D_n)=G(2,2,n)$. There are no diagonal reflections, and
\[
 y_j = \sum_{i=1}^n x_i^{\,jr},\ 1\le j\le n-1, \qquad y_n = x_1\cdots x_n,
 \qquad
 \Delta_{G(r,r,n)} = c_0\prod_{1\le i<j\le n}(x_i^r-x_j^r).
\]
Lemma~\ref{lem:jacobian-general} simplifies to
\[
 \det\mathcal J = (-1)^{n-1} r^{\,n-1}(n-1)! \prod_{i<j}(x_j^r-x_i^r).
\]

\begin{corollary}\label{cor:Grrn}
Let $e=e_T\in\mathbb{C}[G(r,r,n)]$ be an adapted primitive idempotent, attached to an orbit $\mathcal O$ and to an index $l$, with the
notation of Theorem~\ref{thm:main-general}. Then the $\widetilde{\mathcal{D}}_Y$-module $e\,\widetilde{\mathcal{O}}_X$ is simple (this holds for every
primitive idempotent), and
$e\,\widetilde{\mathcal{O}}_X=\widetilde{\mathcal D}_Y F^l_{S,\tau_\mu(T)}$ for every $(\mu,S)\in I(\mathcal O)$, where
$F^l_{S,\tau_\mu(T)}$ is a higher Specht polynomial. Moreover
\[
 \widetilde{\mathcal{O}}_X \cong \bigoplus_{\lambda \in \mathcal{P}_{r,n}/\sim}
 \bigoplus_{l=0}^{e(\lambda)-1}
 \big(\widetilde{\mathcal{D}}_Y F^l_{\lambda,1}\big)^{\oplus f^l_\lambda}.
\]
Here $\lambda$ runs over representatives of the orbits, with the notation of Proposition~\ref{prop:decomposition-general}.
Taking $r=2$ recovers Corollary~\ref{cor:Dn-main} for $W(D_n)$.
\end{corollary}

\begin{proof}
This is Theorem~\ref{thm:main-general} and Proposition~\ref{prop:decomposition-general} for $p=r$ (so $d=1$), with the data above.
\end{proof}

\subsection{The real case $W(B_n) = G(2,1,n)$}\label{subsec:Bn}

The case $r=2$, $p=1$ (so $d=2$) is the Weyl group $W(B_n)$ of type $B_n$ (equivalently $C_n$): the group of signed permutations of $\{1,\dots,n\}$, of order $2^nn!$. Unlike $W(D_n)$, it has no parity constraint on the sign changes, hence $p=1$ rather than $p=r$. This case was treated in~\cite{NonkaneLawson2021B}; here it is the specialization $r=2$, $p=1$ of Theorem~\ref{thm:main-general}.

Its root system consists of the $n$ short roots $\epsilon_i$
($1 \le i \le n$, giving the reflections $x_i \mapsto -x_i$) together
with the $n(n-1)$ long roots $\epsilon_i \pm \epsilon_j$ ($i<j$, the same reflections as for $D_n$), so $W(B_n)$ has $n+n(n-1)=n^2$ reflections. By Subsection~\ref{subsec:localization} (with $r=2$, $d=2$), the fundamental invariants are
\[
  y_j = \sum_{i=1}^n x_i^{2j}, \quad 1 \le j \le n-1,
  \qquad
  y_n = (x_1 \cdots x_n)^2,
\]
and Lemma~\ref{lem:jacobian-general} (with $r=2$, $d=2$) gives the
Jacobian polynomial
\[
  \Delta_{W(B_n)} = c_0 \Big(\prod_{i=1}^n x_i\Big) \prod_{1 \le i < j \le n} (x_j^2 - x_i^2)
\]
for some nonzero constant $c_0$. It has one linear factor $x_i$ per short root and one
quadratic factor $x_j^2-x_i^2$ per pair of long roots.

\begin{corollary}\label{cor:Bn}
Let $e=e_T\in\mathbb{C}[W(B_n)]$ be an adapted primitive idempotent (Subsection~\ref{subsec:idempotents}), attached to
$\lambda\in\mathcal P_{2,n}$; here the orbit is $\{\lambda\}$, $l=0$, $\lambda_0=\lambda$ and $\tau_\lambda(T)=T$. Then the
$\widetilde{\mathcal{D}}_Y$-module $e\,\widetilde{\mathcal{O}}_X$ is simple (this holds for every primitive idempotent), and
$e\,\widetilde{\mathcal{O}}_X=\widetilde{\mathcal D}_Y F_{S,T}$ for every $S\in\mathrm{STab}(\lambda)$, where $F_{S,T}=F^0_{S,T}$ is a
higher Specht polynomial, by Theorem~\ref{thm:main-general}. Moreover
\[
  \widetilde{\mathcal{O}}_X \cong \bigoplus_{\lambda' \in \mathcal{P}_{2,n}}
    \big(\widetilde{\mathcal{D}}_Y F_{\lambda',1}\big)^{\oplus f_{\lambda'}},
\]
where $F_{\lambda',1}=F^0_{\lambda',1}$ is as in Proposition~\ref{prop:decomposition-general}, and the sum runs over all $\lambda'\in\mathcal P_{2,n}$: since $p=1$, $\mathrm{Sh}$ is the identity and $e(\lambda)=1$, so there is no quotient by $\sim$ and no index $l$, unlike the case of $W(D_n)$.
\end{corollary}

\begin{proof}
Immediate specialization of Theorem~\ref{thm:main-general} and
Proposition~\ref{prop:decomposition-general} to $r=2$, $p=1$, as in Corollary~\ref{cor:Grrn}.
\end{proof}

\begin{remark}
The case $n=2$ of Corollary~\ref{cor:Bn} is $G(2,1,2)=W(B_2)$. Its eight higher Specht
polynomials and its central idempotents are the case $r=2$ of
Subsections~\ref{subsec:Gr1n}--\ref{subsec:Gr1n-idem} (Examples~\ref{ex:B2-basis} and~\ref{ex:B2-idem}), and its adapted primitive
idempotents are in Example~\ref{ex:prop29}; Examples~\ref{ex:CW-decomp}
and~\ref{ex:prop29} in Section~\ref{sec:general} use it as well.
\end{remark}

\subsection{The case $G(r,1,n)$ with $n=2$}\label{subsec:Gr1n}

Take $p=1$ (so $d=r$): this is the group $G(r,n)$ of
Subsection~\ref{subsec:specht}, i.e.\ the full monomial
group $(\mathbb{Z}/r\mathbb{Z})^n \rtimes S_n$, with fundamental
invariants $e_j(x_1^r,\dots,x_n^r)$, $1\le j\le n$ (equivalent to the $y_j$ of
Subsection~\ref{subsec:localization}), and no restriction
by $\sim$ (every class $\lambda \in \mathcal{P}_{r,n}$ is a singleton, since
$d=r$ forces $b(\lambda)=1$, $e(\lambda)=1$).

For $n=2$, $|G(r,1,2)| = 2r^2$. The higher Specht polynomials
(Subsection~\ref{subsec:specht}) can all be written down; the
general-$n$ case of this decomposition is treated
in~\cite{NonkaneLawson2022}.

\paragraph{Irreducible components indexed by a single row/column,
component $\nu \in \{0,\dots,r-1\}$ (dimension 1 each).}
Shape $(2)$ in component $\nu$, standard tableau
$T = \begin{array}{|c|c|}\hline 1 & 2\\\hline\end{array}$ (here $i(S)=(0,0)$):
\[
 F_\nu = (x_1 x_2)^{\nu}.
\]
Shape $(1,1)$ in component $\nu$, standard tableau
$T = \begin{array}{|c|}\hline 1\\\hline 2\\\hline\end{array}$ (here $i(S)=(0,1)$, so the monomial is $(x_1x_2)^\nu x_2^{r}$, and the
antisymmetrizer of the column turns it into a multiple of the polynomial below):
\[
 F'_\nu = (x_1x_2)^{\nu}(x_1^{r} - x_2^{r}).
\]

\paragraph{Irreducible components indexed by a pair of components
$0\le \nu < \mu \le r-1$, each with a single box (dimension 2 each,
multiplicity 2 in the regular representation).} The two standard
tableaux are
\[
  T_a = \Big(\nu:\begin{array}{|c|}\hline 1\\\hline\end{array}
    \ ,\ \mu:\begin{array}{|c|}\hline 2\\\hline\end{array}\Big),
  \qquad
  T_b = \Big(\nu:\begin{array}{|c|}\hline 2\\\hline\end{array}
    \ ,\ \mu:\begin{array}{|c|}\hline 1\\\hline\end{array}\Big),
\]
(box $1$ in component $\nu$ or $\mu$ respectively). For $S=T_a$ one has $i(S)=(0,0)$; for $S=T_b$ the number $2$ is read before
$1$, so $i(S)=(0,1)$ and the variable placed in the cell of component $\nu$ gets the extra exponent $r$. The polynomials $F_{S,T}$,
$T\in\{T_a,T_b\}$, are therefore
\[
 \{x_1^{\nu}x_2^{\mu},\; x_1^{\mu}x_2^{\nu}\} \ \ (S=T_a)
 \qquad\text{and}\qquad
 \{x_1^{r+\nu}x_2^{\mu},\; x_1^{\mu}x_2^{r+\nu}\} \ \ (S=T_b).
\]

\begin{proposition}\label{prop:Gr1n-basis}
The polynomials above form a basis of $\Lambda_0 \cong \mathbb{C}[G(r,1,2)]$
as a vector space, of total dimension
\[
 \underbrace{r}_{F_\nu} + \underbrace{r}_{F'_\nu}
 + \underbrace{4\binom{r}{2}}_{\text{pairs}}
 = 2r + 2r(r-1) = 2r^2 = |G(r,1,2)|.
\]
\end{proposition}

\begin{proof}
\emph{Step 1 -- Enumeration of $\mathcal{P}_{r,2}$.} An $r$-tuple
$\lambda=(\lambda^0,\dots,\lambda^{r-1})$ of Young diagrams of total size
$2$ is of one of three types: (a) a single component $\nu$ carries
shape $(2)$, the rest empty ($r$ choices of $\nu$); (b) a single
component $\nu$ carries shape $(1,1)$ ($r$ choices); (c) two
\emph{distinct} components $\nu < \mu$ each carry shape $(1)$
($\binom{r}{2}$ choices). These exhaust $\mathcal P_{r,2}$: the partitions of $2$ are $(2)$ and $(1,1)$, and $2$ splits across components only as $2$ or $1+1$.

\emph{Step 2 -- Cases (a)--(b).} The shapes $(2)$ and $(1,1)$ each
admit exactly one standard tableau, so $f_\lambda = 1$. Since $p=1$
gives $d=r$, the condition ``$1 \in T^\nu$ for some $\nu < d=r$''
defining $\mathrm{STab}(\lambda)_d$ is automatic (there are only $r$
components), so $\mathrm{STab}(\lambda)_d = \mathrm{STab}(\lambda)$,
reduced to the unique tableau. This yields one higher Specht polynomial $F^0_{S,T}$ per such $\lambda$, namely $F_\nu$,
resp.\ $F'_\nu$, as computed above from the construction of
Subsection~\ref{subsec:specht}.

\emph{Step 3 -- Case (c).} A shape with one box in each of two components admits $2$ standard tableaux $T_a,T_b$ (according to whether $1$ occupies
component $\nu$ or $\mu$), so $f_\lambda = 2$, consistent with the
multinomial formula $\binom{2}{1,1}\cdot 1\cdot 1 = 2$. Again
$\mathrm{STab}(\lambda)_d = \mathrm{STab}(\lambda) = \{T_a,T_b\}$
without restriction. This yields $S \in \{T_a,T_b\}$,
$T \in \{T_a,T_b\}$, i.e.\ four polynomials $F^0_{S,T}$ per such
$\lambda$: these are the two pairs
$\{x_1^\nu x_2^\mu, x_1^\mu x_2^\nu\}$ and
$\{x_1^{r+\nu}x_2^\mu, x_1^\mu x_2^{r+\nu}\}$ displayed above.

\emph{Step 4 -- Conclusion.} Steps 1--3 list exactly the family $\{F^l_{S,T}\}$ of
Theorem~\ref{thm:2.3}(iii) for $(r,p,n)=(r,1,2)$ (where $e(\lambda)=1$ for every
$\lambda$, so only $l=0$ occurs), with neither omission nor repetition. By that theorem it is a
basis of $\mathbb C[x_1,x_2]$ over the invariants, so its classes form a basis of $\Lambda_0$. The
count
\[
  r\cdot 1+r\cdot 1+\binom{r}{2}\cdot 4=2r+2r(r-1)=2r^2
\]
agrees with $|G(r,1,2)|=\sum_\lambda f_\lambda^2$.
\end{proof}

\begin{example}[$r=2$: the Coxeter group $B_2$]\label{ex:B2-basis}
Here $|G(2,1,2)|=8$, $y_1=x_1^2+x_2^2$ and $y_2=(x_1x_2)^2$. The four singleton shapes give the higher Specht polynomials
\[
 F_0 = 1,\quad F_1 = x_1x_2, \qquad
 F'_0 = x_1^2-x_2^2,\quad F'_1 = x_1x_2(x_1^2-x_2^2),
\]
and the single pair $(\nu,\mu)=(0,1)$ gives
\[
 \{x_1,\, x_2\} \ \ (S=T_a) \qquad \{x_1^2x_2,\, x_1x_2^2\} \ \ (S=T_b).
\]
Total: $2+2+4=8$ polynomials, matching $|G(2,1,2)|$. Their degrees are $0,1,1,2,2,3,3,4$, which agrees with the Hilbert series
$(1-t^2)(1-t^4)/(1-t)^2=1+2t+2t^2+2t^3+t^4$ of $\Lambda_0=\mathcal O_X/(y_1,y_2)$, and their classes form a basis of $\Lambda_0$.
These are the polynomials used in Example~\ref{ex:prop29}.
\end{example}

\subsection{Central primitive idempotents for $G(r,1,2)$}\label{subsec:Gr1n-idem}

The polynomials of Proposition~\ref{prop:Gr1n-basis} span representations of $G(r,1,2)$ whose
characters give the central primitive idempotents: the element $(a,b;e)$ multiplies $x_1,x_2$ by $\xi^a,\xi^b$, and $\mathrm{swap}$
exchanges them. The characters below are those of the lines spanned by $F_\nu$ and $F'_\nu$ and of the planes spanned by
$\{x_1^\nu x_2^\mu,x_1^\mu x_2^\nu\}$.

\paragraph{$G(r,1,2)$.}
Write elements as $(a,b;\sigma)$, $a,b \in \mathbb{Z}/r\mathbb{Z}$,
$\sigma \in \{e,\mathrm{swap}\}$, and $\xi = e^{2i\pi/r}$. The
characters are:
\[
  \chi_{F_\nu}(a,b;\sigma) = \xi^{\nu(a+b)} \ (\forall \sigma),
  \qquad
  \chi_{F'_\nu}(a,b;\sigma) = \xi^{\nu(a+b)} \cdot
    \begin{cases} +1 & \sigma = e \\ -1 & \sigma = \mathrm{swap}, \end{cases}
\]
\[
  \chi_{(\nu,\mu)}(a,b;e) = \xi^{a\nu+b\mu}+\xi^{a\mu+b\nu},
  \qquad
  \chi_{(\nu,\mu)}(a,b;\mathrm{swap}) = 0.
\]
The corresponding central primitive idempotents are
\[
  r_{F_\nu} = \frac{1}{2r^2} \sum_{a,b=0}^{r-1} \xi^{-\nu(a+b)}
    \big[(a,b;e)+(a,b;\mathrm{swap})\big],
\]
\[
  r_{F'_\nu} = \frac{1}{2r^2} \sum_{a,b=0}^{r-1} \xi^{-\nu(a+b)}
    \big[(a,b;e)-(a,b;\mathrm{swap})\big],
\]
\[
  r_{(\nu,\mu)} = \frac{1}{r^2} \sum_{a,b=0}^{r-1}
    \big(\xi^{-a\nu-b\mu}+\xi^{-a\mu-b\nu}\big)(a,b;e), \qquad 0 \le \nu < \mu \le r-1.
\]

\begin{example}[$r=2$]\label{ex:B2-idem} Here $\xi=-1$, and $r_{F_0} = \tfrac18\sum_{a,b\in\{0,1\}}\big[(a,b;e)+(a,b;\mathrm{swap})\big]$
is $\tfrac18$ times the sum of all $8$ group elements, while
$r_{F_1} = \tfrac18\sum_{a,b}(-1)^{a+b}\big[(a,b;e)+(a,b;\mathrm{swap})\big]$,
$r_{F'_0}$ and $r_{F'_1}$ are obtained from $r_{F_0}$ and $r_{F_1}$ by changing the sign of the terms in $\mathrm{swap}$,
and $r_{(0,1)} = \tfrac14\sum_{a,b}\big[(-1)^{b}+(-1)^{a}\big](a,b;e)$;
these five idempotents match the character table of $B_2$. Its five conjugacy classes, of sizes $1,1,2,2,2$, are those of the
elements $(0,0;e)$, $(1,1;e)$, $(1,0;e)$, $(0,0;\mathrm{swap})$ and $(1,0;\mathrm{swap})$, and the characters take the values
\[
\renewcommand{\arraystretch}{1.15}
\begin{array}{c|ccccc}
 & (0,0;e) & (1,1;e) & (1,0;e) & (0,0;\mathrm{swap}) & (1,0;\mathrm{swap})\\\hline
\chi_{F_0} & 1 & 1 & 1 & 1 & 1\\
\chi_{F_1} & 1 & 1 & -1 & 1 & -1\\
\chi_{F'_0} & 1 & 1 & 1 & -1 & -1\\
\chi_{F'_1} & 1 & 1 & -1 & -1 & 1\\
\chi_{(0,1)} & 2 & -2 & 0 & 0 & 0
\end{array}
\]
\end{example}

\begin{remark}
For $D_2$ and $D_3$ (Subsection~\ref{subsec:D23-idem}) and for $G(r,1,2)$ (this subsection), the central
idempotents computed above sum to the identity and satisfy $r_\lambda r_\mu=\delta_{\lambda\mu}r_\lambda$. The first
identity is the column orthogonality of the character table with the identity column. The second says that the $r_\lambda$ are the
identities of the simple blocks of $\mathbb C[G]$. We also checked the formulas by computer: for $D_3\cong\mathfrak S_4$ in exact
arithmetic, and for $G(r,1,2)$ with $2\le r\le6$ (sum, idempotence, orthogonality, centrality).
\end{remark}

\subsection{An example with $p>1$: $G(4,2,2)$}\label{subsec:G422}

Let $r=4$, $p=2$, $n=2$, so $d=2$ and $|W|=16$; $W$ is the group of monomial
matrices with entries in $\mu_4$ and product of the entries equal to $\pm1$.
Write $A=\widetilde{\mathcal O}_Y$, $B=\widetilde{\mathcal O}_X$ and
$\mathcal D=\widetilde{\mathcal D}_Y$ as in
Subsection~\ref{subsec:galois-descent}. The basic invariants are
$y_1=x_1^4+x_2^4$ and $y_2=(x_1x_2)^2$, and, since $d=2$,
$\Delta=c\,x_1x_2(x_1^4-x_2^4)$ and $\delta=\Delta^2$ up to a scalar
(Lemma~\ref{lem:delta}); so this example is not affected by the correction of
the localization (see Remark~\ref{rem:delta-sq}), but it exercises the shift
$\mathrm{Sh}=\mathrm{sh}^{\,2}$, which exchanges the components $0\leftrightarrow2$ and
$1\leftrightarrow3$, and the splitting by $l$. Let $s$ be the swap of $x_1,x_2$,
$t_i$ the sign change of $x_i$, and $c=\mathrm{diag}(i,i)$; these generate $W$.

The set $\mathcal P_{4,2}$ has $14$ elements, which form $8$ orbits under $\mathrm{Sh}$.
We write a shape as the $4$-tuple of its components, with $-$ for an empty
diagram. The orbits, the values of the characters on $(s,t_1,c)$ and the
generators are as follows ($F$ is a higher Specht polynomial, and $eB=A\,F$ for
the rank-one blocks).
\[
\renewcommand{\arraystretch}{1.15}
\begin{array}{c|c|c|c|c|c}
\text{orbit } \mathcal O & b,e & l & \dim & (\chi(s),\chi(t_1),\chi(c)) & \text{generator}\\\hline
\{(2,-,-,-),(-,-,2,-)\} & 2,1 & 0 & 1 & (1,1,1) & 1\\
\{(11,-,-,-),(-,-,11,-)\} & 2,1 & 0 & 1 & (-1,1,1) & x_1^4-x_2^4\\
\{(1,-,1,-)\} & 1,2 & 0 & 1 & (1,1,-1) & x_1^2+x_2^2\\
\{(1,-,1,-)\} & 1,2 & 1 & 1 & (-1,1,-1) & x_1^2-x_2^2\\
\{(-,2,-,-),(-,-,-,2)\} & 2,1 & 0 & 1 & (1,-1,-1) & x_1x_2\\
\{(-,11,-,-),(-,-,-,11)\} & 2,1 & 0 & 1 & (-1,-1,-1) & x_1x_2(x_1^4-x_2^4)\\
\{(-,1,-,1)\} & 1,2 & 0 & 1 & (1,-1,1) & x_1x_2(x_1^2+x_2^2)\\
\{(-,1,-,1)\} & 1,2 & 1 & 1 & (-1,-1,1) & x_1x_2(x_1^2-x_2^2)\\\hline
\{(1,1,-,-),(-,-,1,1)\} & 2,1 & 0 & 2 & (0,0,2i) & x_2,\ x_1\\
\{(1,-,-,1),(-,1,1,-)\} & 2,1 & 0 & 2 & (0,0,-2i) & x_2^3,\ x_1^3
\end{array}
\]
(The dimensions add up: $8\cdot1^2+2\cdot2^2=16=|W|$.)

\emph{Splitting.} For $\mathcal O=\{(1,-,1,-)\}$ one has $\mathrm{Sh}\,\lambda=\lambda$, so
$b=1$ and $e=2$. With $S=T=(0{:}\,1,\ 2{:}\,2)$ and $\mathrm{Sh}\,T=(0{:}\,2,\ 2{:}\,1)$,
$F_{S,T}=x_2^2$ and $F_{S,\mathrm{Sh}T}=x_1^2$, and
$F^l_{S,T}=x_2^2+(-1)^lx_1^2$, since $\xi^{lmdb}=i^{2lm}=(-1)^{lm}$. These give the two
one-dimensional representations $l=0,1$ of the table.

\emph{Shift, and why the whole orbit is needed.} Take
$\mathcal O'=\{\lambda',\mathrm{Sh}\,\lambda'\}$ with $\lambda'=(1,-,-,1)$ and
$\mathrm{Sh}\,\lambda'=(-,1,1,-)$; here $b=2$, $e=1$, $f_{\lambda'}=2$. The
standard tableaux with $1$ in a component $<d=2$ are: for $\lambda'$, only
$S=(0{:}\,1,\ 3{:}\,2)$; for $\mathrm{Sh}\,\lambda'$, only $S'=(1{:}\,1,\ 2{:}\,2)$.
Thus $|\mathrm{STab}(\lambda')_d|=1\ne f_{\lambda'}/e=2$: a single shape does not suffice
(Remark~\ref{rem:diffMY}), while $I(\mathcal O')=\{(\lambda',S),(\mathrm{Sh}\,\lambda',S')\}$ has
$2=\dim V$ elements. The higher Specht polynomials are $x_2^3,\,x_1^3$ for
$(\lambda',S)$ (with $T=(0{:}\,1,3{:}\,2)$ and $(0{:}\,2,3{:}\,1)$) and
$x_1x_2^2,\,x_1^2x_2$ for $(\mathrm{Sh}\,\lambda',S')$ (with $T'=(1{:}\,1,2{:}\,2)$ and
$(1{:}\,2,2{:}\,1)$), and the adapted idempotents (Subsection~\ref{subsec:idempotents})
are $e_{T_a}=r_V\frac{1+t_1}2$ and $e_{T_b}=r_V\frac{1+t_2}2$, where $r_V$ is the central idempotent of the block, with
\[
  e_{T_a}B=A\,x_2^3\oplus A\,x_1^2x_2=\mathcal D\,x_2^3,\qquad
  e_{T_b}B=A\,x_1^3\oplus A\,x_1x_2^2=\mathcal D\,x_1^3 .
\]
For the other $2$-dimensional block, $\mathcal O''=\{(1,1,-,-),(-,-,1,1)\}$, the situation is different: both standard tableaux of
$(1,1,-,-)$ have their number $1$ in a component $<2$, and none of those of $(-,-,1,1)$ has. So $I(\mathcal O'')$ consists of two
pairs with the same shape, $S_1=(0{:}\,1,\ 1{:}\,2)$ with $i(S_1)=(0,0)$ and $S_2=(0{:}\,2,\ 1{:}\,1)$ with $i(S_2)=(0,1)$, which carries the
extra exponent $r=4$. They give the polynomials $x_2,\,x_1$ (for $S_1$) and $x_1^4x_2,\,x_1x_2^4$ (for $S_2$), and
$e_{T_a}B=A\,x_2\oplus A\,x_1^4x_2=\mathcal D\,x_2$,
$e_{T_b}B=A\,x_1\oplus A\,x_1x_2^4=\mathcal D\,x_1$.
Put $u=y_1^2-4y_2^2=(x_1^4-x_2^4)^2$, so that
$\delta=y_2u$. Using
$\partial/\partial y_j=\sum_i(\mathcal J^{-1})_{ji}\,\partial/\partial x_i$
(proof of Lemma~\ref{lem:jacobian-general}), a direct computation gives
\[
  x_1^2x_2=-\frac{2u}{3y_2}\,\partial_{y_1}(x_2^3)+\frac{y_1}{2y_2}\,x_2^3,
  \qquad
  x_1^4x_2=\frac{y_1}{2}\,x_2-2u\,\partial_{y_1}(x_2),
\]
with coefficients in $A$ (since $y_2$ divides $\delta$, $1/y_2\in A$), so $x_1^2x_2\in\mathcal D\,x_2^3$ and
$x_1^4x_2\in\mathcal D\,x_2$, in agreement with Theorem~\ref{thm:main-general}(iii).

\emph{Decomposition.} By Proposition~\ref{prop:decomposition-general},
\begin{multline*}
  B=A\oplus A(x_1^4-x_2^4)\oplus A(x_1^2+x_2^2)\oplus A(x_1^2-x_2^2)\oplus A\,x_1x_2
  \oplus A\,x_1x_2(x_1^4-x_2^4)\\
  \oplus A\,x_1x_2(x_1^2+x_2^2)\oplus A\,x_1x_2(x_1^2-x_2^2)
  \oplus\mathcal D\,x_2\oplus\mathcal D\,x_1\oplus\mathcal D\,x_2^3\oplus\mathcal D\,x_1^3 ,
\end{multline*}
a direct sum of $12$ simple $\mathcal D$-modules (eight of rank $1$ and four of rank $2$, total
rank $8+8=16=|W|$), where $\mathcal D\,x_2\cong\mathcal D\,x_1$ and
$\mathcal D\,x_2^3\cong\mathcal D\,x_1^3$ are isomorphic but distinct submodules, and
the two isomorphism classes of rank $2$ are not isomorphic to each other. Up to isomorphism,
$B\cong A\oplus\cdots\oplus A\,x_1x_2(x_1^2-x_2^2)\oplus(\mathcal D\,x_2)^{\oplus2}\oplus(\mathcal D\,x_2^3)^{\oplus2}$.
These data were also checked by computer.
Subsection~\ref{subsec:G622} treats $G(6,2,2)$, where $d=3$ and $\delta\ne\Delta^2$.

\subsection{An example with $d=3$: $G(6,2,2)$}\label{subsec:G622}

Let $r=6$, $p=2$, $n=2$, so $d=3$ and $|W|=36$; $W$ is the group of monomial matrices with
entries in $\mu_6$ whose product of entries is a cube root of unity. Let $\xi=e^{i\pi/3}$.
It is generated by the swap $s$, the diagonal reflection $a=\mathrm{diag}(\xi^2,1)$ of order $3$, and
the central element $c=\mathrm{diag}(\xi,\xi)$. Write $A,B,\mathcal D$ as in
Subsection~\ref{subsec:galois-descent}. The basic invariants are $y_1=x_1^6+x_2^6$ and $y_2=(x_1x_2)^3$.
Put $u=y_1^2-4y_2^2=(x_1^6-x_2^6)^2$. Then
\[
  \Delta=c_0\,(x_1x_2)^2(x_1^6-x_2^6),\qquad \delta=y_2\,u=(x_1x_2)^3(x_1^6-x_2^6)^2 .
\]
Here $d=3$, so $\Delta^2$ is \emph{not} invariant (Lemma~\ref{lem:delta}): $a\cdot\Delta=\xi^4\Delta$ and
$a\cdot\Delta^2=\xi^2\Delta^2\ne\Delta^2$, while $a\cdot\delta=\delta$. In contrast with $G(4,2,2)$, we invert $\delta$, not $\Delta^2$.

\emph{Orbits.} $\mathcal P_{6,2}$ has $27$ elements; $\mathrm{Sh}=\mathrm{sh}^3$ sends component $c$ to $c+3$.
There are $15$ orbits: six of size $2$ with a single diagram $(2)$ or $(1,1)$ in a component ($b=2$,
$e=1$, $f=1$); three fixed shapes with one box in each of the components $k$ and $k+3$ ($b=1$, $e=2$, $f=2$),
which split into two characters ($l=0,1$); and six orbits of size $2$ of shapes with boxes in components
$\{c_1,c_2\}$, $c_2\ne c_1+3$ ($b=2$, $e=1$, $f=2$). Hence $W$ has twelve characters and six
irreducible representations of dimension $2$ ($12+6\cdot4=36$).

\emph{The twelve characters.} For $k=0,1,2$ and each of the four polynomials $F_2$ below, the generator is
$(x_1x_2)^kF_2$:
\[
\renewcommand{\arraystretch}{1.15}
\begin{array}{c|cccc}
 & 1 & x_1^6-x_2^6 & x_1^3+x_2^3 & x_1^3-x_2^3\\\hline
\chi_{F_2}(s,a,c) & (1,1,1) & (-1,1,1) & (1,1,-1) & (-1,1,-1)\\
\text{shapes} & (2) & (1,1) & l=0 & l=1
\end{array}
\]
and $\chi_{(x_1x_2)^k}(s,a,c)=(1,\xi^{2k},\xi^{2k})$; the character of $(x_1x_2)^kF_2$ is the product. The shapes
for a given $k$ lie in the components $\{k,k+3\}$: the diagrams $(2)$ and $(1,1)$ in component $k$, and the
fixed shape with one box in each of $k$ and $k+3$ for the last two columns.

\emph{The six blocks of dimension $2$.} $c$ acts by a scalar on each block, and the six scalars $\xi,\dots,\xi^6$ are
distinct. The sign change $t_1=\mathrm{diag}(-1,1)$ is not in $W$ (the exponents $3,0$ have an odd sum), so the idempotents
of $G(4,2,2)$ do not carry over verbatim. Here the two basis vectors of a block are separated by the eigenvalues of
$a$: if $\mu$ is the eigenvalue of $a$ on a basis vector, the adapted idempotent is
$r_V\cdot\tfrac13\big(1+\mu^{-1}a+\mu^{-2}a^2\big)$, where $r_V$ is the central idempotent of the block. We
checked this for all twelve idempotents. Then $e_{T_a}B=A\,g\oplus A\,h$ (with $\mu$ the eigenvalue of $a$ on $g$),
and $e_{T_b}B$ is its image under the swap:
\[
\renewcommand{\arraystretch}{1.15}
\begin{array}{c|c|c|c|c}
\text{block} & c & \text{eigenvalues of }a & (g,h)\text{ for }e_{T_a}B & h\text{ in terms of }g\\\hline
F & \xi   & \{1,\xi^2\}     & (x_2,\ x_1^6x_2)         & \tfrac{y_1}2\,g-3u\,\partial_{y_1}g\\
E & \xi^2 & \{1,\xi^4\}     & (x_2^2,\ x_1^6x_2^2)     & \tfrac{y_1}2\,g-\tfrac{3u}2\,\partial_{y_1}g\\
B & -1    & \{\xi^2,\xi^4\} & (x_1x_2^2,\ x_1^7x_2^2)  & \tfrac{y_1}2\,g-3u\,\partial_{y_1}g\\
D & \xi^4 & \{1,\xi^2\}     & (x_2^4,\ x_1^3x_2)       & \tfrac{y_1}{2y_2}\,g-\tfrac{3u}{4y_2}\,\partial_{y_1}g\\
C & \xi^5 & \{1,\xi^4\}     & (x_2^5,\ x_1^3x_2^2)     & \tfrac{y_1}{2y_2}\,g-\tfrac{3u}{5y_2}\,\partial_{y_1}g\\
A & 1     & \{\xi^2,\xi^4\} & (x_1x_2^5,\ x_1^4x_2^2)  & \tfrac{y_1}{2y_2}\,g-\tfrac{3u}{4y_2}\,\partial_{y_1}g
\end{array}
\]
The blocks $F,E,B$ come from the orbits whose boxes lie in the components $\{0,1\},\{0,2\},\{1,2\}$ (up to the shift by $3$): both
standard tableaux of the shape in these components have their number $1$ in a component $<3$, so $I(\mathcal O)$ consists of two
pairs on the same shape, and $h$ comes from the second one, with the extra exponent $r=6$ (as for $\mathcal O''$ in
Subsection~\ref{subsec:G422}). The blocks $D,C,A$ come from $\{1,3\},\{2,3\},\{2,4\}$: $I(\mathcal O)$ has one pair on each of the two
shapes of the orbit (as for $\mathcal O'$ there). In all cases $g$ and $h$ are higher Specht polynomials $F_{S,\tau_\mu(T_a)}$ as in
Theorem~\ref{thm:main-general}(iii).
Each identity is an exact relation in $B$ (we checked all six symbolically). Its coefficients lie in $A$: the only denominators are powers of $y_2$, which divides $\delta=y_2u$. So $h\in\mathcal D\,g$ and $e_{T_a}B=\mathcal D\,g$.

\emph{Decomposition.} By Proposition~\ref{prop:decomposition-general}, $B$ is the direct sum of the twelve
rank-one modules $A\,F$ and of the twelve rank-two modules $\mathcal D\,g$, $\mathcal D\,g'$ ($g,g'$ the first generators of
$e_{T_a}B$, $e_{T_b}B$ in the six blocks), of total rank $12+12\cdot2=36=|W|$. Inside a block, the two
modules are isomorphic but distinct; the six isomorphism classes of rank $2$ are pairwise non-isomorphic. The shapes, characters and generators of the table were also checked by computer.

\section{Conclusion and discussion}\label{sec:conclusion}

\paragraph{Summary.} For the imprimitive complex reflection groups
$G(r,p,n)$, the polynomial ring localized away from the
reflecting hyperplanes is a semisimple module over the ring of differential
operators on the localized invariant ring. Its simple summands are the
modules $\nabla(V^*)$ attached to the irreducible representations of $W$, and they are
generated by higher Specht polynomials, with index sets and multiplicities
for $p>1$. We replaced $\Delta^2$ by the
discriminant $\delta=\prod_H\alpha_H^{e_H}$, since $\Delta^2$ is not invariant when
$r/p\ge3$. The earlier decomposition theorems for $\mathcal S_n$,
$W(D_n)$, $W(B_n)$ and $G(r,1,n)$ are recovered, for $G(r,1,n)$ with $r\ge3$ after this change
of localization. The cases with $p>1$ other than $W(D_n)$, for instance $G(4,2,2)$
and $G(6,2,2)$, are new; they need the shift of Morita and Yamada. To our knowledge, the characterization of the direct summands
of $\widetilde{\mathcal O}_X$ (Corollary~\ref{cor:characterization}) is new as well.

\paragraph{Exceptional complex reflection groups.} The Shephard--Todd classification~\cite{ShephardTodd1954} consists of the infinite family $G(r,p,n)$ together with $34$
exceptional groups. The \'etale cover, the descent theorem, the simplicity
proposition and the structure theorem of
Subsection~\ref{subsec:galois-descent} should extend to every finite complex
reflection group $W$, replacing Lemma~\ref{lem:jacobian-general} by the general
fact that the Jacobian of the basic invariants is $c\prod_H\alpha_H^{e_H-1}$,
and the Specht basis in Lemma~\ref{lem:free-rank} by Chevalley's theorem that
$\mathcal O_X$ is free over $\mathcal O_X^W$ with coinvariants the regular
representation; we have not checked every detail in that generality. What is specific to $G(r,p,n)$ is
the combinatorics of the higher Specht polynomials, hence the
generators; for the exceptional groups these would have to be found case by
case.

\paragraph{Rational Cherednik algebras.} For the parameters $t=1$, $c=0$, the
rational Cherednik algebra \cite{EtingofGinzburg2002} is $H_{1,0}(W)=\mathcal D(\h)\rtimes W$, where $\mathcal D(\h)$ is the Weyl algebra of $\h=\mathbb C^n$, and its
polynomial representation $\mathcal O_X$ is already simple over $\mathcal D(\h)$,
so there is nothing to decompose there. We decompose instead the
localized polynomial ring as a module over the invariant operators. Let $U\subset\h$ be the complement of the reflecting hyperplanes, so $U=\operatorname{Spec}\widetilde{\mathcal O}_X$. The invariant operators form the spherical
subalgebra $eH_{1,0}(W)e=\mathcal D(\h)^W$, which becomes $\mathcal D(U)^W=\mathcal D(U/W)$ after
localization at $\delta$ (Proposition~\ref{prop:etale}(d)). The module $\widetilde{\mathcal O}_X$ is then
the direct image of $\mathcal O_U$ under $U\to U/W$, and by Galois descent
(Theorem~\ref{thm:descent}) its summands $M_V$ correspond to the equivariant modules
$\mathcal O_U\otimes V^*$. We do not address category $\mathcal O$ or the deformation to $c\ne0$, where Dunkl operators
would enter.

\paragraph{Monodromy.} The modules $\nabla V$ are flat bundles on $U/W$ whose
monodromy factors through the quotient $W$ of the fundamental group $\pi_1(U/W)$ (defined in Remark~\ref{rem:classical}), the complex
braid group of $W$ \cite{BMR1998}. Conversely, a flat bundle whose monodromy factors through $W$ is one of the $\nabla V$ only if it is regular at infinity,
since an irregular connection can have trivial monodromy. By the Riemann--Hilbert correspondence, regular flat bundles
on $U/W$ correspond to representations of the braid group; we have not studied those that do not
factor through $W$.

\paragraph{Explicit coordinates for $G(r,1,n)$.} For $p=1$, \cite{KKN2026} construct invariant
polynomials $U_1,\dots,U_n$ that generate $\mathcal O_Y$, and rational operators
$\mathcal D_1,\dots,\mathcal D_n$ with $\mathcal D_dU_k=\delta_{dk}$ whose denominators divide
$\Delta$. Each $\mathcal D_d$ is then a derivation of $\widetilde{\mathcal O}_X$ that restricts to
$\partial/\partial U_d$ on $\widetilde{\mathcal O}_Y$, so by the uniqueness in
Proposition~\ref{prop:etale}(d) it is the extension of $\partial/\partial U_d$. Hence
$\widetilde{\mathcal D}_Y=\mathbb C[U_1,\dots,U_n,\delta^{-1}]\langle\mathcal D_1,\dots,\mathcal D_n\rangle$.
By Proposition~\ref{prop:etale}(d), the ring of $W$-invariant differential operators on
$\widetilde{\mathcal O}_X$ is therefore the algebra generated by the $U_k$, the $\mathcal D_d$ and
$\delta^{-1}$. This answers, after inverting $\delta$, the question left open
in~\cite[Remark~8]{KKN2026}, for algebraic operators on $\widetilde{\mathcal O}_X$. The transfer
principle of~\cite{KKN2026} is not extended there to $p>1$. The decomposition theorem of the present
paper holds for $p>1$, but we do not construct dual coordinates and operators in that case.

\paragraph{Physical motivation.} The work on $W(D_n)$ and $W(B_n)$ was motivated by the rational
quantum Olshanetsky--Perelomov system \cite{OlshanetskyPerelomov1983}; the present treatment is purely algebraic.

\paragraph{What is not done.} The theorems hold for all $n$, but we have written out the generators
only in the examples $D_2$, $D_3$, $B_2$, $G(3,1,1)$, $G(4,2,2)$ and $G(6,2,2)$, and the higher
Specht polynomials of $G(r,1,2)$ (Subsection~\ref{subsec:Gr1n}). Writing them out for $D_4$ and beyond, and
finding a closed formula for the lowest-degree generator in each isomorphism class, remain to be done.

\section*{Acknowledgments}

We thank Professor Rikard B\"ogvad for introducing us to this theory
and for instructive comments during the writing of this paper. We also thank   anonymous
referees, whose reports on the first version led us to correct the localization, the indexing of the higher Specht
polynomials, and several statements and proofs in Sections~\ref{sec:prelim} and~\ref{sec:general}.

\end{document}